\documentclass[11pt,a4paper]{amsart}
\usepackage[utf8]{inputenc}

\usepackage[T1]{fontenc}
\usepackage[english]{babel}

\usepackage{sidecap}

\usepackage{verbatim}
\usepackage{lmodern}
\usepackage{amsmath}
\usepackage{amssymb} 
\usepackage{amsthm} 
\usepackage{thmtools}
\usepackage{enumitem}
\usepackage[margin=1.2in]{geometry}
\usepackage{graphicx}
\usepackage{hyperref}
\usepackage{mathtools}
\usepackage{indentfirst}
\usepackage{tabularx}
\usepackage{bbm}
\usepackage[capitalise]{cleveref}
\usepackage{stmaryrd}

\usepackage{tikz} 
\usepackage{tikz-cd} 
\usetikzlibrary{arrows,calc,decorations.markings}
\usepackage{xparse}

\usepackage{blindtext}

\makeatletter
\newcommand{\@bbify}[1]{
  \ifcsname b#1\endcsname
  \message{WARNING: Overwriting b#1 with blackboard letter!}
  \fi
  \expandafter\edef\csname b#1\endcsname
  {\noexpand\ensuremath{\noexpand\mathbb #1}\noexpand\xspace}}
\newcommand{\@calify}[1]{
  \ifcsname c#1\endcsname
  \message{WARNING: Overwriting c#1 with calligraphic letter!}
  \fi 
  \expandafter\edef\csname c#1\endcsname
  {\noexpand\ensuremath{\noexpand\mathcal #1}\noexpand\xspace}}
\newcommand{\@bfify}[1]{
  \ifcsname bf#1\endcsname
  \message{WARNING: Overwriting c#1 with bold letter!}
  \fi
  \expandafter\edef\csname bf#1\endcsname
  {\noexpand\ensuremath{\noexpand\mathbf #1}\noexpand\xspace}}
\newcounter{@letter}\stepcounter{@letter}
\loop\@bbify{\Alph{@letter}}\@calify{\Alph{@letter}}\@bfify{\Alph{@letter}}
\ifnum\the@letter<26\stepcounter{@letter}\repeat
\makeatother

\newcommand{\Cat}{\mathrm{Cat}}
\newcommand{\DblCat}{\mathrm{DblCat}}

\newcommand{\TwoCat}{2\mathrm{Cat}}

\newcommand{\ps}{{\mathrm{ps}}}
\newcommand{\id}{{\mathrm{id}}}

\newcommand{\Adj}{\mathbb{A}\mathrm{dj}}

\newcommand{\Eadj}{E_\mathrm{adj}}

\newcommand{\cof}{c}
\newcommand{\fib}{f}
\newcommand{\whi}{\mathrm{whi}}

\newcommand{\Iwcof}{\cI_w\mathrm{-cof}}
\newcommand{\Jwcof}{\cJ_w\mathrm{-cof}}
\newcommand{\Iwinj}{\cI_w\mathrm{-inj}}
\newcommand{\Jwinj}{\cJ_w\mathrm{-inj}}

\newcommand{\Jwcell}{\cJ_w\mathrm{-cell}}
\newcommand{\Gray}{\mathrm{Gr}}
\newcommand{\Lsim}{L^\simeq}
\newcommand{\bHsim}{\bH^\simeq}
\newcommand{\vtwo}{\mathbb{V}\mathbbm{2}}
\newcommand{\A}{\cA}
\newcommand{\B}{\cB}

\newcommand{\vcong}{\rotatebox{270}{$\cong$}}
\newcommand{\lsimeq}{\rotatebox{270}{$\simeq$}}
\newcommand{\rsimeq}{\rotatebox{90}{$\simeq$}}

\newenvironment{tz}{\begin{center}\begin{tikzpicture}}{\end{tikzpicture}\end{center}}

\tikzstyle{d}=[double distance=.3ex]

\tikzset{%
node distance=1.5cm, la/.style={scale=0.8}, rr/.style={xshift=1.5cm},
space/.style={xshift=.5cm}, over/.style={auto=false,fill=white,inner sep=1.5pt, minimum size=0, outer sep=0},
    symbol/.style={%
        draw=none,
        every to/.append style={%
            edge node={node [sloped, allow upside down, auto=false]{$#1$}}},
            
    }, pro/.style={postaction={decorate,decoration={
        markings,
        mark=at position .5 with {\node at (0,0) {$\bullet$};}
      }},
      inner sep=.9ex,
      },
  n/.style={double equal sign distance, -implies}, t/.style={double distance=2.5pt, -implies, postaction={draw,-}},
}
\newcommand{\arrowdot}{
\ensuremath{\begin{tikzpicture}
\node (A) at (0,-.4) {};
\node (B) at (.4,-.4) {};
\draw[->, line width=.1ex] (0,-.6) -- (.4,-.6);
\node[shape=circle, fill=black, scale=0.35] (A) at  (.17,-.6) {};
\end{tikzpicture}
}}

\NewDocumentCommand{\celli}{ O{} O{n} O{\cellslide} O{\celllength} m m m }{
  \coordinate (mid) at ($({#5})!{#3}!({#6})$);
  \coordinate (start) at ($(mid)!{#4}!({#5})$);
  \coordinate (end) at ($(mid)!{#4}!({#6})$);
  \draw[#2] (start) to node(label)[inner sep=4pt,outer sep=0,minimum size=0,#1]{{#7}} (end);
   \coordinate (far) at ($(end)+(mid)-(label)$);
   \node[] at ($(end)!7pt!(far)$) {$\scriptscriptstyle\cong$} ;
}
\newcommand{\sq}[5]{{#1}\colon({#4} \; ^{{#2}}_{\substack{{#3}}} \; {#5})}
\newcommand{\push}[3]{{#1}\,\square_{{#3}}\, {#2}}
\newcommand{\pushprod}[7]{{#1}\,\square\, {#4}\colon {#2}\otimes_{#7} {#6}\coprod_{{#2}\otimes_{#7} {#5}} {#3}\otimes_{#7} {#5}\to {#3}\otimes_{#7} {#6}}

\def\cellslide{0.5}
\def\celllength{.2cm}

\NewDocumentCommand{\cell}{ O{} O{n} O{\cellslide} O{\celllength} m m m }{
  \coordinate (mid) at ($({#5})!{#3}!({#6})$);
  \coordinate (start) at ($(mid)!{#4}!({#5})$);
  \coordinate (end) at ($(mid)!{#4}!({#6})$);
  \draw[#2] (start) to node
  [inner sep=6pt,outer sep=0,minimum size=0,#1]{{#7}} (end);
}

\newlist{rome}{enumerate}{7}
\setlist[rome]{label=(\roman*)}

\newtheorem{theorem}{Theorem}[section]
\newtheorem{cor}[theorem]{Corollary}
\newtheorem{prop}[theorem]{Proposition}
\newtheorem{lemme}[theorem]{Lemma}
\declaretheorem[name=Theorem,numbered=yes]{theoremA}

\theoremstyle{definition}
\newtheorem{defn}[theorem]{Definition}
\newtheorem{ex}[theorem]{Example}
\newtheorem{notation}[theorem]{Notation}
\newtheorem{descr}[theorem]{Description}
\newtheorem{constr}[theorem]{Construction}

\theoremstyle{remark}
\newtheorem{rem}[theorem]{Remark}

\crefname{theorem}{Theorem}{Theorems}
\crefname{cor}{Corollary}{Corollaries}
\crefname{prop}{Proposition}{Propositions}
\crefname{lemme}{Lemma}{Lemmas}

\crefname{defn}{Definition}{Definitions}
\crefname{ex}{Example}{Examples}
\crefname{notation}{Notation}{Notations}
\crefname{descr}{Description}{Descriptions}
\crefname{constr}{Construction}{Constructions}

\crefname{rem}{Remark}{Remarks}

\title{A model structure for weakly horizontally invariant double categories}

\author[L.\ Moser]{Lyne Moser}
\address{Max Planck Institute for Mathematics, Vivatsgasse 7, 53111 Bonn, Germany}
\email{moser@mpim-bonn.mpg.de}

\author[M.\ Sarazola]{Maru Sarazola}
\address{Department of Mathematics, Cornell University, Ithaca NY, 14853, USA}
\email{mes462@cornell.edu}

\author[P.\ Verdugo]{Paula Verdugo}
\address{Department of Mathematics and Statistics, Macquarie University, NSW 2109, Australia}
\email{paula.verdugo@hdr.mq.edu.au}

\begin{document}

\maketitle

\begin{abstract}
    We construct a model structure on the category $\mathrm{DblCat}$ of double categories and double functors, whose  trivial fibrations are the double functors that are surjective on objects, full on horizontal and vertical morphisms, and fully faithful on squares; and whose fibrant objects are the weakly horizontally invariant double categories.
    
    We show that the functor $\mathbb H^{\simeq}\colon \mathrm{2Cat}\to \mathrm{DblCat}$, a more homotopical version of the usual horizontal embedding $\mathbb H$, is right Quillen and homotopically fully faithful when considering Lack's model structure on $\mathrm{2Cat}$. In particular, $\mathbb H^{\simeq}$ exhibits a levelwise fibrant replacement of $\mathbb H$. Moreover, Lack's model structure on $\mathrm{2Cat}$ is right-induced along $\mathbb H^{\simeq}$ from the model structure for weakly horizontally invariant double categories. 
    
    We also show that this model structure is monoidal with respect to B\"ohm's Gray tensor product. Finally, we prove a Whitehead Theorem characterizing the weak equivalences with fibrant source as the double functors which admit a pseudo inverse up to horizontal pseudo natural equivalence.
\end{abstract}

\section{Introduction}

This paper aims to study and compare the homotopy theories of two related types of $2$-dimensional categories: \emph{$2$-categories} and \emph{double categories}. While $2$-categories consist of objects, morphisms, and $2$-morphisms, double categories admit two types of morphisms between objects -- \emph{horizontal} and \emph{vertical} morphisms -- and their $2$-morphisms are given by \emph{squares}. In particular, a $2$-category $\A$ can always be seen as a horizontal double category~$\bH\A$ with only trivial vertical morphisms. This assignment $\bH$ gives a full embedding of $2$-categories into double categories.

The category $\TwoCat$ of $2$-categories and $2$-functors admits a model structure, constructed by Lack in \cite{Lack2Cat,LackBicat}. In this model structure, the weak equivalences are the biequivalences; the trivial fibrations are the $2$-functors which are surjective on objects, full on morphisms, and fully faithful on $2$-morphisms; and all $2$-categories are fibrant. Moreover, Lack gives a characterization of the cofibrant objects as the $2$-categories whose underlying category is free. With this well-established model structure at hand, we raise the question of whether there is a homotopy theory for double categories which contains that of $2$-categories.

Several model structures for double categories were first constructed by Fiore and Paoli in \cite{FP}, and by Fiore, Paoli, and Pronk in \cite{FPP}, but the homotopy theory of $2$-categories does not embed in any of these homotopy theories for double categories. The first positive answer to this question is given by the authors in \cite{MSV}, and further related results appear in work in progress by Campbell \cite{Camp}. In \cite{MSV}, we construct a model structure on the category $\DblCat$ of double categories and double functors that is right-induced from two copies of Lack's model structure on $\TwoCat$; its weak equivalences are called the \emph{double biequivalences}. 

This model structure is very well-behaved with respect to the horizontal embedding~$\bH$: the functor $\bH\colon \TwoCat\to \DblCat$ is both left and right Quillen, and Lack's model structure is both left- and right-induced along it. In particular, this says that Lack's model structure on $\TwoCat$ is created by $\bH$ from the model structure on $\DblCat$ of~\cite{MSV}. Moreover, the functor~$\bH$ is homotopically fully faithful, and it embeds the homotopy theory of $2$-categories into that of double categories in a reflective and coreflective way.

As it was constructed with a pronounced horizontal bias, this model structure is unsurprisingly not well-behaved with respect to the vertical direction. For example, trivial fibrations, which are full on horizontal morphisms, are only surjective on vertical morphisms, and the free double category on two composable vertical morphisms is not cofibrant, as opposed to its horizontal analogue. In particular, this prevents the model structure from being monoidal with respect to the Gray tensor product for double categories defined by B\"ohm in \cite{Bohm}.

Additionally, the model structure of \cite{MSV} is not compatible with the first-named author's nerve construction from double categories to double $(\infty,1)$-categories in \cite{Lyne}. Since all objects of this model structure on $\DblCat$ are fibrant, while the nerve of a double category is in general not fibrant, we see that the nerve functor fails to be right Quillen. In fact, the double categories whose nerve is fibrant are precisely the \emph{weakly horizontally invariant} ones. This condition requires that every vertical morphism in the double category can be lifted along horizontal equivalences at its source and target; see \cref{def:whidbl}. 

The aim of this paper is to provide a new model structure on $\DblCat$, whose trivial fibrations behave symmetrically with respect to the horizontal and vertical directions, and whose fibrant objects are the weakly horizontally invariant double categories. We achieve this by adding the inclusion $\mathbbm 1\sqcup\mathbbm 1\to \vtwo$ of the two end-points into the vertical morphism to the class of cofibrations of the model structure in \cite{MSV}. In particular, by making this inclusion into a cofibration, the trivial fibrations will now be given by the double functors that are surjective on objects, full on horizontal \emph{and} vertical morphisms, and fully faithful on squares. The existence of this model structure was independently noticed at roughly the same time by Campbell~\cite{Camp}.

As an anonymous referee pointed out, this change in the generating cofibrations requires us to enlarge the class of weak equivalences, since now the class of double functors that are both cofibrations and double biequivalences is not closed under pushouts, and therefore cannot be the class of trivial cofibrations in a model structure. Instead, we find that the weak equivalences of the desired model structure can be described as the double functors which induce a double biequivalence between fibrant replacements.

\begin{theoremA} \label{theoremA}
There is a model structure on $\DblCat$, in which the trivial fibrations are the double functors which are surjective on objects, full on horizontal and vertical morphisms, and fully faithful on squares, and the fibrant objects are the weakly horizontally invariant double categories.
\end{theoremA}

This new model structure on $\DblCat$ takes care of the issues posed above. Namely, it is compatible with the double $(\infty,1)$-categorical nerve construction of \cite{Lyne}, and it is moreover monoidal, as we prove in \cref{thm:DblCat2monoidal}.

\begin{theoremA}
The model structure on $\DblCat$ of \cref{theoremA} is monoidal with respect to B\"ohm's Gray tensor product.
\end{theoremA}

While the horizontal embedding $\bH\colon \TwoCat\to \DblCat$ remains a left Quillen and homotopically fully faithful functor between Lack's model structure and our new model structure, it is not right Quillen anymore. Indeed, the horizontal double category $\bH\A$ associated to a $2$-category $\A$ is typically not weakly horizontally invariant; see \cref{rem:hordblcatarenotwhi}. 

Instead, we consider a more homotopical version of~the horizontal embedding given by the functor ${\bHsim\colon \TwoCat\to \DblCat}$. It sends a $2$-category $\A$ to the double category $\bHsim\A$, whose underlying horizontal $2$-category is still $\A$, but whose vertical morphisms are given by the adjoint equivalences of $\A$. In particular, the inclusion $\bH\A\to \bHsim\A$ is a weak equivalence, as shown in \cref{rem:Hsimiswhi}, and therefore exhibits $\bHsim\A$ as a fibrant replacement of $\bH\A$ in the model structure for weakly horizontally invariant double categories. 

In \cref{thm:LsimHsimrefsecond}, we prove that $\bHsim$ is a right Quillen functor, and that the derived counit is level-wise a biequivalence in $\TwoCat$; therefore, $\bHsim$ embeds the homotopy theory of $2$-categories into that of weakly horizontally invariant double categories in a reflective way. Furthermore, we show in \cref{thm:2CatriHsim} that $\bHsim$ not only preserves, but also reflects weak equivalences and fibrations.

\begin{theoremA}
The adjunction 
\begin{tz}
\node[](A) {$\TwoCat$};
\node[right of=A,xshift=1cm](B) {$\DblCat$};
\draw[->] ($(B.west)+(0,.25cm)$) to [bend right=25] node[above,la]{$\Lsim$} ($(A.east)+(0,.25cm)$);
\draw[->] ($(A.east)-(0,.25cm)$) to [bend right=25] node[below,la]{$\bHsim$} ($(B.west)+(0,-.25cm)$);
\node[la] at ($(A.east)!0.5!(B.west)$) {$\bot$};
\end{tz}
is a Quillen pair between Lack's model structure on $\TwoCat$ and the model structure on $\DblCat$ of \cref{theoremA}. Moreover, the derived counit of this adjunction is level-wise a biequivalence, and Lack's model structure on $\TwoCat$ is right-induced along $\bHsim$ from the model structure on~$\DblCat$. 
\end{theoremA}

We also show in \cref{thm:idQuillenrefl} that the identity functor from our new model structure on $\DblCat$ to the one of \cite{MSV} is right Quillen and homotopically fully faithful. This implies that, unsurprisingly, the homotopy theory of weakly horizontally invariant double categories is embedded into that of double categories. 

To summarize, we have a triangle of right Quillen and homotopically fully faithful functors
\begin{tz}
\node[](1) {$\TwoCat$};
\node[below right of =1,xshift=.5cm,yshift=-.2cm](2) {$\DblCat$};
\node[below left of=1,xshift=-.5cm,yshift=-.2cm](3) {$\DblCat_{\mathsf{whi}}$}; 
\draw[->] (1) to node[right,la,yshift=7pt]{$\bH$} (2);
\draw[->] (1) to node[left,la,yshift=7pt]{$\bHsim$} (3);
\draw[->] (3) to node[below,la]{$\id$} (2);
\coordinate(a) at ($(1)!0.25!(2)$);
\coordinate(b) at ($(3.east)+(.1cm,0)$);
\cell[la,left,xshift=-3pt][n][0.45]{a}{b}{$\simeq$};
\end{tz}
filled by a natural transformation which is level-wise a weak equivalence.

Finally, in a similar vein to Grandis' result \cite[Theorem 4.4.5]{Grandis}, we obtain a Whitehead Theorem characterizing the weak equivalences with fibrant source in our model structure as the double functors which admit a pseudo inverse up to horizontal pseudo natural equivalences; see \cref{thm:whitehead}. Indeed, the weakly horizontally invariant condition on a double category is a $2$-categorical analogue of the horizontally invariant condition of \cite[Theorem and Definition 4.1.7]{Grandis}.

\begin{theoremA}[Whitehead Theorem for double categories] \label{Whiteheadintro}
Let $\bA$ and $\bB$ be double categories such that $\bA$ is weakly horizontally invariant. Then a double functor $F\colon \bA\to \bB$ is a double biequivalence if and only if there is a pseudo double functor $G\colon \bB\to \bA$ together with horizontal pseudo natural equivalences $\id_\bA\simeq GF$ and $FG\simeq \id_\bB$.
\end{theoremA}

This result implies that the weak equivalences between fibrant objects in the model structure for weakly horizontally invariant double categories resemble the biequivalences between $2$-categories. 

\subsection*{Outline}

In \cref{sec:prelim}, we recall some notations and definitions of double category theory introduced in \cite{MSV}. We also introduce weakly horizontally invariant double categories and the homotopical horizontal embedding functor $\bHsim\colon \TwoCat\to \DblCat$. Then, in \cref{sec:MS}, we give the main features of the model structure on $\DblCat$. In particular, we describe the cofibrations, trivial fibrations, and weak equivalences. The proof of the existence of this model structure uses several technical results presented in \cref{subsec:JwcofJwinj} and is completed in \cref{subsec:prooftheorem}. After establishing the model structure, we study in \cref{sec:Quillenpairs} its relation with the model structure on $\DblCat$ of \cite{MSV} and with Lack's model structure on $\TwoCat$. In \cref{sec:monoidality}, we prove that it is monoidal with respect to the Gray tensor product for double categories. The last section, \cref{section:whitehead}, is devoted to the proof of the Whitehead theorem for double categories.

\subsection*{Acknowledgements} 

The authors would like to thank tslil clingman for sharing LaTeX commands which greatly simplify the drawing of diagrams. We would also like to thank Rune Haugseng, Viktoriya Ozornova and Jérôme Scherer for interesting discussions related to the subject of this paper, and Martina Rovelli, Hadrian Heine, and Yuki Maehara for useful answers to our questions. We also acknowledge several discussions with Alexander Campbell regarding ours and his work on the topic. 

In a first version of this paper, the authors had proposed a model structure with the same cofibrations and fibrant objects, but with double biequivalences as the class of weak equivalences. We are thankful for the careful reading of an anonymous referee, who alerted us to this mistake and provided an example showing that the proposed trivial cofibrations were not stable under pushout. Since then, the paper has gone through a thorough rewriting, and we are grateful to Jérôme Scherer and Denis-Charles Cisinski for helpful suggestions regarding the construction of our new class of weak equivalences.

During the realization of this work, the first-named author was supported by the Swiss National Science Foundation under the project P1ELP2\_188039 and the Max Planck Institute of Mathematics.

\section{Double categorical preliminaries} \label{sec:prelim}

We introduce in this section the concepts and notations that will be used throughout this paper. We denote by $\TwoCat$ the category of $2$-categories and $2$-functors, and by $\DblCat$ the category of double categories and double functors. Let us recall the following notation.

\begin{notation}
We write $\bH\colon \TwoCat\to \DblCat$ for the horizontal embedding, which sends a $2$-category $\cA$ to the double category $\bH\cA$ with the same objects as $\cA$, the morphisms of~$\cA$ as its horizontal morphisms, only trivial vertical morphisms, and the $2$-morphisms of~$\cA$ as its squares. This functor has a right adjoint $\bfH\colon \DblCat\to \TwoCat$ that sends a double category $\bA$ to its underlying horizontal $2$-category $\bfH\bA$ obtained by forgetting the vertical morphisms of $\bA$. Note that $\bfH\bH=\id_{\TwoCat}$. 

Similarly, there is a vertical embedding $\bV\colon \TwoCat\to \DblCat$ which also admits a right adjoint $\bfV\colon \DblCat\to\TwoCat$ extracting from a double category its underlying vertical $2$-category. 
\end{notation}

\subsection{Weak horizontal invertibility in a double category}

We recall the notions of weak horizontal invertibility for horizontal morphisms and squares introduced in \cite[\S 2]{MSV}. These notions were independently developed by Grandis and Par\'e in \cite[\S 2]{GraPar19}, where the weakly horizontally invertible squares are called \emph{equivalence cells}.

\begin{defn}\label{def:horeq}
A horizontal morphism $a\colon A\to B$ in a double category $\bA$ is a \textbf{horizontal (adjoint) equivalence} if it is an (adjoint) equivalence in the underlying horizontal $2$-category $\bfH \bA$. We write $a\colon A\xrightarrow{\simeq} B$.
\end{defn} 

For the next definition, we remind the reader that the category $\DblCat$ is cartesian closed, and we denote its internal hom double category by $[-,-]$. In particular, we consider the functor $\bfH[\vtwo,-]\colon \DblCat\to \TwoCat$, where $\vtwo$ is the free double category on a vertical morphism. See  \cite[Definition 2.13]{MSV} for an explicit description.

\begin{defn} \label{def:whisquare}
A square $\sq{\alpha}{a}{a'}{u}{w}$ in a double category $\bA$ is \textbf{weakly horizontally invertible} if it is an equivalence in the $2$-category $\bfH[\vtwo,\bA]$. In other words, if there is a square $\sq{\gamma}{c}{c'}{w}{u}$ in $\bA$ together with four vertically invertible squares $\eta$, $\eta'$, $\epsilon$, and $\epsilon'$ as in the following pasting equalities. 
\begin{tz}
\node[](1) {$A$}; 
\node[below of=1](2) {$A$}; 
\node[right of=2](3) {$C$}; 
\node[right of=3](4) {$A$}; 
\node[above of=4](5) {$A$}; 
\draw[d] (1) to (5);
\draw[d,pro] (1) to (2); 
\draw[d,pro] (5) to (4); 
\draw[->] (2) to node[above,la]{$a$} (3); 
\draw[->] (3) to node[above,la]{$c$} (4); 

\node[la] at ($(1)!0.5!(4)-(5pt,0)$) {$\eta$}; 
\node[la] at ($(1)!0.5!(4)+(5pt,0)$) {$\vcong$};

\node[below of=2](2') {$A'$}; 
\node[below of=3](3') {$C'$}; 
\node[below of=4](4') {$A'$}; 
\draw[->,pro] (2) to node[left,la]{$u$} (2');
\draw[->,pro] (3) to node[left,la]{$w$} (3');
\draw[->,pro] (4) to node[right,la]{$u$} (4');
\draw[->] (2') to node[below,la]{$a'$} (3'); 
\draw[->] (3') to node[below,la]{$c'$} (4'); 

\node[la] at ($(2)!0.5!(3')$) {$\alpha$}; 
\node[la] at ($(3)!0.5!(4')$) {$\gamma$}; 

\node[right of=5,xshift=.3cm](1') {$A$}; 
\node[la] at ($(4')!0.5!(1')$) {$=$}; 
\node[below of=1'](1) {$A'$}; 
\node[below of=1](2) {$A'$}; 
\node[right of=2](3) {$C'$}; 
\node[right of=3](4) {$A'$}; 
\node[above of=4](5) {$A'$}; 
\node[above of=5](5') {$A$};
\draw[d] (1) to (5);
\draw[d] (1') to (5');
\draw[d,pro] (1) to (2); 
\draw[d,pro] (5) to (4); 
\draw[->] (2) to node[below,la]{$a'$} (3); 
\draw[->] (3) to node[below,la]{$c'$} (4); 
\draw[->,pro] (1') to node[left,la]{$u$} (1);
\draw[->,pro] (5') to node[right,la]{$u$} (5);

\node[la] at ($(1')!0.5!(5)$) {$\id_u$}; 

\node[la] at ($(1)!0.5!(4)-(5pt,0)$) {$\eta'$}; 
\node[la] at ($(1)!0.5!(4)+(5pt,0)$) {$\vcong$};
\end{tz}
\begin{tz}
\node[](1) {$C$}; 
\node[above of=1](2) {$C$}; 
\node[right of=2](3) {$A$}; 
\node[right of=3](4) {$C$}; 
\node[below of=4](5) {$C$}; 
\node[below of=1](1') {$C'$}; 
\node[below of=5](5') {$C'$};
\draw[d] (1) to (5);
\draw[d] (1') to (5');
\draw[d,pro] (2) to (1); 
\draw[d,pro] (4) to (5); 
\draw[->,pro] (1) to node[left,la]{$w$} (1');
\draw[->,pro] (5) to node[right,la]{$w$} (5');
\draw[->] (2) to node[above,la]{$c$} (3); 
\draw[->] (3) to node[above,la]{$a$} (4); 

\node[la] at ($(2)!0.5!(4')$) {$\id_w$}; 

\node[la] at ($(1)!0.5!(4)-(5pt,0)$) {$\epsilon$}; 
\node[la] at ($(1)!0.5!(4)+(5pt,0)$) {$\vcong$};

\node[right of=4,xshift=.3cm](2') {$C$}; 
\node[la] at ($(5')!0.5!(2')$) {$=$}; 
\node[right of=2'](3') {$A$}; 
\node[right of=3'](4') {$C$};
\node[below of=2'](2) {$C'$};  
\node[below of=2](1) {$C'$}; 
\node[right of=2](3) {$A'$}; 
\node[right of=3](4) {$C'$}; 
\node[below of=4](5) {$C'$}; 
\draw[d] (1) to (5);
\draw[d,pro] (2) to (1); 
\draw[d,pro] (4) to (5); 
\draw[->,pro] (2') to node[left,la]{$w$} (2);
\draw[->,pro] (3') to node[left,la]{$u$} (3);
\draw[->,pro] (4') to node[right,la]{$w$} (4);
\draw[->] (2) to node[below,la]{$c'$} (3); 
\draw[->] (3) to node[below,la]{$a'$} (4);
\draw[->] (2') to node[above,la]{$c$} (3'); 
\draw[->] (3') to node[above,la]{$a$} (4'); 

\node[la] at ($(2')!0.5!(3)$) {$\gamma$}; 
\node[la] at ($(3')!0.5!(4)$) {$\alpha$}; 

\node[la] at ($(1)!0.5!(4)-(5pt,0)$) {$\epsilon'$}; 
\node[la] at ($(1)!0.5!(4)+(5pt,0)$) {$\vcong$};
\end{tz}
We call $\gamma$ a \textbf{weak inverse} of $\alpha$.
\end{defn}

\begin{rem}
In particular, the horizontal boundaries $a$ and $a'$ of a weakly horizontally invertible square $\alpha$ as above are horizontal equivalences witnessed by the data $(a,c,\eta,\epsilon)$ and $(a',c',\eta',\epsilon')$. We call them the \emph{horizontal equivalence data} of $\alpha$. Moreover, if $(a,c,\eta,\epsilon)$ and $(a',c',\eta',\epsilon')$ are both horizontal adjoint equivalences, we call them the \emph{horizontal adjoint equivalence data} of $\alpha$.
\end{rem}

\begin{rem}
Note that a horizontal equivalence can always be promoted to a horizontal \emph{adjoint} equivalence, since the corresponding result holds for 2-categories (see, for example, \cite[Lemma 2.1.11]{RiehlVerity}). Similarly, a weakly horizontally invertible square can always be promoted to one with horizontal \emph{adjoint} equivalence data.
\end{rem}

The next result ensures that the weak inverse of a weakly horizontally invertible square is unique with respect to fixed horizontal adjoint equivalences. 

\begin{lemme}[{\cite[Lemma A.1.1]{Lyne}}]\label{uniqueweakinverse}
 Given a weakly horizontally invertible square $\sq{\alpha}{a}{a'}{u}{w}$ and two horizontal adjoint equivalences $(a,c,\eta,\epsilon)$ and $(a',c',\eta',\epsilon')$ in a double category~$\bA$, there is a unique weak inverse $\sq{\gamma}{c}{c'}{w}{u}$ of $\alpha$ with respect to these horizontal adjoint equivalences. 
\end{lemme}

\subsection{Double biequivalences and weakly horizontally invariant double categories}

The weak equivalences of the desired model structure for double categories rely on \emph{double biequivalences}, which are the weak equivalences of the model structure on $\DblCat$ constructed in \cite{MSV}.

\begin{defn} \label{def:doublebieq}
Let $\bA$ and $\bB$ be double categories. A double functor $F\colon \bA\to \bB$ is a \textbf{double biequivalence} if 
\begin{enumerate}
    \item[(db1)] for every object $B\in \bB$, there is an object $A\in \bA$ together with a horizontal equivalence $B\xrightarrow{\simeq} FA$ in $\bB$,
    \item[(db2)] for every pair of objects $A,C\in \bA$ and every horizontal morphism $b\colon FA\to FC$ in~$\bB$, there is a horizontal morphism $a\colon A\to C$ together with a vertically invertible square in $\bB$ of the form
    \begin{tz}
    \node[](1) {$FA$}; 
    \node[right of=1](2) {$FC$}; 
    \node[below of=1](3) {$FA$}; 
    \node[below of=2](4) {$FC$};
    \draw[d,pro] (1) to (3);
    \draw[d,pro] (2) to (4);
    \draw[->] (1) to node[above,la]{$b$} (2);
    \draw[->] (3) to node[below,la]{$Fa$} (4);
    \node[la] at ($(1)!0.5!(4)$) {$\vcong$};
    \end{tz}
    \item[(db3)] for every vertical morphism $v\colon B\arrowdot B'$ in $\bB$, there is a vertical morphism $u\colon A\arrowdot A'$ in $\bA$ together with a weakly horizontally invertible square in $\bB$ of the form
    \begin{tz}
    \node[](1) {$B$}; 
    \node[right of=1](2) {$FA$}; 
    \node[below of=1](3) {$B'$}; 
    \node[below of=2](4) {$FA'$};
    \draw[->,pro] (1) to node[left,la]{$v$} (3);
    \draw[->,pro] (2) to node[right,la]{$Fu$} (4);
    \draw[->] (1) to node[above,la]{$\simeq$} (2);
    \draw[->] (3) to node[below,la]{$\simeq$} (4);
    \node[la] at ($(1)!0.5!(4)$) {$\simeq$};
    \end{tz}
    \item[(db4)] for every pair of horizontal morphisms $a\colon A\to C$ and $a'\colon A'\to C'$ in $\bA$, every pair of vertical morphisms $u\colon A\arrowdot A'$ and $w\colon C\arrowdot C'$ in $\bA$, and every square $\beta$ in $\bB$ as depicted below left, there is a unique square $\alpha$ in $\bA$ as depicted below right such that $\beta=F\alpha$. 
    \begin{tz}
    \node[](1) {$FA$}; 
    \node[right of=1](2) {$FC$}; 
    \node[below of=1](3) {$FA'$}; 
    \node[below of=2](4) {$FC'$};
    \draw[->,pro] (1) to node[left,la]{$Fu$} (3);
    \draw[->,pro] (2) to node[right,la]{$Fw$} (4);
    \draw[->] (1) to node[above,la]{$Fa$} (2);
    \draw[->] (3) to node[below,la]{$Fa'$} (4);
    \node[la] at ($(1)!0.5!(4)$) {$\beta$};
    
    \node[right of=2,xshift=1cm](1) {$A$}; 
    \node[right of=1](2) {$C$}; 
    \node[below of=1](3) {$A'$}; 
    \node[below of=2](4) {$C'$};
    \draw[->,pro] (1) to node[left,la]{$u$} (3);
    \draw[->,pro] (2) to node[right,la]{$w$} (4);
    \draw[->] (1) to node[above,la]{$a$} (2);
    \draw[->] (3) to node[below,la]{$a'$} (4);
    \node[la] at ($(1)!0.5!(4)$) {$\alpha$};
    \end{tz}
\end{enumerate}
\end{defn}

We also introduce the notion of \emph{weakly horizontally invariant} double categories, which will form the class of fibrant objects in our model structure. This is a $2$-categorical analogue of the notion of \emph{horizontally invariant} double categories, introduced by Grandis and Par\'e in \cite[\S 2.4]{GrandisPare} as the double categories whose vertical morphisms are transferable along horizontal isomorphisms. 

\begin{defn} \label{def:whidbl}
A double category $\bA$ is \textbf{weakly horizontally invariant} if, for every diagram in $\bA$ as depicted below left, where $a$ and $a'$ are horizontal equivalences, there is a vertical morphism $u\colon A\arrowdot A'$ together with a weakly horizontally invertible square in~$\bA$ as depicted below right.
\begin{tz}
\node[](A) {$A$};
    \node[right of=A](B) {$C$};
    \node[below of=A](A') {$A'$};
    \node[right of=A'](B') {$C'$};
    \draw[->] (A) to node[above,la] {$a$} node[below,la]{$\simeq$} (B);
    \draw[->] (A') to node[below,la] {$a'$} node[above,la] {$\simeq$} (B');
    \draw[->,pro] (B) to node[right,la]{$w$} (B');

\node[right of=B,xshift=1cm](A) {$A$};
    \node[right of=A](B) {$C$};
    \node[below of=A](A') {$A'$};
    \node[right of=A'](B') {$C'$};
    \draw[->] (A) to node[above,la] {$a$} node[below,la]{$\simeq$} (B);
    \draw[->] (A') to node[below,la] {$a'$} node[above,la] {$\simeq$} (B');
    \draw[->,pro] (A) to node[left,la]{$u$} (A');
    \draw[->,pro] (B) to node[right,la]{$w$} (B');

    \node[la] at ($(A)!0.5!(B')$) {$\simeq$};
\end{tz}
\end{defn}

The class of weakly horizontally invariant double categories contains many examples of interest.

\begin{ex}
One can easily check that the (flat) double category $\bR\mathrm{elSet}$ of relations of sets is weakly horizontally invariant. More relevantly, this class of double categories also contains the double categories of quintets $\bQ\cA$ and of adjunctions $\Adj\cA$ built from any $2$-category $\cA$. A precise description of these double categories can be found in \cite[\S 3.1]{Grandis}; in fact, the reader may check that all examples presented in that section are weakly horizontally invariant. 
\end{ex}

\subsection{The homotopical horizontal embedding}

As we will see in \cref{rem:hordblcatarenotwhi}, the horizontal double category $\bH\cA$ associated to a $2$-category $\cA$ is not always weakly horizontally invariant. Hence, since all $2$-categories are fibrant, the functor $\bH\colon \TwoCat\to \DblCat$ is not right Quillen with respect to the desired model structure for weakly horizontally invariant double categories. Instead, we need to consider a more homotopical version of the horizontal embedding $\bH$, which provides a levelwise fibrant replacement for $\bH$. 

\begin{defn} \label{def:htilde}
The \textbf{homotopical horizontal embedding} is defined as the functor $\bHsim\colon \TwoCat\to \DblCat$ that sends a $2$-category $\cA$ to the double category $\bHsim\cA$ having the same objects as $\cA$, the morphisms of $\cA$ as horizontal morphisms, one vertical morphism for each adjoint equivalence $(u,u^\sharp,\eta,\epsilon)$ in $\cA$, and squares
\begin{tz}
    \node[](A) {$A$};
    \node[right of=A](B) {$C$};
    \node[below of=A](A') {$A'$};
    \node[right of=A'](B') {$C'$};
    \draw[->] (A) to node[above,la] {$a$} (B);
    \draw[->] (A') to node[below,la] {$a'$} (B');
    \draw[->,pro] (A) to node[left,la]{$(u,u^\sharp,\eta,\epsilon)$} (A');
    \draw[->,pro] (B) to node[right,la]{$(w,{w}^\sharp,\eta',\epsilon')$} (B');
    
    \node[la] at ($(A)!0.5!(B')$) {$\alpha$};
\end{tz}
given by the $2$-morphisms $\alpha\colon wa\Rightarrow a'u$ in $\cA$. Although a vertical morphism always contains the whole data of an adjoint equivalence, we often denote it by its left adjoint $u$.
\end{defn}

\begin{rem}
Note that every vertical morphism in the double category $\bHsim\cA$ is a \emph{vertical equivalence}, i.e., an equivalence in the underlying vertical $2$-category. 
\end{rem}

The functor $\bHsim$ admits a left adjoint given by the following.

\begin{defn} \label{def:Lsim}
We define the functor $\Lsim\colon \DblCat\to \TwoCat$ which sends a double category $\bA$ to the $2$-category $\Lsim\bA$ whose 
\begin{rome}
\item objects are the objects of $\bA$, 
\item morphisms are generated by a morphism $a\colon A\to C$, for each horizontal morphisms $a\colon A\to C$ in $\bA$, and two morphisms $\overline{u}\colon A\to A'$ and $\overline{u}^\sharp\colon A'\to A$, for each vertical morphism $u\colon A\arrowdot A'$ in $\bA$,
\item $2$-morphisms are generated by $2$-isomorphisms $\eta_{\overline{u}}\colon \id_{A}\cong \overline{u}^\sharp\overline{u}$ and $\epsilon_{\overline{u}}\colon \overline{u}\,\overline{u}^\sharp\cong \id_{A'}$ satisfying the triangle identities, for each vertical morphism $u\colon A\arrowdot A'$ in $\bA$, and a $2$-morphism $\alpha\colon \overline{w}a\Rightarrow a'\overline{u}$, for each square $\sq{\alpha}{a}{a'}{u}{w}$ in $\bA$, 
\end{rome}
submitted to minimal relations making the inclusion $\bA\to \bHsim\Lsim\bA$ into a double functor. The functor $\Lsim$ sends a double functor $F\colon \bA\to\bB$ to the $2$-functor $\Lsim F\colon \Lsim\bA\to \Lsim\bB$ which acts as $F$ does on the corresponding data. In particular, it sends an adjoint equivalence $(\overline{u},\overline{u}^\sharp,\eta_{\overline{u}},\epsilon_{\overline{u}})$ in $\Lsim\bA$ associated to a vertical morphism $u$ in $\bA$ to the adjoint equivalence $(\overline{Fu},\overline{Fu}^\sharp,\eta_{\overline{Fu}},\epsilon_{\overline{Fu}})$ in $\Lsim\bB$ associated to the vertical morphism $Fu$ in $\bB$. 
\end{defn}

\begin{rem} 
The relations on the morphisms in $\Lsim\bA$ expressed by the fact that the inclusion $\bA\to \bHsim\Lsim\bA$ is a double functor can be interpreted as follows. The composite of two morphisms coming from horizontal morphisms in $\bA$ is given by their composite in~$\bA$, and the composite of two adjoint equivalences coming from vertical morphisms in~$\bA$ is given by the adjoint equivalence induced by their composite in $\bA$, while two morphisms with one coming from a horizontal morphism and one coming from a vertical morphism compose freely.
\end{rem}

\begin{prop} \label{prop:bHsimadj}
The functors $\Lsim$ and $\bHsim$ form an adjunction
\begin{tz}
\node[](A) {$\TwoCat$};
\node[right of=A,xshift=1cm](B) {$\DblCat$};
\node at ($(B.east)-(0,4pt)$) {.};
\draw[->] ($(B.west)+(0,.25cm)$) to [bend right=25] node[above,la]{$\Lsim$} ($(A.east)+(0,.25cm)$);
\draw[->] ($(A.east)-(0,.25cm)$) to [bend right=25] node[below,la]{$\bHsim$} ($(B.west)+(0,-.25cm)$);
\node[la] at ($(A.east)!0.5!(B.west)$) {$\bot$};
\end{tz}
\end{prop}

\begin{proof}
For every $2$-category $\cC$ and every double category $\bA$, there is an isomorphism \[ \DblCat(\bA,\bHsim\cC)\cong\TwoCat(\Lsim\bA,\cC)  \]
natural in $\cC$ and $\bA$. Indeed, every double functor $F\colon \bA\to \bHsim\cC$ induces a $2$-functor ${F^\flat\colon \Lsim\bA\to \cC}$ which acts as $F$ does on the corresponding data. In particular, it sends an adjoint equivalence $(\overline{u},\overline{u}^\sharp,\eta_{\overline{u}},\epsilon_{\overline{u}})$ in $\Lsim\bA$ associated to a vertical morphism $u$ in $\bA$ to the adjoint equivalence in $\cC$ corresponding to the vertical morphism $Fu$ in $\bHsim\cC$. Conversely, if $G\colon \Lsim\bA\to \cC$ is a $2$-functor, it induces a double functor $G^\sharp\colon \bA\to \bHsim\cC$ which acts as~$G$ does on the corresponding data. In particular, it sends a vertical morphism $u$ in $\bA$ to the vertical morphism in $\bHsim\cC$ corresponding to the adjoint equivalence $(G\overline{u}, G\overline{u}^\sharp,G\eta_{\overline{u}},G\epsilon_{\overline{u}})$ in $\cC$.
\end{proof}

\begin{rem}
The functor $\bHsim$ is not a left adjoint since it does not preserve colimits. To see this, consider the following span of $2$-categories $\cB\leftarrow\cA\to\cC$. We set $\cA$ to be the $2$-category with two objects $0$ and $1$, and freely generated by two morphisms $f\colon 0\to 1$ and $g\colon 1\to 0$ and two $2$-morphisms $\eta\colon \id_0\Rightarrow gf$ and $\epsilon\colon fg\Rightarrow\id_1$. Then let $\cB$ be the category obtained from $\cA$ by inverting the $2$-morphism $\eta$, and $\cC$ be the category obtained from $\cA$ by inverting the $2$-morphism $\epsilon$. The pushout $\cB\sqcup_\cA\cC$ contains an equivalence $(f,g,\eta,\epsilon)$ and hence the double category $\bHsim(\cB\sqcup_\cA\cC)$ contains a vertical morphism induced by the corresponding adjoint equivalence. However, the double categories $\bHsim\cA$, $\bHsim\cB$, and $\bHsim\cC$ do not have non-trivial vertical morphisms, since there are no equivalences in~$\cA$,~$\cB$, or~$\cC$. This shows that $\bHsim$ does not preserve pushouts. 
\end{rem}

\section{The model structure} \label{sec:MS}

Just as there are nerve constructions embedding categories into $(\infty,1)$-categories, and 2-categories into $(\infty,2)$-categories, in \cite{Lyne} the first-named author constructs a double categorical nerve, embedding double categories into double $(\infty, 1)$-categories. As the latter admit a natural model structure when considered as double Segal spaces, we expect to have a model structure on $\DblCat$ making this nerve into a right Quillen functor. 

Since the double categories whose nerve is fibrant are precisely the weakly horizontally invariant ones, this suggests that such a model structure should have the weakly horizontally invariant double categories as its class of fibrant objects. Moreover, since the cofibrations of the model structure for double $(\infty,1)$-categories are the monomorphisms, the inclusion $\mathbbm 1\sqcup \mathbbm 1\to \vtwo$ should be added to the class of cofibrations of the model structure on $\DblCat$ of \cite{MSV}; this allows us to characterize the trivial fibrations as the double functors which are surjective on objects, full on horizontal and vertical morphisms, and fully faithful on squares. 
A first attempt to keep the double biequivalences (which were shown by the authors to be the class of weak equivalences in a model structure on $\DblCat$ in \cite{MSV}) as the weak equivalences of this new model structure proves unsuccessful. Indeed, the resulting class of trivial cofibrations would not be closed under pushouts. Instead, we identify the weak equivalences as the double functors which induce a double biequivalence between weakly horizontally invariant replacements.
 
Since many technical results, presented in \cref{subsec:JwcofJwinj}, are needed to prove the existence of such a model structure, its proof is delayed to \cref{subsec:prooftheorem}.

\subsection{Trivial fibrations, cofibrations, and cofibrant objects}

Using the description of the trivial fibrations mentioned above, we first identify a set of generating cofibrations that yields this class of trivial fibrations. For this, let us recall the following terminology. 

\begin{notation}
Let $\mathcal{I}$ be a class of morphisms in a cocomplete category $\cC$. Then a morphism in $\cC$ is
\begin{rome}
   \item \textbf{$\cI$-injective} if it has the right lifting property with respect to every morphism in~$\cI$. The class of all such morphisms is denoted $\cI\mathrm{-inj}$.
    \item an \textbf{$\cI$-cofibration} if it has the left lifting property with respect to every $\cI$-injective morphism. The class of all such morphisms is denoted $\cI\mathrm{-cof}$.
    \item a \textbf{relative $\cI$-cell complex} if it is a transfinite composition of pushouts of morphisms in $\cI$. The class of all such morphisms is denoted $\cI\mathrm{-cell}$.
\end{rome}
\end{notation}

\begin{rem}
Recall that every $\cI$-cofibration can be obtained as a retract of a relative $\cI$-cell complex with the same source, and that in a locally presentable category $\cC$, the pair $(\cI\mathrm{-cof},\cI\mathrm{-inj})$ forms a weak factorization system, for any set of morphisms $\cI$ in $\cC$.
\end{rem} 

We now identify a set $\cI_w$ of double functors such that the $\cI_w$-injective morphisms are precisely the trivial fibrations we seek. 

\begin{notation} \label{not:gencofDblCat2}
We denote by $\mathbbm 1$ the terminal double category, by $\mathbbm 2$ the free ($2$-)category on a morphism, by $\bS=\bH \mathbbm{2}\times \bV \mathbbm{2}$ the free double category on a square, by $\delta \bS$ its boundary, and by $\bS_2$ the free double category on two squares with the same boundary. 

Let $\cI_w$ denote the set containing the following double functors: 
\begin{rome}
    \item the unique map $I_1\colon \emptyset\to \mathbbm{1}$,
    \item the inclusion $I_2\colon \mathbbm{1}\sqcup \mathbbm{1}\to \bH \mathbbm{2}$,
    \item the inclusion $I_3\colon \mathbbm{1}\sqcup \mathbbm{1}\to \bV \mathbbm{2}$,
    \item the inclusion $I_4\colon\delta \bS\to \bS$,
    \item the double functor $I_5\colon \bS_2\to \bS$ sending the two non trivial squares in $\bS_2$ to the non trivial square of $\bS$.
\end{rome}
\end{notation}

\begin{prop} \label{prop:trivfibinsecond}
A double functor $F\colon \bA\to \bB$ is in $\Iwinj$ if and only if it is surjective on objects, full on horizontal and vertical morphisms, and fully faithful on squares. 
\end{prop}

\begin{proof}
This is obtained directly from a close inspection of the right lifting properties with respect to the double functors in $\cI_w$.
\end{proof}

\begin{rem} \label{rem:trivfibaredblbieq}
It is straightforward to check that any double functor in $\Iwinj$ is a double biequivalence; see \cref{def:doublebieq}.
\end{rem}

The class of $\cI_w$-cofibrations admits a nice characterization in terms of their underlying horizontal and vertical functors. We denote by $U\colon \TwoCat\to \Cat$ the functor sending a $2$-category to its underlying category, where $\Cat$ is the category of categories and functors.

\begin{theorem}\label{thm:charcofDblsecond}
A double functor $F\colon \bA\to \bB$ is in $\Iwcof$ if and only if its underlying horizontal and vertical functors $U\bfH F$ and $U\bfV F$ have the left lifting property with respect to surjective on objects and full functors. 
\end{theorem}

\begin{proof}
The proof works as in \cite[Proposition 4.7]{MSV}, with the evident modifications for the vertical direction.
\end{proof}

\begin{rem} \label{rem:injectivityofIw}
An equivalent characterization of functors which have the left lifting property with respect to surjective on objects and full functors can be found in \cite[Corollary 4.12]{Lack2Cat}. In particular, we can see that a double functor in $\Iwcof$ is injective on objects, and faithful on horizontal and vertical morphisms.
\end{rem}

Using the characterization mentioned in \cref{rem:injectivityofIw}, we can see that the cofibrant objects in the desired model structure are precisely the double categories whose underlying horizontal and vertical categories are free.

\begin{cor} \label{cor:charcofibrantsecond}
A double category $\bA$ is such that the unique map $\emptyset\to\bA$ is in $\Iwcof$ if and only if its underlying horizontal and vertical categories $U\bfH\bA$ and $U\bfV\bA$ are free.
\end{cor}

\begin{proof}
The proof works as in \cite[Proposition 4.11]{MSV}, with the evident modifications for the vertical direction.
\end{proof}

\subsection{Weakly horizontally invariant replacements and weak equivalences} 

Our next goal is to introduce the class of weak equivalences; these will be the double functors that induce a double biequivalence between weakly horizontally invariant replacements. To construct a weakly horizontally invariant double category from a double category $\bA$, we attach $\bHsim\Eadj$-data freely to every horizontal adjoint equivalence in $\bA$, where the $2$-category $\Eadj$ is the free-living adjoint equivalence $\{0\xrightarrow{\simeq} 1\}$. Since this will be a key notion throughout the paper, let us first describe the double category $\bHsim\Eadj$.

\begin{descr} \label{descr:J4} 
The double category $\bHsim\Eadj$ is generated by the data of a horizontal adjoint equivalence $(f,g,\eta,\epsilon)$, and by vertical morphisms $u$, $v$ and weakly horizontally invertible squares $\alpha$, $\gamma$ as depicted below,
\begin{tz}
    \node[](1) {$0$}; 
    \node[right of=1](2) {$1$}; 
    \node[below of=1](3) {$1$}; 
    \node[below of=2](4) {$1$};
    \draw[->,pro] (1) to node[left,la]{$u$} (3);
    \draw[d,pro] (2) to (4);
    \draw[->] (1) to node[above,la]{$f$} node[below,la]{$\simeq$} (2);
    \draw[d] (3) to (4);
    \node[la] at ($(1)!0.5!(4)-(5pt,0)$) {$\alpha$};
    \node[la] at ($(1)!0.5!(4)+(5pt,0)$) {$\simeq$};
    
    \node[right of=2,xshift=1cm](1) {$1$}; 
    \node[right of=1](2) {$0$}; 
    \node[below of=1](3) {$0$}; 
    \node[below of=2](4) {$0$};
    \draw[->,pro] (1) to node[left,la]{$v$} (3);
    \draw[d,pro] (2) to (4);
    \draw[->] (1) to node[above,la]{$g$} node[below,la]{$\simeq$} (2);
    \draw[d] (3) to (4);
    \node[la] at ($(1)!0.5!(4)-(5pt,0)$) {$\gamma$};
    \node[la] at ($(1)!0.5!(4)+(5pt,0)$) {$\simeq$};
\end{tz}
where $u$ and $v$ are induced by the adjoint equivalences $(f,g,\eta,\epsilon)$ and $(g,f,\epsilon^{-1},\eta^{-1})$, respectively, and the squares $\alpha$ and $\gamma$ are induced by the identity $2$-morphisms at $f$ and $g$, respectively.

In particular, one can show that the pairs $(f,u)$ and $(g,v)$ are \emph{orthogonal companion pairs} as defined for example in \cite[\S 4.1.1]{Grandis}, and the pairs $(g,u)$ and $(f,v)$ are \emph{orthogonal adjoint pairs} as defined in \cite[\S 4.1.2]{Grandis}. Furthermore, the vertical morphisms $(u,v)$ form a vertical adjoint equivalence, i.e., an adjoint equivalence in the underlying vertical $2$-category $\bfV \bHsim\Eadj$. 
\end{descr}

\begin{notation} \label{not:J4}
There is an inclusion $J_4\colon \bH\Eadj\to \bHsim\Eadj$ which sends the horizontal adjoint equivalence in $\bH\Eadj$ to the horizontal adjoint equivalence $(f,g,\eta,\epsilon)$ in $\bHsim\Eadj$.
\end{notation}

\begin{rem} \label{rem:uniqueHsimEadjdata}
By uniqueness of weak inverses with respect to fixed horizontal adjoint equivalence data of \cref{uniqueweakinverse}, we can see that a double functor $G\colon \bHsim\Eadj\to \bA$ is completely determined by its value on the horizontal adjoint equivalence $(f,g,\eta,\epsilon)$ and the squares $\alpha$, $\gamma$ in $\bHsim\Eadj$.
\end{rem}

We are now ready to construct a functorial weakly horizontally invariant replacement $(-)^\whi\colon\DblCat\to\DblCat^{\mathbbm{2}}$.

\begin{constr} \label{def:whireplacement}
Let $\bA$ be a double category and let $\mathrm{HorEq}(\bA)$ denote the set of all horizontal adjoint equivalence data in $\bA$. Each horizontal adjoint equivalence $(a,c,\eta,\epsilon)$ in $\bA$ defines a double functor $\bH\Eadj\to \bA$, and we define $\bA^\whi$ as the pushout below left.
\begin{tz}
\node[](1) {$\bigsqcup_{\mathrm{HorEq}(\bA)} \bH\Eadj$}; 
\node[right of=1,xshift=2cm](2) {$\bA$};
\node[below of=1](3) {$\bigsqcup_{\mathrm{HorEq}(\bA)} \bHsim\Eadj$};
\node[below of=2](4) {$\bA^\whi$}; 
\draw[->] (1) to (2);
\draw[->] (1) to node[left,la]{$\bigsqcup_{\mathrm{HorEq}(\bA)} J_4$} (3); 
\draw[->] (2) to node[right,la]{$j_\bA$} (4); 
\draw[->] (3) to (4);

\node at ($(4)-(8pt,-8pt)$) {$\ulcorner$};

\node[right of=2,xshift=2cm](1) {$\bA$}; 
\node[right of=1,xshift=.5cm](2) {$\bB$};
\node[below of=1](3) {$\bA^\whi$};
\node[below of=2](4) {$\bB^\whi$}; 
\draw[->] (1) to node[above,la]{$F$} (2);
\draw[->] (1) to node[left,la]{$j_\bA$} (3); 
\draw[->] (2) to node[right,la]{$j_\bB$} (4); 
\draw[->] (3) to node[below,la]{$F^\whi$} (4);
\end{tz}
This extends naturally to a functor $(-)^\whi\colon \DblCat\to \DblCat^{\mathbbm{2}}$. In particular, it sends a double category $\bA$ to the double functor $j_\bA\colon \bA\to \bA^\whi$ and a double functor $F\colon \bA\to \bB$ to a commutative square in $\DblCat$ as depicted above right.
\end{constr}

\begin{rem} \label{rem:idonunderhorcat}
The double functor $j_\bA\colon \bA\to \bA^\whi$ is the identity on underlying horizontal categories and it is fully faithful on squares for every double category $\bA$, since it is a pushout of coproducts of the double functor $J_4\colon \bH\Eadj\to \bHsim\Eadj$. Hence a double functor $F\colon \bA\to \bB$ coincides with $F^\whi\colon \bA^\whi\to \bB^\whi$ on underlying horizontal categories.
\end{rem}

\begin{rem}
The construction $j_\bA\colon \bA\to \bA^\whi$ adds $\bHsim\Eadj$-data in $\bA^\whi$ to each horizontal adjoint equivalence $(a,c,\eta,\epsilon)$ in $\bA$, as detailed in \cref{descr:J4}. In particular,  we can see that two vertical morphisms $u$ and $v$ were freely added in $\bA^\whi$ for each equivalence $(a,c,\eta,\epsilon)$, as well as weakly horizontally invertible squares as in \cref{descr:J4}. We henceforth say that the morphisms $u$ and $v$ were \emph{added using the horizontal adjoint equivalence data $(a,c,\eta,\epsilon)$ in $\bA$}.
\end{rem}

As claimed, the double category $\bA^\whi$ is indeed weakly horizontally invariant.

\begin{prop}
For every double category $\bA$, the double category $\bA^\whi$ is weakly horizontally invariant.
\end{prop}

\begin{proof}
Let $a\colon A\xrightarrow{\simeq} C$ and $a'\colon A'\xrightarrow{\simeq} C'$ be horizontal equivalences in $\bA$ and $w\colon C\arrowdot C'$ be a vertical morphism in $\bA^\whi$. By construction of $\bA^\whi$, we have vertical morphisms $u\colon A\arrowdot C$ and $v\colon C'\arrowdot A'$ in $\bA^\whi$ together with weakly horizontally invertible squares~$\alpha$ and $\delta$ in $\bA^\whi$ as depicted below.
\begin{tz}
    \node[](1) {$A$}; 
    \node[right of=1](2) {$C$}; 
    \node[below of=1](3) {$C$}; 
    \node[below of=2](4) {$C$};
    \draw[->,pro] (1) to node[left,la]{$u$} (3);
    \draw[d,pro] (2) to (4);
    \draw[->] (1) to node[above,la]{$a$} node[below,la]{$\simeq$} (2);
    \draw[d] (3) to (4);
    \node[la] at ($(1)!0.5!(4)-(5pt,0)$) {$\alpha$};
    \node[la] at ($(1)!0.5!(4)+(5pt,0)$) {$\simeq$};
    
    \node[right of=2,xshift=1cm](1) {$C'$}; 
    \node[right of=1](2) {$C'$}; 
    \node[below of=1](3) {$A'$}; 
    \node[below of=2](4) {$C'$};
    \draw[->,pro] (1) to node[left,la]{$v$} (3);
    \draw[d,pro] (2) to (4);
    \draw[d] (1) to (2);
    \draw[->] (3) to node[below,la]{$a'$} node[above,la]{$\simeq$} (4);
    \node[la] at ($(1)!0.5!(4)-(5pt,0)$) {$\delta$};
    \node[la] at ($(1)!0.5!(4)+(5pt,0)$) {$\simeq$};
    \end{tz}
Then the composite of vertical morphisms $vwu\colon A\arrowdot A'$ together with the weakly horizontally invertible square given by the vertical composite of the squares $\alpha$, $\id_w$, and $\delta$ gives the desired lift.
\end{proof}

The foresight that $(-)^\whi$ will give a fibrant replacement in our desired model structure (as we show in \cref{prop:j_AinCcapW}) and that the double biequivalences will precisely be the weak equivalences between fibrant objects (proved in \cref{prop:Wwithwhisourceisdblbieq}) motivates us to define our weak equivalences as the double functors inducing double biequivalences between weakly horizontally invariant replacements. 

\begin{defn} \label{def:W}
We define $\cW$ to be the class of double functors $F\colon \bA\to \bB$ such that the induced double functor $F^\whi\colon \bA^\whi\to \bB^\whi$ is a double biequivalence. 
\end{defn}

\begin{rem} \label{prop:2outof3andretractW}
Since double biequivalences are the weak equivalences in the model structure on $\DblCat$ of \cite[Theorem 3.19]{MSV}, they satisfy 2-out-of-3 and are closed under retracts. As a consequence, the class $\cW$ also has these properties, as the replacement $(-)^\whi$ is functorial.
\end{rem}

Although double biequivalences are more tractable than our proposed weak equivalences, the passage to this bigger class is truly needed. Indeed, the class of double functors that are both in $\Iwcof$ and double biequivalences is not closed under pushouts, and thus cannot be the class of trivial cofibrations in a model structure. However, as we now show, double biequivalences are contained in $\cW$. The reverse inclusion does not hold, but, as we will see in \cref{prop:Wwithwhisourceisdblbieq}, a weak equivalence whose source is a weakly horizontally invariant double category is a double biequivalence.

We use the following technical lemma to prove that double biequivalences are contained in our class of weak equivalences.

\begin{lemme} \label{lem:freelyadded}
Let $F\colon \bA\to \bB$ be a double biequivalence. Then, for every vertical morphism $v\colon B\arrowdot B'$ in $\bB^\whi$ which is a composite of freely added vertical morphisms along $j_\bB\colon \bB\to \bB^\whi$, and every pair of horizontal equivalences $b\colon FA\xrightarrow{\simeq} B$ and $b'\colon FA'\xrightarrow{\simeq} B'$, there is a vertical morphism $u\colon A\arrowdot A'$ in $\bA^\whi$ together with a weakly horizontally invertible square in $\bB^\whi$ of the form
\begin{tz}
    \node[](1) {$FA$}; 
    \node[right of=1](2) {$B$}; 
    \node[below of=1](3) {$FA'$}; 
    \node[below of=2](4) {$B'$};
    \draw[->,pro] (1) to node[left,la]{$F^\whi u$} (3);
    \draw[->,pro] (2) to node[right,la]{$v$} (4);
    \draw[->] (1) to node[above,la]{$b$} node[below,la]{$\simeq$} (2);
    \draw[->] (3) to node[below,la]{$b'$} node[above,la]{$\simeq$} (4);
    \node[la] at ($(1)!0.5!(4)-(5pt,0)$) {$\beta$};
    \node[la] at ($(1)!0.5!(4)+(5pt,0)$) {$\simeq$};
\end{tz}
\end{lemme}

\begin{proof}
First note that there is a horizontal adjoint equivalence $(f,g,\eta,\epsilon)$ in $\bB$ and a weakly horizontally invertible square $\alpha$ in $\bB^\whi$ of the form
\begin{tz}
    \node[](1) {$B$}; 
    \node[right of=1](2) {$B'$}; 
    \node[below of=1](3) {$B'$}; 
    \node[below of=2](4) {$B'$};
    \draw[->,pro] (1) to node[left,la]{$v$} (3);
    \draw[d,pro] (2) to (4);
    \draw[->] (1) to node[above,la]{$f$} node[below,la]{$\simeq$} (2);
    \draw[d] (3) to (4);
    \node[la] at ($(1)!0.5!(4)-(5pt,0)$) {$\alpha$};
    \node[la] at ($(1)!0.5!(4)+(5pt,0)$) {$\simeq$};
\end{tz}
obtained by composing the corresponding weakly horizontally invertible squares for each freely added vertical morphism appearing in the decomposition of $v$. Let $(b',d',\eta',\epsilon')$ be a choice of horizontal adjoint equivalence data for $b'$. Since $F$ satisfies (db2) and (db4) of \cref{def:doublebieq}, there is a horizontal equivalence $a\colon A\xrightarrow{\simeq} A'$ in $\bA$ together with a vertically invertible square $\psi$ in $\bB$ of the form
\begin{tz}
    \node[](1) {$FA$}; 
    \node[right of=1](2) {$B$}; 
    \node[right of=2](3) {$B'$}; 
    \node[right of=3](4) {$FA'$}; 
    \draw[->] (1) to node[above,la]{$b$} node[below,la]{$\simeq$} (2);
    \draw[->] (2) to node[above,la]{$f$} node[below,la]{$\simeq$} (3);
    \draw[->] (3) to node[above,la]{$d'$} node[below,la]{$\simeq$} (4);
    \node[below of=1](1') {$FA$}; 
    \node[below of=4](4') {$FA'$}; 
    \draw[->] (1') to node[below,la]{$Fa$} node[above,la]{$\simeq$} (4');
    \draw[d,pro](1) to (1');
    \draw[d,pro](4) to (4');
    \node[la] at ($(1)!0.5!(4')+(5pt,0)$) {$\vcong$};
    \node[la] at ($(1)!0.5!(4')-(5pt,0)$) {$\psi$};
\end{tz}
Let $u\colon A\arrowdot A'$ be a vertical morphism in $\bA^\whi$ freely added using horizontal adjoint equivalence data for $a$. We get a weakly horizontally square $\beta$, as desired,
\begin{tz}
\node[](1) {$FA$}; 
    \node[right of=1](2) {$B$}; 
    \node[below of=1](3) {$FA'$}; 
    \node[below of=2](4) {$B'$};
    \draw[->,pro] (1) to node[left,la]{$F^\whi u$} (3);
    \draw[->,pro] (2) to node(a)[right,la]{$v$} (4);
    \draw[->] (1) to node[above,la]{$b$} (2);
    \draw[->] (3) to node[below,la]{$b'$} (4);
    \node[la] at ($(1)!0.5!(4)-(5pt,0)$) {$\beta$};
    \node[la] at ($(1)!0.5!(4)+(5pt,0)$) {$\simeq$};

\node[right of=2,yshift=2.25cm,xshift=.3cm](1) {$FA$}; 
\node[right of=1](2) {$B$}; 
\node[below of=1](1') {$FA$}; 
\node[right of=1'](2') {$B$}; 
\node[right of=2'](3') {$B'$}; 
\node[below of=1'](1'') {$FA$}; 
\node[la] at ($(a)!0.5!(1'')$) {$=$};
\node[right of=1''](2'') {$B$}; 
\node[right of=2''](3'') {$B'$}; 
\node[right of=3''](4'') {$FA'$}; 
\node[right of=4''](5'') {$B'$}; 
\node[right of=5''](6'') {$B$}; 
\node[above of=5''](5') {$B'$}; 
\node[right of=5'](6') {$B$}; 
\node[above of=6'](6) {$B$};
\draw[d] (2) to (6);
\draw[d] (3') to (5');
\draw[d,pro] (1) to (1');
\draw[d,pro] (2) to (2');
\draw[d,pro] (6) to (6');
\draw[d,pro] (1') to (1'');
\draw[d,pro] (2') to (2'');
\draw[d,pro] (3') to (3'');
\draw[d,pro] (5') to (5'');
\draw[d,pro] (6') to (6'');
\draw[->] (1) to node[above,la]{$b$} (2);
\draw[->] (1') to node[above,la]{$b$} (2');
\draw[->] (2') to node[above,la]{$f$} (3');
\draw[->] (5') to node[above,la]{$g$} (6');

\draw[->] (1'') to node[below,la]{$b$} (2'');
\draw[->] (2'') to node[below,la]{$f$} (3'');
\draw[->] (3'') to node[above,la]{$d'$} (4'');
\draw[->] (4'') to node[above,la]{$b'$} (5'');
\draw[->] (5'') to node[above,la]{$g$} (6'');

\node[la] at ($(1)!0.5!(2')$) {$e_b$};
\node[la] at ($(1')!0.5!(2'')$) {$e_b$};
\node[la] at ($(2')!0.5!(3'')$) {$e_f$};
\node[la] at ($(5')!0.5!(6'')$) {$e_g$};
    \node[la] at ($(2)!0.5!(6')+(5pt,0)$) {$\vcong$};
    \node[la] at ($(2)!0.5!(6')-(5pt,0)$) {$\eta$};
    \node[la] at ($(3')!0.5!(5'')+(7pt,0)$) {$\vcong$};
    \node[la] at ($(3')!0.5!(5'')-(7pt,0)$) {$(\epsilon')^{-1}$};

\node[below of=1''](1) {$FA$}; 
\node[below of=4''](4) {$FA'$}; 
\node[right of=4](5) {$B'$}; 
\node[right of=5](6) {$B$}; 
\draw[d,pro] (1'') to (1);
\draw[d,pro] (4'') to (4);
\draw[d,pro] (5'') to (5);
\draw[d,pro] (6'') to (6);
\draw[->] (1) to node[below,la]{$Fa$} (4);
\draw[->] (4) to node[above,la]{$b'$} (5);
\draw[->] (5) to node[above,la]{$g$} (6);

\node[la] at ($(4'')!0.5!(5)$) {$e_{b'}$};
\node[la] at ($(5'')!0.5!(6)$) {$e_g$};
    \node[la] at ($(1'')!0.5!(4)+(5pt,0)$) {$\vcong$};
    \node[la] at ($(1'')!0.5!(4)-(5pt,0)$) {$\psi$};
    
\node[below of=1](1') {$FA'$}; 
\node[below of=4](4') {$FA'$}; 
\node[right of=4'](5') {$B'$}; 
\node[right of=5'](6') {$B'$}; 
\draw[d] (1') to (4');
\draw[d] (5') to (6');
\draw[->] (4') to node[below,la]{$b'$} (5');
\draw[->,pro] (1) to node[left,la]{$F^\whi u$} (1');
\draw[->,pro] (6) to node[right,la]{$v$} (6');
\draw[d,pro] (4) to (4');
\draw[d,pro] (5) to (5');
\node[la] at ($(4)!0.5!(5')$) {$e_{b'}$};
    \node[la] at ($(1)!0.5!(4')+(10pt,0)$) {$\simeq$};
    \node[la] at ($(1)!0.5!(4')-(10pt,0)$) {$F^\whi\overline{\alpha}$};
    \node[la] at ($(5)!0.5!(6')+(5pt,0)$) {$\simeq$};
    \node[la] at ($(5)!0.5!(6')-(5pt,0)$) {$\alpha'$};
\end{tz}
where $\overline{\alpha}$ is the weakly horizontally invertible square in $\bA^\whi$ that was freely added with~$u$ (see \cref{descr:J4}), and $\alpha'$ is the weak inverse of the square $\alpha$.
\end{proof}

\begin{prop} \label{prop:dblbieqareinW}
Every double biequivalence is in $\cW$.
\end{prop}

\begin{proof}
Let $F\colon \bA\to \bB$ be a double biequivalence; we show that $F^\whi$ satisfies (db1-4) of \cref{def:doublebieq}. Since $F$ and $F^\whi$ agree on underlying horizontal categories by \cref{rem:idonunderhorcat}, and $F$ satisfies (db1-2), then so does  $F^\whi$. Moreover, since $j_\bA$, $j_\bB$, and $F$ are fully faithful on squares and $F^\whi j_\bA=j_\bB F$, then $F^\whi$ is also fully faithful on squares, i.e., it satisfies (db4). Finally, since every vertical morphism in $\bB^\whi$ can be decomposed as an alternate composite of vertical morphisms in $\bB$ and of composites of freely added vertical morphisms, then the fact that $F^\whi$ satisfies (db3) follows from (db3) for $F$ and \cref{lem:freelyadded}.
\end{proof}

\subsection{The model structure}

By taking cofibrations as the $\cI_w$-cofibrations and weak equivalences as the double functors in $\cW$, we obtain the desired model structure on~$\DblCat$. The relevant classes of morphisms, as well as an outline of the proof with shortcuts to the corresponding results, is provided below; the technical details are deferred to \cref{subsec:prooftheorem}.

\begin{theorem} \label{thm:secondMS}
There is a model structure $(\cC,\cF,\cW)$ on $\DblCat$ such that
\begin{rome}
    \item the class $\cC$ of cofibrations is given by $\cC\coloneqq \Iwcof$, where $\cI_w$ is the set described in \cref{not:gencofDblCat2}, 
    \item the class $\cW$ of weak equivalences is as described in \cref{def:W},
    \item the class $\cF$ of fibrations is given by $\cF\coloneqq (\cC\cap\cW)^\boxslash$, and 
    \item the fibrant objects are the weakly horizontally invariant double categories.
\end{rome}
\end{theorem}

\begin{proof}
By \cref{prop:2outof3andretractW}, we know that the class $\cW$ of weak equivalences satisfies the $2$-out-of-$3$ property. Furthermore, by \cref{prop:FcapW=Iinj}, we have that $\cF\cap\cW=\Iwinj$, and hence the pair $(\cC,\cF\cap\cW)=(\Iwcof,\Iwinj)$ is the weak factorization system generated by the set $\cI_w$ of \cref{not:gencofDblCat2}. The fact that the pair $(\cC\cap\cW,\cF)$ forms a weak factorization system is the content of \cref{thm:factinCcapWandF,cor:CcapWLLPF}. We present in  \cref{cor:whiarefibrant} the desired characterization of fibrant objects.
\end{proof}

\section{\texorpdfstring{$\cJ_w$}{Jw}-cofibrations and \texorpdfstring{$\cJ_w$}{Jw}-injective double functors}\label{subsec:JwcofJwinj}
  
As we saw in the previous section, our proposed classes of cofibrations and of trivial fibrations can be constructed from a generating set $\cI_w$, and admit concise descriptions. Unfortunately, a nice description of the proposed fibrations and trivial cofibrations is not available in general. To prove that these classes of double functors form a weak factorization system, we introduce an auxiliary weak factorization system $(\Jwcof, \Jwinj)$ generated by a set $\cJ_w$ of double functors. Aside from admitting a simple description, the $\cJ_w$-injective double functors contain our proposed fibrations, and agree with these when we restrict to double functors with weakly horizontally invariant target; in particular, they can be used to identify our fibrant objects.

This section is largely technical, and the reader willing to trust our claims is encouraged to jump ahead to \cref{subsec:prooftheorem}.

Let us first introduce the set $\cJ_w$.

\begin{notation} \label{not:Jw}
Let $\cJ_w$ denote the set containing the following double functors: 
\begin{rome}
    \item the inclusion $J_1\colon \mathbbm{1}\to \bH\Eadj$, where the $2$-category $\Eadj$ is the free-living adjoint equivalence, 
    \item the inclusion $J_2\colon \bH \mathbbm{2}\to \bH C_{\mathrm{inv}}$, where the $2$-category $C_{\mathrm{inv}}$ is the free-living $2$-isomorphism, 
    \item the inclusion $J_3\colon \bW^{-}\to \bW$, where the double category $\bW$ is the free-living weakly horizontally invertible square with horizontal adjoint equivalence data, and $\bW^{-}$ is its double subcategory where we remove one of the vertical morphisms.
\begin{tz}
    \node[](A) {$0$};
    \node[right of=A](B) {$1$};
    \node[below of=A](A') {$0'$};
    \node[right of=A'](B') {$1'$};
    \node at ($(B)!0.5!(B')+(.5cm,0)$) {;};
    \node at ($(A)!0.5!(A')-(.75cm,0)$) {$\bW=$};
    \draw[->] (A) to node[above,la]{$\simeq$} (B);
    \draw[->] (A') to node[below,la]{$\simeq$} (B');
    \draw[->,pro] (A) to (A');
    \draw[->,pro] (B) to (B');

    \node[la] at ($(A)!0.5!(B')$) {$\simeq$};
    
    \node[right of=B,space](A) {$0$};
    \node[right of=A](B) {$1$};
    \node[below of=A](A') {$0'$};
    \node[right of=A'](B') {$1'$};
    \node at ($(B)!0.5!(B')+(.5cm,0)$) {.};
    \node at ($(A)!0.5!(A')-(.75cm,0)$) {$\bW^{-}=$};
    \draw[->] (A) to node[above,la]{$\simeq$} (B);
    \draw[->] (A') to node[below,la]{$\simeq$} (B');
    \draw[->,pro] (B) to (B');
\end{tz}
\end{rome}
\end{notation}

\begin{rem} \label{rem:JinIcof}
It is straightforward from the characterization of $\cI_w$-cofibrations given in \cref{thm:charcofDblsecond} that the double functors $J_1$, $J_2$, and $J_3$ are in $\Iwcof$, and from \cref{def:doublebieq} that they are double biequivalences. In particular, by \cref{prop:dblbieqareinW} this implies that they are trivial cofibrations in our proposed model structure on $\DblCat$. 
\end{rem}

\subsection{\texorpdfstring{$\cJ_w$}{Jw}-injective double functors}

By studying what it means to have the right lifting property with respect to the double functors in $\cJ_w$, we can characterize the $\cJ_w$-injective double functors as follows. 

\begin{prop} \label{prop:fibinsecond}
A double functor $F\colon \bA\to \bB$ is in $\Jwinj$ if and only if it satisfies the following conditions: 
\begin{enumerate}
    \item[\textnormal{(df1)}] for every object $C\in \bA$ and every horizontal equivalence $b\colon B\xrightarrow{\simeq} FC$ in $\bB$, there is a horizontal equivalence $a\colon A\xrightarrow{\simeq} C$ in $\bA$ such that $b=Fa$,
    \item[\textnormal{(df2)}] for every horizontal morphism $c\colon A\to C$ in $\bA$ and every vertically invertible square $\beta$ in $\bB$ as depicted below left, there is a vertically invertible square $\alpha$ in $\bA$ as depicted below right such that $\beta=F\alpha$,
    \begin{tz}
    \node[](1) {$FA$}; 
    \node[right of=1](2) {$FC$}; 
    \node[below of=1](3) {$FA$}; 
    \node[below of=2](4) {$FC$};
    \draw[d,pro] (1) to (3);
    \draw[d,pro] (2) to (4);
    \draw[->] (1) to node[above,la]{$b$} (2);
    \draw[->] (3) to node[below,la]{$Fc$} (4);
    \node[la] at ($(1)!0.5!(4)-(5pt,0)$) {$\beta$};
    \node[la] at ($(1)!0.5!(4)+(5pt,0)$) {$\vcong$};

    \node[right of=2,xshift=1cm](1) {$A$}; 
    \node[right of=1](2) {$C$}; 
    \node[below of=1](3) {$A$}; 
    \node[below of=2](4) {$C$};
    \draw[d,pro] (1) to (3);
    \draw[d,pro] (2) to (4);
    \draw[->] (1) to node[above,la]{$a$} (2);
    \draw[->] (3) to node[below,la]{$c$} (4);
    \node[la] at ($(1)!0.5!(4)-(5pt,0)$) {$\alpha$};
    \node[la] at ($(1)!0.5!(4)+(5pt,0)$) {$\vcong$};
    \end{tz}
    \item[\textnormal{(df3)}] for every diagram in $\bA$ as depicted below left, where $a$ and $a'$ are horizontal equivalences, and every weakly horizontally invertible square~$\beta$ in $\bB$ as depicted below middle, there is a weakly horizontally invertible square $\alpha$ in~$\bA$ as depicted below right such that $\beta=F\alpha$.
    \begin{tz}
    \node[](A) {$A$};
    \node[right of=A](B) {$C$};
    \node[below of=A](A') {$A'$};
    \node[right of=A'](B') {$C'$};
    \draw[->] (A) to node[above,la] {$a$} node[below,la]{$\simeq$} (B);
    \draw[->] (A') to node[below,la] {$a'$} node[above,la] {$\simeq$} (B');
    \draw[->,pro] (B) to node[right,la]{$w$} (B');

    \node[right of=B,xshift=1cm](1) {$FA$}; 
    \node[right of=1](2) {$FC$}; 
    \node[below of=1](3) {$FA'$}; 
    \node[below of=2](4) {$FC'$};
    \draw[->,pro] (1) to node[left,la]{$v$} (3);
    \draw[->,pro] (2) to node[right,la]{$Fw$} (4);
    \draw[->] (1) to node[above,la]{$Fa$} node[below,la]{$\simeq$} (2);
    \draw[->] (3) to node[below,la]{$Fa'$} node[above,la]{$\simeq$} (4);
    \node[la] at ($(1)!0.5!(4)-(5pt,0)$) {$\beta$};
    \node[la] at ($(1)!0.5!(4)+(5pt,0)$) {$\simeq$};
    
    \node[right of=2,xshift=1cm](1) {$A$}; 
    \node[right of=1](2) {$C$}; 
    \node[below of=1](3) {$A'$}; 
    \node[below of=2](4) {$C'$};
    \draw[->,pro] (1) to node[left,la]{$u$} (3);
    \draw[->,pro] (2) to node[right,la]{$w$} (4);
    \draw[->] (1) to node[above,la]{$a$} node[below,la]{$\simeq$} (2);
    \draw[->] (3) to node[below,la]{$a'$} node[above,la]{$\simeq$} (4);
    \node[la] at ($(1)!0.5!(4)-(5pt,0)$) {$\alpha$};
    \node[la] at ($(1)!0.5!(4)+(5pt,0)$) {$\simeq$};
    \end{tz}
\end{enumerate}
\end{prop}

\begin{proof}
This is obtained directly from a close inspection of the right lifting properties with respect to the double functors in $\cJ_w$.
\end{proof}

As a consequence, we can use the class $\Jwinj$ to identify the weakly horizontally invariant double categories (see \cref{def:whidbl}).

\begin{cor} \label{prop:whiareinJinj}
A double category $\bA$ is weakly horizontally invariant if and only if the unique double functor $\bA\to \mathbbm{1}$ is in $\Jwinj$.
\end{cor}

The following result tells us that every $\cJ_w$-injective double functor has the right lifting property with respect to $J_4\colon \bH\Eadj\to \bHsim\Eadj$, which is useful when proving that $\cJ_w$-injective double functors with weakly horizontally invariant targets are fibrations.

\begin{prop} \label{prop:JinjRLPJ4}
Let $F\colon \bA\to \bB$ be a double functor in $\Jwinj$. Then $F$ is in $\{J_4\}\mathrm{-inj}$, where $J_4\colon \bH\Eadj\to \bHsim\Eadj$ is the inclusion of \cref{not:J4}.
\end{prop}

\begin{proof}
Consider a commutative square in $\DblCat$ of the form
\begin{tz}
\node[](1) {$\bH\Eadj$}; 
\node[below of=1](2) {$\bHsim\Eadj$}; 
\node[right of=1,xshift=.5cm](3) {$\bA$}; 
\node[below of=3](4) {$\bB$}; 
\draw[->] (1) to node[above,la]{$(a,c,\eta,\epsilon)$} (3);
\draw[->] (2) to node[below,la]{$G$} (4);
\draw[->] (1) to node[left,la]{$J_4$} (2);
\draw[->] (3) to node[right,la]{$F$} (4);
\draw[->,dashed] (2) to node[pos=0.4,above,la]{$L$} (3);
\end{tz}
where $a\colon A\xrightarrow{\simeq} C$ is a horizontal adjoint equivalence with data $(a,c,\eta,\epsilon)$; we want to find a lift $L$ as depicted. The images under $G$ of the weakly horizontally invertible squares $\alpha,\gamma\in\bHsim\Eadj$ from \cref{descr:J4} are as in the two leftmost diagrams below. 
\begin{tz}
    \node[](1) {$FA$}; 
    \node[right of=1](2) {$FC$}; 
    \node[below of=1](3) {$FC$}; 
    \node[below of=2](4) {$FC$};
    \draw[->,pro] (1) to node[left,la]{$Gu$} (3);
    \draw[d,pro] (2) to (4);
    \draw[->] (1) to node[above,la]{$Fa$} node[below,la]{$\simeq$} (2);
    \draw[d] (3) to (4);
    \node[la] at ($(1)!0.5!(4)-(7pt,0)$) {$G\alpha$};
    \node[la] at ($(1)!0.5!(4)+(7pt,0)$) {$\simeq$};
    
    \node[right of=2,xshift=.5cm](1) {$FC$}; 
    \node[right of=1](2) {$FA$}; 
    \node[below of=1](3) {$FA$}; 
    \node[below of=2](4) {$FA$};
    \draw[->,pro] (1) to node[left,la]{$Gv$} (3);
    \draw[d,pro] (2) to (4);
    \draw[->] (1) to node[above,la]{$Fc$} node[below,la]{$\simeq$} (2);
    \draw[d] (3) to (4);
    \node[la] at ($(1)!0.5!(4)-(7pt,0)$) {$G\gamma$};
    \node[la] at ($(1)!0.5!(4)+(7pt,0)$) {$\simeq$};
    
     \node[right of=2,xshift=1cm](1) {$A$}; 
    \node[right of=1](2) {$C$}; 
    \node[below of=1](3) {$C$}; 
    \node[below of=2](4) {$C$};
    \draw[->,pro] (1) to node[left,la]{$\overline{u}$} (3);
    \draw[d,pro] (2) to (4);
    \draw[->] (1) to node[above,la]{$a$} node[below,la]{$\simeq$} (2);
    \draw[d] (3) to (4);
    \node[la] at ($(1)!0.5!(4)-(5pt,0)$) {$\overline\alpha$};
    \node[la] at ($(1)!0.5!(4)+(5pt,0)$) {$\simeq$};
    
    \node[right of=2,xshift=.5cm](1) {$C$}; 
    \node[right of=1](2) {$A$}; 
    \node[below of=1](3) {$A$}; 
    \node[below of=2](4) {$A$};
    \draw[->,pro] (1) to node[left,la]{$\overline{v}$} (3);
    \draw[d,pro] (2) to (4);
    \draw[->] (1) to node[above,la]{$c$} node[below,la]{$\simeq$} (2);
    \draw[d] (3) to (4);
    \node[la] at ($(1)!0.5!(4)-(5pt,0)$) {$\overline\gamma$};
    \node[la] at ($(1)!0.5!(4)+(5pt,0)$) {$\simeq$};
\end{tz}
By (df3) of \cref{prop:fibinsecond}, there are weakly horizontally invertible squares $\overline{\alpha}$ and $\overline{\gamma}$ in~$\bA$, as in the two rightmost diagrams above, such that $F\overline{\alpha}=G\alpha$ and $F\overline{\gamma}=G\gamma$. Finally, by \cref{rem:uniqueHsimEadjdata}, the data $(a,c,\eta,\epsilon)$, $\overline{\alpha}$, and $\overline{\gamma}$ determine a unique double functor $L\colon \bHsim\Eadj\to \bA$ which gives the desired lift. Note that we indeed have $G=FL$ since their images on the generating data of $\bHsim\Eadj$ coincide. 
\end{proof}

\begin{rem} \label{rem:whiRLPwrtJ4}
This result, together with \cref{prop:whiareinJinj}, guarantees that for every weakly horizontally invariant double category $\bA$, the double functor $\bA\to \mathbbm{1}$ is in $\{J_4\}\mathrm{-inj}$.
\end{rem}

Next, we show that the double functors which are $\cJ_w$-injective and double biequivalences are precisely the ones that are $\cI_w$-injective.

\begin{prop} \label{prop:JwinjcapdblbieqisIwinj}
A double functor $F\colon \bA\to \bB$ is $\cJ_w$-injective and a double biequivalence if and only if it is $\cI_w$-injective.
\end{prop}

\begin{proof}
Since $\cJ_w\subseteq \Iwcof$ by \cref{rem:JinIcof}, we have $\Iwinj=\Iwcof^\boxslash\subseteq \cJ_w^\boxslash=\Jwinj$. Furthermore, by \cref{rem:trivfibaredblbieq}, a double functor in $\Iwinj$ is in particular a double biequivalence, which shows the converse statement.

Now suppose that $F$ is $\cJ_w$-injective and a double biequivalence. We prove that $F$ is $\cI_w$-injective using \cref{prop:trivfibinsecond}. It is straightforward to see that $F$ is surjective on objects, full on morphisms, and fully faithful on squares using (db1-2) and (db4) of \cref{def:doublebieq} and (df1-2) of \cref{prop:fibinsecond}. To prove that $F$ is full on vertical morphisms, let $A$,~$A'$ be objects in $\bA$, and $v\colon FA\arrowdot FA'$ be a vertical morphism in $\bB$. Since $F$ satisfies~(db3), there is a vertical morphism $w\colon C\arrowdot C'$ in $\bA$ together with a weakly horizontally invertible square $\beta$ in $\bB$ as depicted below left. 
 \begin{tz}
    \node[](1) {$FA$}; 
    \node[right of=1](2) {$FC$}; 
    \node[below of=1](3) {$FA'$}; 
    \node[below of=2](4) {$FC'$};
    \draw[->,pro] (1) to node[left,la]{$v$} (3);
    \draw[->,pro] (2) to node[right,la]{$Fw$} (4);
    \draw[->] (1) to node[above,la]{$b$} node[below,la]{$\simeq$} (2);
    \draw[->] (3) to node[below,la]{$b'$} node[above,la]{$\simeq$} (4);
    \node[la] at ($(1)!0.5!(4)-(5pt,0)$) {$\beta$};
    \node[la] at ($(1)!0.5!(4)+(5pt,0)$) {$\simeq$};
    
    \node[right of=2,xshift=1cm](1) {$A$}; 
    \node[right of=1](2) {$C$}; 
    \node[below of=1](3) {$A'$}; 
    \node[below of=2](4) {$C'$};
    \draw[->,pro] (1) to node[left,la]{$u$} (3);
    \draw[->,pro] (2) to node[right,la]{$w$} (4);
    \draw[->] (1) to node[above,la]{$a$} node[below,la]{$\simeq$} (2);
    \draw[->] (3) to node[below,la]{$a'$} node[above,la]{$\simeq$} (4);
    \node[la] at ($(1)!0.5!(4)-(5pt,0)$) {$\alpha$};
    \node[la] at ($(1)!0.5!(4)+(5pt,0)$) {$\simeq$};
    \end{tz}
Since $F$ is full on horizontal morphisms and fully faithful on squares, there are horizontal equivalences $a\colon A\xrightarrow{\simeq} C$ and $a'\colon A'\xrightarrow{\simeq} C'$ in $\bA$ such that $b=Fa$ and $b'=Fa'$. Then, by~(df3), there is a weakly horizontally invertible square $\alpha$ in $\bA$ as depicted above right such that $\beta=F\alpha$; in particular, we have $v=Fu$. This completes the proof.
\end{proof}

\subsection{\texorpdfstring{$\cJ_w$}{Jw}-cofibrations and double biequivalences}

We now focus on the $\cJ_w$-co\-fib\-ra\-tions. First, we show that they are cofibrations in our proposed model structure, which additionally satisfy the requirements of a double biequivalence except for condition (db3) on vertical morphisms.

\begin{prop} \label{prop:trivcofibssecond}
Let $J\colon \bA\to \bB$ be a double functor in $\Jwcof$. Then $J$ satisfies the following conditions:
\begin{rome}
    \item it is injective on objects, and faithful on horizontal and vertical morphisms,
    \item it satisfies \textnormal{(db1-2)} and \textnormal{(db4)} of \cref{def:doublebieq}.
\end{rome}
\end{prop}

\begin{proof}
Since $\cJ_w\subseteq \Iwcof$ by \cref{rem:JinIcof}, we have that $\Jwcof\subseteq \Iwcof$; hence $J$ is injective on objects, and faithful on horizontal and vertical morphisms, by \cref{rem:injectivityofIw}.

Now, since objects can only be added along $J_1\colon \mathbbm 1\to \bH\Eadj$, i.e., with a horizontal equivalence to an object which was already there, then $J$ satisfies (db1). Similarly, as horizontal morphisms can only be added along $J_2\colon \bH\mathbbm 2\to C_\mathrm{inv}$, we can check that $J$ satisfies (db2). Finally note that $J$ satisfies (db4), since taking pushouts along $J_1$, $J_2$, and $J_3$ does not create new squares within an existing boundary, nor does it identify squares. 
\end{proof}

When the source of a $\cJ_w$-cofibration is a weakly horizontally invariant double category, we can further show that (db3) of \cref{def:doublebieq} is satisfied, and hence that every such $\cJ_w$-cofibration is a double biequivalence. 

\begin{prop} \label{prop:Jwcellwithwhisourcearedblbieq}
Let $J\colon \bA\to \bB$ be a double functor in $\Jwcof$ such that $\bA$ is weakly horizontally invariant. Then $J$ is a double biequivalence. 
\end{prop}

\begin{proof}
We first prove the case when $J\in\Jwcell$. By \cref{prop:trivcofibssecond}, we have that $J$ satisfies (db1-2) and (db4) of \cref{def:doublebieq}; it remains to show (db3). Let $\lambda$ be an ordinal and let $\bX\colon \lambda\to \DblCat$ be a transfinite composition of pushouts of double functors in $\cJ_w$ such that $J$ is the composite
\[ J\colon \bA\cong J(\bA)=\bX_0\xrightarrow{\iota_0} \mathrm{colim}_{\mu<\lambda} \bX_\mu=\bB. \]

Let $v\colon B\arrowdot B'$ be a vertical morphism in $\bB$. We use transfinite induction to show that there is a vertical morphism $u\colon A\arrowdot A'$ in $\bA$ and a weakly horizontally invertible square~$\beta$ in $\bB$ of the form
\begin{tz}
\node[](A) {$JA$};
\node[right of=A](B) {$B$};
\node[below of=A](A') {$JA'$};
\node[right of=A'](B') {$B'$};
\draw[->,pro] (A) to node[left,la] {$Ju$} (A');
\draw[->,pro] (B) to node[right,la] {$v$} (B');
\draw[->] (A) to node[above, la] {$\simeq$} (B);
\draw[->] (A') to node[below, la] {$\simeq$} (B');

\node[la] at ($(A)!0.5!(B')+(5pt,0)$) {$\simeq$};
\node[la] at ($(A)!0.5!(B')-(5pt,0)$) {$\beta$};
\end{tz}
If $v\in \bX_0=J(\bA)$, then there is a vertical morphism $u\colon A\arrowdot A'$ in $\bA$ such that $Ju=v$ and we can take $\beta=\id_{Ju}$.

Now suppose that $v\in \bX_{\mu+1}$ for a successor ordinal $\mu+1<\lambda$. If $v\in \bX_\mu$, then we are done by induction. Otherwise, the double category $\bX_{\mu+1}$ was obtained as a pushout along $J_3$ as depicted below left 
\begin{tz}
\node[](1) {$\bW^{-}$}; 
\node[below of=1](2) {$\bW$}; 
\node[right of=1,xshift=.5cm](3) {$\bX_\mu$}; 
\node[below of=3](4) {$\bX_{\mu+1}$}; 
\draw[->] (1) to node[above,la]{$(w,d,d')$} (3);
\draw[->] (1) to node[left,la]{$J_3$} (2);
\draw[->] (3) to node[right,la]{$i_\mu$} (4);
\draw[->] (2) to node[below,la]{$\delta$} (4);
\node at ($(4)+(-8pt,10pt)$) {$\ulcorner$};

\node[right of=3,xshift=1cm](A) {$D$};
\node[right of=A](B) {$Y$};
\node[below of=A](A') {$D'$};
\node[right of=A'](B') {$Y'$};
\draw[->,pro] (A) to node[left,la] {$w$} (A');
\draw[->,pro] (B) to node[right,la] {$\overline{w}$} (B');
\draw[->] (A) to node[below, la] {$\simeq$} node[above,la]{$d$} (B);
\draw[->] (A') to node[below, la]{$d'$} node[above,la] {$\simeq$} (B');

\node[la] at ($(A)!0.5!(B')+(5pt,0)$) {$\simeq$};
\node[la] at ($(A)!0.5!(B')-(5pt,0)$) {$\delta$};
\end{tz} 
where $w$ is a vertical morphism in $\bX_\mu$, $d$, $d'$ are horizontal equivalences in $\bX_\mu$, and $\delta$ is a weakly horizontally invertible square in $\bB$, as depicted above right. Then the vertical morphism $v\in \bX_{\mu+1}$ is a composite of vertical morphisms in $\bX_\mu$ and the freely added vertical morphism $\overline{w}$. We prove that the result holds for a composite of the form $v=v_1 \overline{w}v_0$ with $v_0\colon B\arrowdot Y$ and $v_1\colon Y'\arrowdot B'$ two vertical morphisms in $\bX_\mu$; the general case proceeds similarly.

By induction, since $v_0$, $v_1$, and $w$ are in $\bX_\mu$, there are vertical morphisms $u_0$, $u_1$, and $t$ in $\bA$, and weakly horizontally invertible squares $\beta_0$, $\beta_1$ and $\varphi$ in~$\bB$, as depicted below.
\begin{tz}
\node[](A) {$JA$};
\node[right of=A](B) {$B$};
\node[below of=A](A') {$JC$};
\node[right of=A'](B') {$Y$};
\draw[->,pro] (A) to node[left,la] {$Ju_0$} (A');
\draw[->,pro] (B) to node[right,la] {$v_0$} (B');
\draw[->] (A) to node[above, la] {$b_0$} node[below,la] {$\simeq$} (B);
\draw[->] (A') to node[below, la] {$b'_0$} node[above, la] {$\simeq$} (B');

\node[la] at ($(A)!0.5!(B')+(5pt,0)$) {$\simeq$};
\node[la] at ($(A)!0.5!(B')-(5pt,0)$) {$\beta_0$};

\node[right of=B,xshift=.5cm](A) {$JC'$};
\node[right of=A](B) {$Y'$};
\node[below of=A](A') {$JA'$};
\node[right of=A'](B') {$B'$};
\draw[->,pro] (A) to node[left,la] {$Ju_1$} (A');
\draw[->,pro] (B) to node[right,la] {$v_1$} (B');
\draw[->] (A) to node[above, la] {$b_1$} node[below,la] {$\simeq$} (B);
\draw[->] (A') to node[below, la] {$b'_1$} node[above, la] {$\simeq$} (B');

\node[la] at ($(A)!0.5!(B')+(5pt,0)$) {$\simeq$};
\node[la] at ($(A)!0.5!(B')-(5pt,0)$) {$\beta_1$};

\node[right of=B,xshift=.5cm](A) {$JX$};
\node[right of=A](B) {$D$};
\node[below of=A](A') {$JX'$};
\node[right of=A'](B') {$D'$};
\draw[->,pro] (A) to node[left,la] {$Jt$} (A');
\draw[->,pro] (B) to node[right,la] {$w$} (B');
\draw[->] (A) to node[above, la] {$f$} node[below,la] {$\simeq$} (B);
\draw[->] (A') to node[below, la] {$f'$} node[above, la]{$\simeq$} (B');

\node[la] at ($(A)!0.5!(B')+(5pt,0)$) {$\simeq$};
\node[la] at ($(A)!0.5!(B')-(5pt,0)$) {$\varphi$};
\end{tz}

Let $(df,g,\eta,\epsilon)$ and $(d'f',g',\eta',\epsilon')$ be horizontal adjoint equivalence data in $\bB$ for the composites $df$ and $d'f'$. Since $J$ satisfies (db2) and (db4), there are horizontal equivalences $a\colon C\xrightarrow{\simeq} X$ and $a'\colon C'\xrightarrow{\simeq} X'$ in $\bA$ together with vertically invertible squares $\psi$ and $\psi'$ in~$\bB$ as in the two leftmost squares below.
\begin{tz}
    \node[](1) {$JC$}; 
    \node[right of=1](2) {$Y$}; 
    \node[right of=2](4) {$JX$}; 
    \draw[->] (1) to node[above,la]{$b'_0$} node[below,la]{$\simeq$} (2);
    \draw[->] (2) to node[above,la]{$g$} node[below,la]{$\simeq$} (4);
    \node[below of=1](1') {$JC$}; 
    \node[below of=4](4') {$JX$}; 
    \draw[->] (1') to node[below,la]{$Ja$} node[above,la]{$\simeq$} (4');
    \draw[d,pro](1) to (1');
    \draw[d,pro](4) to (4');
    \node[la] at ($(1)!0.5!(4')+(5pt,0)$) {$\vcong$};
    \node[la] at ($(1)!0.5!(4')-(5pt,0)$) {$\psi$};
    
    \node[right of=4,xshift=.5cm](1) {$JC'$}; 
    \node[right of=1](2) {$Y'$}; 
    \node[right of=2](4) {$JX'$}; 
    \draw[->] (1) to node[above,la]{$b_1$} node[below,la]{$\simeq$} (2);
    \draw[->] (2) to node[above,la]{$g'$} node[below,la]{$\simeq$} (4);
    \node[below of=1](1') {$JC'$}; 
    \node[below of=4](4') {$JX'$}; 
    \draw[->] (1') to node[below,la]{$Ja'$} node[above,la]{$\simeq$} (4');
    \draw[d,pro](1) to (1');
    \draw[d,pro](4) to (4');
    \node[la] at ($(1)!0.5!(4')+(5pt,0)$) {$\vcong$};
    \node[la] at ($(1)!0.5!(4')-(5pt,0)$) {$\psi'$};
    
    \node[right of=4,xshift=1cm](A) {$C$};
\node[right of=A](B) {$X$};
\node[below of=A](A') {$C'$};
\node[right of=A'](B') {$X$};
\draw[->,pro] (A) to node[left,la] {$\overline{u}$} (A');
\draw[->,pro] (B) to node[right,la] {$t$} (B');
\draw[->] (A) to node[above, la] {$a$} node[below,la] {$\simeq$} (B);
\draw[->] (A') to node[below, la] {$a'$} node[above, la] {$\simeq$} (B');

\node[la] at ($(A)!0.5!(B')+(5pt,0)$) {$\simeq$};
\node[la] at ($(A)!0.5!(B')-(5pt,0)$) {$\alpha$};
\end{tz}
Then, as $\bA$ is weakly horizontally invariant, there is a vertical morphism $\overline{u}\colon C\arrowdot C'$ and a weakly horizontally invertible square $\alpha$ in $\bA$ as depicted above right. Setting $u\coloneqq u_1 \overline{u} u_0\colon A\arrowdot A'$ and  considering the following pasting of squares in~$\bB$
\begin{tz}
\node[](1) {$JA$}; 
\node[right of=1](2) {$B$}; 
\node[right of=2,xshift=3cm](5) {$B$}; 
\draw[->] (1) to node[above,la]{$b_0$} (2); 
\draw[d] (2) to (5); 

\node[below of=1](1') {$JC$}; 
\node[right of=1'](2') {$Y$}; 
\node[below of=5](5') {$Y$}; 
\draw[->] (1') to node[below,la]{$b'_0$} (2'); 
\draw[d] (2') to (5'); 
\draw[->,pro] (1) to node[left,la]{$Ju_0$} (1');
\draw[->,pro] (2) to node[right,la]{$v_0$} (2');
\draw[->,pro] (5) to node[right,la]{$v_0$} (5');

\node[la] at ($(1)!0.5!(2')-(5pt,0)$) {$\beta_0$}; 
\node[la] at ($(1)!0.5!(2')+(5pt,0)$) {$\simeq$}; 
\node[la] at ($(2)!0.5!(5')$) {$\id_{v_0}$};  

\node[below of=1'](1) {$JC$}; 
\node[right of=1](2) {$Y$}; 
\node[right of=2](3) {$JX$}; 
\node[right of=3](4) {$D$}; 
\node[right of=4](5) {$Y$};
\draw[->] (1) to node[below,la]{$b'_0$} (2);
\draw[->] (2) to node[below,la]{$g$} (3); 
\draw[->] (3) to node[above,la]{$f$} (4); 
\draw[->] (4) to node[above,la]{$d$} (5);
\draw[d,pro] (1') to (1);
\draw[d,pro] (2') to (2);
\draw[d,pro] (5') to (5);

\node[la] at ($(1')!0.5!(2)$) {$e_{b'_0}$}; 
\node[la] at ($(2')!0.5!(5)-(5pt,0)$) {$\epsilon^{-1}$}; 
\node[la] at ($(2')!0.5!(5)+(5pt,0)$) {$\vcong$}; 

\node[below of=1](1') {$JC$}; 
\node[below of=3](3') {$JX$}; 
\node[right of=3'](4') {$D$}; 
\node[right of=4'](5') {$Y$};
\draw[->] (1') to node[above,la]{$Ja$} (3'); 
\draw[->] (3') to node[above,la]{$f$} (4'); 
\draw[->] (4') to node[above,la]{$d$} (5');
\draw[d,pro] (1) to (1');
\draw[d,pro] (3) to (3');
\draw[d,pro] (4) to (4');
\draw[d,pro] (5) to (5');

\node[la] at ($(1)!0.5!(3')-(5pt,0)$) {$\psi$}; 
\node[la] at ($(1)!0.5!(3')+(5pt,0)$) {$\vcong$}; 
\node[la] at ($(3)!0.5!(4')$) {$e_{f}$}; 
\node[la] at ($(4)!0.5!(5')$) {$e_{d}$}; 

\node[below of=1'](1) {$JC'$}; 
\node[below of=3'](3) {$JX'$}; 
\node[right of=3](4) {$D'$}; 
\node[right of=4](5) {$Y'$};
\draw[->] (1) to node[below,la]{$Ja'$} (3); 
\draw[->] (3) to node[below,la]{$f'$} (4); 
\draw[->] (4) to node[below,la]{$d'$} (5);
\draw[->,pro] (1') to node[left,la]{$J\overline{u}$} (1);
\draw[->,pro] (3') to node[left,la]{$Jt$} (3);
\draw[->,pro] (4') to node[left,la]{$w$} (4);
\draw[->,pro] (5') to node[right,la]{$\overline{w}$} (5);

\node[la] at ($(1')!0.5!(3)-(7pt,0)$) {$J\alpha$}; 
\node[la] at ($(1')!0.5!(3)+(7pt,0)$) {$\simeq$}; 
\node[la] at ($(3')!0.5!(4)-(5pt,0)$) {$\varphi$}; 
\node[la] at ($(3')!0.5!(4)+(5pt,0)$) {$\simeq$}; 
\node[la] at ($(4')!0.5!(5)-(5pt,0)$) {$\delta$}; 
\node[la] at ($(4')!0.5!(5)+(5pt,0)$) {$\simeq$};

\node[below of=1](1') {$JC'$}; 
\node[right of=1'](2') {$Y'$}; 
\node[right of=2'](3') {$JX'$}; 
\node[right of=3'](4') {$D'$}; 
\node[right of=4'](5') {$Y'$};
\draw[->] (1') to node[above,la]{$b_1$} (2');
\draw[->] (2') to node[above,la]{$g'$} (3'); 
\draw[->] (3') to node[below,la]{$f'$} (4'); 
\draw[->] (4') to node[below,la]{$d'$} (5');
\draw[d,pro] (1) to (1');
\draw[d,pro] (3) to (3');
\draw[d,pro] (4) to (4');
\draw[d,pro] (5) to (5');

\node[la] at ($(1)!0.5!(3')-(7pt,0)$) {$(\psi')^{-1}$}; 
\node[la] at ($(1)!0.5!(3')+(7pt,0)$) {$\vcong$}; 
\node[la] at ($(3)!0.5!(4')$) {$e_{f'}$}; 
\node[la] at ($(4)!0.5!(5')$) {$e_{d'}$}; 

\node[below of=1'](1) {$JC'$}; 
\node[right of=1](2) {$Y'$}; 
\node[below of=5'](5) {$Y'$}; 
\draw[->] (1) to node[above,la]{$b_1$} (2); 
\draw[d] (2) to (5); 
\draw[d,pro] (1') to (1);
\draw[d,pro] (2') to (2);
\draw[d,pro] (5') to (5);

\node[la] at ($(1')!0.5!(2)$) {$e_{b_1}$}; 
\node[la] at ($(2')!0.5!(5)-(5pt,0)$) {$\epsilon'$}; 
\node[la] at ($(2')!0.5!(5)+(5pt,0)$) {$\vcong$};

\node[below of=1](1') {$JA'$}; 
\node[right of=1'](2') {$B'$}; 
\node[below of=5](5') {$B'$}; 
\draw[->] (1') to node[below,la]{$b'_1$} (2'); 
\draw[d] (2') to (5'); 
\draw[->,pro] (1) to node[left,la]{$Ju_1$} (1');
\draw[->,pro] (2) to node[right,la]{$v_1$} (2');
\draw[->,pro] (5) to node[right,la]{$v_1$} (5');

\node[la] at ($(1)!0.5!(2')-(5pt,0)$) {$\beta_1$}; 
\node[la] at ($(1)!0.5!(2')+(5pt,0)$) {$\simeq$}; 
\node[la] at ($(2)!0.5!(5')$) {$\id_{v_1}$};  
\end{tz}
we obtain a weakly horizontally invertible square of the desired form between the vertical morphisms $Ju=(Ju_1)(J\overline{u})(Ju_0)$ and $v=v_1\overline{w}v_0$. 

Finally, if $v\in \bX_\kappa=\mathrm{colim}_{\mu<\kappa}\bX_\mu$ for a limit ordinal $\kappa<\lambda$, there is an ordinal $\mu<\kappa$ such that $v\in \bX_\mu$, and we are done by induction. This shows (db3) for $J$, and proves that $J$ is a double biequivalence.

Now if $J\colon \bA\to \bB$ is in $\Jwcof$, then it is a retract of a double functor $K\colon \bA\to \bC$ in $\Jwcell$, whose source is also the weakly horizontally invariant double category $\bA$. By the first part of the proof, the double functor $K$ is a double biequivalence, and therefore so is $J$.
\end{proof}

\subsection{Fibrations and \texorpdfstring{$\cJ_w$}{Jw}-injective double functors}

To conclude this section, we prove our claim that a double functor whose target is weakly horizontally invariant is a fibration precisely when it is $\cJ_w$-injective. We start by showing that the class of fibrations is included in $\Jwinj$. 

\begin{lemme} \label{lem:FinJwinj}
We have that $\cF\subseteq \Jwinj$. 
\end{lemme}

\begin{proof}
Since every double functor in $\cJ_w$ is a double biequivalence by \cref{rem:JinIcof}, it is in $\cW$ by \cref{prop:dblbieqareinW}. This, together with \cref{rem:JinIcof}, implies that $\cJ_w\subseteq \cC\cap\cW$. Therefore $\cF=(\cC\cap\cW)^\boxslash\subseteq \cJ_w^\boxslash=\Jwinj$, which concludes the proof.
\end{proof}

For the converse inclusion, we will use the next incremental lemmas, which ultimately ensure that the weakly horizontally invariant replacement of a trivial cofibration is a $\cJ_w$-cofibration. 

\begin{lemme} \label{lem:whiofCcapWisIcof}
Let $I\colon \bA\to \bB$ be a double functor in $\cC=\Iwcof$ which is fully faithful on squares. Then the induced double functor $I^\whi\colon \bA^\whi\to \bB^\whi$ is in $\cC$. 
\end{lemme}

\begin{proof}
We show that $I^\whi$ is in $\Iwcof$ by using \cref{thm:charcofDblsecond}. Since the double functors~$I$ and $I^\whi$ coincide on underlying horizontal categories by \cref{rem:idonunderhorcat}, and $I\in\Iwcof$, then the functor $U\bfH I=U\bfH I^\whi$ has the left lifting property with respect to surjective on objects and full functors. It remains to prove that $U\bfV I^\whi$ satisfies this lifting property. Let $P\colon \cX\to \cY$ be a surjective on objects and full functor, and consider a commutative square as below left. 
\begin{tz}
\node[](1) {$U\bfV\bA^\whi$}; 
\node[below of=1](2) {$U\bfV\bB^\whi$}; 
\node[right of=1,xshift=.5cm](3) {$\cX$}; 
\node[below of=3](4) {$\cY$}; 
\draw[->] (1) to node[above,la]{$F$} (3);
\draw[->] (2) to node[below,la]{$G$} (4);
\draw[->] (1) to node[left,la]{$U\bfV I^\whi$} (2);
\draw[->] (3) to node[right,la]{$P$} (4);
\draw[->,dashed] (2) to node[pos=0.4,above,la]{$L$} (3);

\node[right of=3,xshift=1cm](1) {$U\bfV\bA$}; 
\node[below of=1](2) {$U\bfV\bB$}; 
\node[right of=1,xshift=.5cm](3) {$\cX$}; 
\node[below of=3](4) {$\cY$}; 
\draw[->] (1) to node[above,la]{$F \circ U\bfV j_\bA$} (3);
\draw[->] (2) to node[below,la]{$G \circ U\bfV j_\bB$} (4);
\draw[->] (1) to node[left,la]{$U\bfV I$} (2);
\draw[->] (3) to node[right,la]{$P$} (4);
\draw[->,dashed] (2) to node[pos=0.4,above,la]{$K$} (3);
\end{tz}
Since $I$ is in $\Iwcof$, then $U\bfV I$ has the left lifting property with respect to $P$ by \cref{thm:charcofDblsecond}, and hence there is a lift $K\colon U\bfV\bB\to \cX$ in the  diagram above right. We define the functor $L\colon U\bfV \bB^\whi\to \cX$ to be $K$ on the subcategory $U\bfV \bB$ and as follows on the freely added morphisms. Let $v\colon A\arrowdot A'$ be a vertical morphism in $\bB^\whi$ freely added to $\bB$ using the horizontal adjoint equivalence $\underline{b}=(b,d,\eta,\epsilon)$. If there is a horizontal adjoint equivalence $\underline{a}=(a,c,\eta',\epsilon')$ in $\bA$ such that $I\underline{a}=\underline{b}$, then $\underline{a}$ is the unique such data since $I$ is injective on objects and faithful on horizontal morphisms by \cref{rem:injectivityofIw}, and fully faithful on squares by assumption. Hence there is a unique vertical morphism $u$ in~$\bA^\whi$ (freely added using $\underline{a}$) such that $I^\whi(u)=v$ and we set $Lv=Fu$. If there are no such $\underline{a}$ in $\bA$ with $I\underline{a}=\underline{b}$, we can choose a morphism $w\colon LA\to LA'$ in $\cX$ such that $Pw=Gv$, by fullness of $P$, and set $Lv=w$. This gives the desired lift. 
\end{proof}

\begin{lemme} \label{prop:IinCcapWthenIwhiinJcof}
If $I\colon \bA\to \bB$ is a double functor in $\cC\cap \cW$, then $I^\whi\colon \bA^\whi\to \bB^\whi$ is in $\Jwcof$.
\end{lemme}

\begin{proof}
First recall that, since $I\in \cW$, the double functor $I^\whi\colon \bA^\whi\to \bB^\whi$ is a double biequivalence by definition. Next, consider a factorization $I^\whi=PJ$ with $J\in \Jwcof$ and $P\in \Jwinj$. As $\bA^\whi$ is weakly horizontally invariant, \cref{prop:Jwcellwithwhisourcearedblbieq} ensures $J$ is also a double biequivalence; then, by 2-out-of-3, so is $P$. Hence $P$ is both $\cJ_w$-injective and a double biequivalence, and therefore it is $\cI_w$-injective by \cref{prop:JwinjcapdblbieqisIwinj}.

Now, since $I$ is in $\cW$, it is fully faithful on squares, and so it follows from \cref{lem:whiofCcapWisIcof} that $I^\whi$ is in $\Iwcof$. Then $I^\whi$ has the left lifting property with respect to $P\in \Iwinj$, so, by the retract argument, it is a retract of $J\in \Jwcell$ and hence is itself in $\Jwcof$.
\end{proof}

Finally, we prove that every $\cJ_w$-injective double functor with weakly horizontally invariant target has the right lifting property with respect to every trivial cofibration $I$, by using its lifting property against the weakly horizontally invariant replacement $I^\whi$. 

\begin{prop} \label{prop:F=Jinjbtwwhi}
Let $P\colon \bA\to \bB$ be a double functor with $\bB$ weakly horizontally invariant. Then $P$ is in $\cF= (\cC\cap\cW)^\boxslash$ if and only if $P$ is in $\Jwinj$.
\end{prop}

\begin{proof}
If $P$ is in $\cF$, then $P$ is in $\Jwinj$ by \cref{lem:FinJwinj}. 

Now suppose that $P$ is in $\Jwinj$. We show that $P$ has the right lifting property with respect to every double functor in $\cC\cap\cW$, i.e., it is in $\cF$. Let $I\colon \bC\to \bD$ be a double functor in $\cC\cap\cW$ and consider a commutative square in $\DblCat$ as below; we want to find a lift  $L\colon \bD\to \bA$  as pictured. 
\begin{tz}
\node[](1) {$\bC$}; 
\node[below of=1](2) {$\bD$}; 
\node[right of=1](3) {$\bA$}; 
\node[below of=3](4) {$\bB$}; 
\draw[->] (1) to node[above,la]{$F$} (3);
\draw[->] (2) to node[below,la]{$G$} (4);
\draw[->] (1) to node[left,la]{$I$} (2);
\draw[->] (3) to node[right,la]{$P$} (4);
\draw[->,dashed] (2) to node[pos=0.4,above,la]{$L$} (3);
\end{tz}
Since $\bB$ is weakly horizontally invariant, there is a lift in the diagram below left by \cref{rem:whiRLPwrtJ4}, which yields a double functor $\hat{G}\colon \bD^\whi\to \bB$ as in the diagram below right, given by the universal property of the pushout.
\begin{tz}
\node[](1) {$\bigsqcup_{\mathrm{HorEq}(\bD)} \bH\Eadj$};
\node[right of=1,xshift=1cm](2) {$\bD$}; 
\node[below of=1](3) {$\bigsqcup_{\mathrm{HorEq}(\bD)} \bHsim\Eadj$}; 
\node[right of=2](4) {$\bB$}; 
\draw[->] (1) to (2);
\draw[->] (2) to node[above,la]{$G$} (4);
\draw[->] (1) to node[left,la]{$\bigsqcup_{\mathrm{HorEq}(\bD)} J_4$} (3);
\draw[->,dashed] (3) to node[pos=0.4,below,la]{$K$} (4);

\node[right of=4,xshift=2cm,yshift=.5cm](1) {$\bigsqcup_{\mathrm{HorEq}(\bD)} \bH\Eadj$}; 
\node[right of=1,xshift=1.5cm](2) {$\bD$};
\node[below of=1](3) {$\bigsqcup_{\mathrm{HorEq}(\bD)} \bHsim\Eadj$};
\node[below of=2](4) {$\bD^\whi$}; 
\draw[->] (1) to (2);
\draw[->] (1) to node[left,la]{$\bigsqcup_{\mathrm{HorEq}(\bD)} J_4$} (3); 
\draw[->] (2) to node[right,la]{$j_\bD$} (4); 
\draw[->] (3) to node[below,la]{$\tau_\bD$} (4);
\node[below right of=4,xshift=.5cm](5) {$\bB$}; 
\draw[->,bend left] (2) to node[right,la]{$G$} (5);
\draw[->,bend right=20] (3) to node[below,la]{$K$} (5);
\draw[->,dashed] (4) to node[pos=0.4,left,la,xshift=-3pt,yshift=-3pt]{$\hat{G}$} (5);

\node at ($(4)-(8pt,-8pt)$) {$\ulcorner$};
\end{tz}
Now, since $P\in \Jwinj$, by \cref{prop:JinjRLPJ4} there is a lift in the commutative diagram below left, which in turns yields a double functor $\hat{F}\colon \bC^\whi\to \bA$ as in the diagram below right, given by the universal property of the pushout.
\begin{tz}
\node[](1) {$\bigsqcup_{\mathrm{HorEq}(\bC)} \bH\Eadj$};
\node[right of=1,xshift=1.5cm](2) {$\bC$}; 
\node[below of=1](3) {$\bigsqcup_{\mathrm{HorEq}(\bC)} \bHsim\Eadj$}; 
\node[right of=2](4) {$\bA$}; 
\node[below of=4](5) {$\bB$};
\draw[->] (1) to (2);
\draw[->] (2) to node[above,la]{$F$} (4);
\draw[->] (4) to node[right,la]{$P$} (5); 
\draw[->] (1) to node[left,la]{$\bigsqcup_{\mathrm{HorEq}(\bC)} J_4$} (3);
\draw[->,dashed] (3) to node[pos=0.4,below,la]{$K'$} (4);
\draw[->] (3) to node[below,la]{$K\circ (\bigsqcup_I \id)$} (5);

\node[right of=4,xshift=2cm,yshift=.5cm](1) {$\bigsqcup_{\mathrm{HorEq}(\bC)} \bH\Eadj$}; 
\node[right of=1,xshift=1.5cm](2) {$\bC$};
\node[below of=1](3) {$\bigsqcup_{\mathrm{HorEq}(\bC)} \bHsim\Eadj$};
\node[below of=2](4) {$\bC^\whi$}; 
\draw[->] (1) to (2);
\draw[->] (1) to node[left,la]{$\bigsqcup_{\mathrm{HorEq}(\bC)} J_4$} (3); 
\draw[->] (2) to node[right,la]{$j_\bC$} (4); 
\draw[->] (3) to (4);
\node[below right of=4,xshift=.5cm](5) {$\bA$}; 
\draw[->,bend left] (2) to node[right,la]{$F$} (5);
\draw[->,bend right=20] (3) to node[below,la]{$K'$} (5);
\draw[->,dashed] (4) to node[pos=0.4,left,la,xshift=-3pt,yshift=-3pt]{$\hat{F}$} (5);

\node at ($(4)-(8pt,-8pt)$) {$\ulcorner$};
\end{tz}
Here $\bigsqcup_I \id\colon \bigsqcup_{\mathrm{HorEq}(\bC)} \bHsim\Eadj\to \bigsqcup_{\mathrm{HorEq}(\bD)} \bHsim\Eadj$ is the double functor induced by the action of $I$ on $\mathrm{HorEq}(\bC)$. By construction of $\hat{F}$ and $\hat{G}$, we have that the following diagram commutes.
\begin{tz}
\node[](A) {$\bC$};
\node[below of=A](B) {$\bD$};
\node[right of=A](1) {$\bC^\whi$}; 
\node[below of=1](2) {$\bD^\whi$}; 
\node[right of=1](3) {$\bA$}; 
\node[below of=3](4) {$\bB$}; 
\draw[->] (A) to node[above,la]{$j_\bC$} (1);
\draw[->] (1) to node[above,la]{$\hat{F}$} (3);
\draw[->] (B) to node[below,la]{$j_\bD$} (2);
\draw[->] (2) to node[below,la]{$\hat{G}$} (4);
\draw[->] (A) to node[left,la]{$I$} (B);
\draw[->] (1) to node[left,la]{$I^\whi$} (2);
\draw[->] (3) to node[right,la]{$P$} (4);
\draw[->,dashed] (2) to node[pos=0.4,above,la]{$\hat{L}$} (3);
\draw[->,bend left=45] (A) to node[above,la]{$F$} (3);
\draw[->,bend right=45] (B) to node[below,la]{$G$} (4);
\end{tz}
Since $I^\whi\in\Jwcof$ by \cref{prop:IinCcapWthenIwhiinJcof} as $I\in \cC\cap\cW$, there is a lift $\hat{L}$ in the right-hand square of the diagram above, and the composite $L\coloneqq \hat{L}j_\bD$ gives the desired lift. 
\end{proof}

\section{Proof of \texorpdfstring{\cref{thm:secondMS}}{Theorem 3.20}} \label{subsec:prooftheorem}

We now use the technical results of \cref{subsec:JwcofJwinj} to prove the remaining claims in \cref{thm:secondMS}. Namely, we show that the pairs $(\cC, \cF\cap\cW)$ and $(\cC\cap\cW,\cF)$ form weak factorization systems, and identify the fibrant objects as the weakly horizontally invariant double categories.

Since by definition we have that $\cC=\Iwcof$, in order to prove that $(\cC, \cF\cap\cW)$ is a weak factorization system it suffices to show that $\cF\cap \cW=\Iwinj$; this is the content of the following result.

\begin{prop}\label{prop:FcapW=Iinj}
We have that $\cF\cap \cW=\Iwinj$.
\end{prop}

\begin{proof}
Since $\cC\cap\cW\subseteq \cC=\Iwcof$, it follows that $\Iwinj=\Iwcof^\boxslash\subseteq (\cC\cap\cW)^\boxslash=\cF$. Moreover, every double functor in $\Iwinj$ is a double biequivalence by \cref{rem:trivfibaredblbieq}, and these are in $\cW$ by \cref{prop:dblbieqareinW}; hence $\Iwinj\subseteq \cW$.

For the inclusion $\cF\cap\cW\subseteq \Iwinj$, note that every double functor $P$ in $\cF\cap \cW$ factors as $P=QI$ with $I\in \cC=\Iwcof$ and $Q\in \Iwinj$. Since $Q\in \cW$ by the above inclusion, and $P\in \cW$ by assumption, we get that $I\in \cW$ by $2$-out-of-$3$; hence $I\in \cC\cap \cW$. Therefore, since $P\in \cF=(\cC\cap\cW)^\boxslash$ has the right lifting property with respect to $I$, by the retract argument we have that $P$ is a retract of $Q$ and hence is also in $\Iwinj$. 
\end{proof}

Before moving on to the next factorization system, we focus on the fibrant objects. Aside from obtaining the desired characterization as the double categories which are weakly horizontally invariant, we see that the weakly horizontally invariant replacements $j_\bA\colon \bA\to \bA^\whi$ of \cref{def:whireplacement} are trivial cofibrations, and hence fibrant replacements in our model structure.

\begin{theorem} \label{cor:whiarefibrant}
A double category $\bA$ is fibrant if and only if it is weakly horizontally invariant. 
\end{theorem}

\begin{proof}
We recall from \cref{prop:whiareinJinj} that a double category $\bA$ is weakly horizontally invariant if and only if the double functor $\bA\to \mathbbm{1}$ is in $\Jwinj$. Since $\mathbbm{1}$ is weakly horizontally invariant, by \cref{prop:F=Jinjbtwwhi} this holds if and only if $\bA\to \mathbbm 1$ is in $\cF$, i.e., $\bA$ is fibrant.
\end{proof}

\begin{prop} \label{lem:J_Aforwhiisdblbieq}
Let $\bA$ be a weakly horizontally invariant double category. Then the double functor $j_\bA\colon \bA\to \bA^\whi$ is a double biequivalence.
\end{prop}

\begin{proof}
By construction, $j_\bA\colon \bA\to \bA^\whi$ is a double functor in $\{J_4\}\mathrm{-cof}$ (see \cref{def:whireplacement}). Since $\Jwinj\subseteq \{J_4\}\mathrm{-inj}$ by \cref{prop:JinjRLPJ4}, we have that  \[ \{J_4\}\mathrm{-cof}={}^\boxslash \{J_4\}\mathrm{-inj}\subseteq {}^\boxslash\Jwinj=\Jwcof. \]
Hence $j_\bA$ is a $\cJ_w$-cofibration with weakly horizontally invariant source, and thus a double biequivalence by \cref{prop:Jwcellwithwhisourcearedblbieq}.
\end{proof}

\begin{cor} \label{prop:j_AinCcapW}
The double functor $j_\bA\colon \bA\to \bA^\whi$ is in $\cC\cap\cW$. In particular, this exhibits $\bA^\whi$ as a fibrant replacement of $\bA$.
\end{cor}

\begin{proof}
Since $\Jwinj\subseteq \{J_4\}\mathrm{-inj}$ by \cref{prop:JinjRLPJ4} and $\Jwcof\subseteq \Iwcof$ by \cref{rem:JinIcof}, we have that $J_4$ is in $\Iwcof=\cC$. Hence so is $j_\bA$, as it is constructed as a pushout of coproducts of $J_4$. The fact that $j_\bA$ is in $\cW$ follows from the relation $(j_\bA)^\whi=j_{\bA^\whi}$ and the fact that the latter is a double biequivalence by \cref{lem:J_Aforwhiisdblbieq}. The second statement then follows from \cref{cor:whiarefibrant}.
\end{proof}

We can also prove, using \cref{lem:J_Aforwhiisdblbieq}, that every weak equivalence with fibrant source is a double biequivalence. In particular, this implies that while our weak equivalences are more general, when restricted to the fibrant double categories they agree with the weak equivalences of the model structure on $\DblCat$ of \cite{MSV}: the double biequivalences. As we will see in \cref{section:whitehead}, the weak equivalences with fibrant source also admit a familiar description in terms of pseudo-inverses.

\begin{prop} \label{prop:Wwithwhisourceisdblbieq}
Let $F\colon \bA\to \bB$ be a double functor with $\bA$ weakly horizontally invariant. Then $F$ is in $\cW$ if and only if $F$ is a double biequivalence.
\end{prop}

\begin{proof}
If $F$ is a double biequivalence, then $F$ is in $\cW$ by \cref{prop:dblbieqareinW}. 

Now suppose that $F$ is in $\cW$, i.e., that $F^\whi\colon \bA^\whi\to \bB^\whi$ is a double biequivalence. We prove that $F$ satisfies (db1-4) of \cref{def:doublebieq}. Since $F$ and $F^\whi$ coincide on underlying horizontal categories by \cref{rem:idonunderhorcat} and $j_\bB$ is fully faithful on squares, then $F$ satisfies (db1-2) as $F^\whi$ does so. Moreover, since $j_\bA$, $j_\bB$, and $F^\whi$ are fully faithful on squares and $F^\whi j_\bA=j_\bB F$, then $F$ satisfies (db4). 

It remains to prove (db3). Since $\bA$ is weakly horizontally invariant, \cref{lem:J_Aforwhiisdblbieq} guarantees that $j_\bA\colon \bA\to \bA^\whi$ is a double biequivalence; hence so is the composite $F^\whi j_\bA\colon \bA\to \bB^\whi$. Then, given a vertical morphism $v\colon B\arrowdot B'$ in~$\bB$, there is a vertical morphism $u\colon A\arrowdot A'$ in $\bA$ and a weakly horizontally invertible square $\beta$ in $\bB^\whi$ as depicted below left.
\begin{tz}
\node[](A) {$FA$};
\node[right of=A](B) {$B$};
\node[below of=A](A') {$FA'$};
\node[right of=A'](B') {$B'$};
\draw[->,pro] (A) to node[left,la] {$F^\whi j_\bA u=j_\bB Fu$} (A');
\draw[->,pro] (B) to node[right,la] {$j_\bB v$} (B');
\draw[->] (A) to node[above, la] {$b$} node[below,la] {$\simeq$} (B);
\draw[->] (A') to node[below, la] {$b'$} node[above, la]{$\simeq$} (B');

\node[la] at ($(A)!0.5!(B')+(5pt,0)$) {$\simeq$};
\node[la] at ($(A)!0.5!(B')-(5pt,0)$) {$\beta$};

\node[right of=B,xshift=1cm](A) {$FA$};
\node[right of=A](B) {$B$};
\node[below of=A](A') {$FA'$};
\node[right of=A'](B') {$B'$};
\draw[->,pro] (A) to node[left,la] {$Fu$} (A');
\draw[->,pro] (B) to node[right,la] {$v$} (B');
\draw[->] (A) to node[above, la] {$b$} node[below,la] {$\simeq$} (B);
\draw[->] (A') to node[below, la] {$b'$} node[above, la]{$\simeq$} (B');

\node[la] at ($(A)!0.5!(B')+(5pt,0)$) {$\simeq$};
\node[la] at ($(A)!0.5!(B')-(5pt,0)$) {$\beta'$};
\end{tz}
By fully faithfulness on squares of $j_\bB$, we get a weakly horizontally invertible square $\beta'$ in~$\bB$ as depicted above right, which shows (db3).
\end{proof}

We are now ready to finish the proof of \cref{thm:secondMS} by showing that the classes of trivial cofibrations and fibrations form a weak factorization system. We first show that every double functor can be factored as a trivial cofibration followed by a fibration. 

\begin{theorem} \label{thm:factinCcapWandF}
Every double functor $F\colon \bA\to \bB$ can be factored as $F=RI$ with $I\in \cC\cap\cW$ and $R\in \cF$. 
\end{theorem}

\begin{proof}
Given a double functor $F\colon \bA\to \bB$, we factor $F^\whi$ as follows
\begin{tz}
    \node[](1) {$\bA^\whi$}; 
    \node[below right of=1,xshift=.5cm](2) {$\bC$}; 
    \node[above right of=2,xshift=.5cm](3) {$\bB^\whi$}; 
    \draw[->] (1) to node[above,la]{$F^\whi$} (3);
    \draw[->] (1) to node[pos=0.4,below,la,xshift=-2pt]{$J$}  (2);
    \draw[->] (2) to node[pos=0.6,below,la]{$P$} (3);
\end{tz}
where $J\in \Jwcof$ and $P\in \Jwinj$. As $\bB^\whi$ is weakly horizontally invariant, by \cref{prop:F=Jinjbtwwhi} we have that $P\in\cF$, and hence $\bC$ is also weakly horizontally invariant. Define~$\bD$ to be the pullback of $P$ along $j_\bB$ as in the following diagram.
\begin{tz}
\node[](1) {$\bD$}; 
\node[right of=1,xshift=.5cm](2) {$\bB$};
\node[below of=1](3) {$\bC$};
\node[below of=2](4) {$\bB^\whi$}; 
\draw[->] (1) to node[above,la]{$P'$} (2);
\draw[->] (1) to node[left,la]{$\pi$} (3); 
\draw[->] (2) to node[right,la]{$j_\bB$} (4); 
\draw[->] (3) to node[below,la]{$P$} (4);
\node[above left of=1,xshift=-.5cm,yshift=-.5cm](5) {$\bA$}; 
\node[below of=5](6) {$\bA^\whi$};
\draw[->,bend left] (5) to node[above,la]{$F$} (2);
\draw[->] (5) to node[left,la]{$j_\bA$} (6);
\draw[->,dashed] (5) to node[pos=0.6,above,la,xshift=2pt]{$K$} (1);
\draw[->] (6) to node[pos=0.4,below,la,xshift=-2pt]{$J$} (3);

\node at ($(1)+(8pt,-8pt)$) {$\lrcorner$};
\end{tz}
Then there is a unique double functor $K\colon \bA\to \bD$ making the above diagram commute. To prove the result, it suffices to show that $K$ is in $\cW$. Indeed, assume that this is the case and factor $K$ as $K=QI$ with $I\in \Iwcof=\cC$ and $Q\in \Iwinj=\cF\cap\cW$, where the latter equality holds by \cref{prop:FcapW=Iinj}. As $K,Q\in \cW$, then $I\in \cC\cap \cW$ by $2$-out-of-$3$. Hence, as $F=P'K$, this gives a factorization of $F$ as $F=RI$ with $I\in \cC\cap\cW$ and $R\coloneqq P'Q\in \cF$, as desired.

As $J$ is in $\cW$ by \cref{prop:dblbieqareinW}, $j_\bA$ is in $\cW$ by \cref{prop:j_AinCcapW}, and $\pi K =Jj_\bA$, in order to prove that $K$ is in $\cW$, by $2$-out-of-$3$ it is enough to show that $\pi$ is in $\cW$. For this, we construct a double functor $\hat{\pi}\colon \bD^\whi\to \bC$ such that $\pi=\hat{\pi} j_\bD$ and then show that $\hat{\pi}$ is a double biequivalence; this implies that $\pi\in \cW$ by $2$-out-of-$3$. 

Let $T\coloneqq \mathrm{HorEq}(\bD)\setminus K(\mathrm{HorEq}(\bA))$. As $\bC$ is weakly horizontally invariant, by \cref{rem:whiRLPwrtJ4} there is a lift $L$ in the following diagram.
\begin{tz}
\node[](1) {$(\bigsqcup_{\mathrm{HorEq}(\bA)} \bHsim\Eadj)\bigsqcup (\bigsqcup_T\bH\Eadj)$};
\node[right of=1,xshift=2.75cm](2) {$\bA^\whi\bigsqcup \bD$}; 
\node[below of=1](3) {$\bigsqcup_{\mathrm{HorEq}(\bD)} \bHsim\Eadj$}; 
\node[right of=2,xshift=1cm](4) {$\bC$}; 
\draw[->] (1) to (2);
\draw[->] (2) to node[above,la]{$J\bigsqcup \pi$} (4);
\draw[->] (1) to node[left,la]{$(\bigsqcup_{K}\id)\bigsqcup(\bigsqcup_{T} J_4)$} (3);
\draw[->,dashed] (3) to node[pos=0.4,below,la]{$L$} (4);
\end{tz}
This yields a double functor $\hat{\pi}\colon \bD^\whi\to \bC$, given by the universal property of the pushout, as depicted below.
\begin{tz}
\node[](1) {$\bigsqcup_{\mathrm{HorEq}(\bD)} \bH\Eadj$}; 
\node[right of=1,xshift=1.5cm](2) {$\bD$};
\node[below of=1](3) {$\bigsqcup_{\mathrm{HorEq}(\bD)} \bHsim\Eadj$};
\node[below of=2](4) {$\bD^\whi$}; 
\draw[->] (1) to (2);
\draw[->] (1) to node[left,la]{$\bigsqcup_{\mathrm{HorEq}(\bD)} J_4$} (3); 
\draw[->] (2) to node[right,la]{$j_\bD$} (4); 
\draw[->] (3) to (4);
\node[below right of=4,xshift=.5cm](5) {$\bC$}; 
\draw[->,bend left] (2) to node[right,la]{$\pi$} (5);
\draw[->,bend right=20] (3) to node[below,la]{$L$} (5);
\draw[->,dashed] (4) to node[pos=0.4,left,la,xshift=-3pt,yshift=-3pt]{$\hat{\pi}$} (5);

\node at ($(4)-(8pt,-8pt)$) {$\ulcorner$};
\end{tz}

We finally show $\hat{\pi}$ satisfies (db1-4). First note that $\pi$ is fully faithful on squares as it is a pullback of $j_{\bB}$ which satisfies this condition. Hence $\hat{\pi}$ also satisfies (db4), since $\hat{\pi} j_\bD=\pi$. By \cref{prop:Jwcellwithwhisourcearedblbieq}, we know that $J\colon \bA^\whi\to \bC$ in $\Jwcof$ is a double biequivalence, and so (db1-3) for $\hat{\pi}$ follow from the fact that $J$ satisfies (db1-3) and that $\hat{\pi} K^\whi=J$, by construction.
\end{proof}

As a direct consequence of this result, we get that the trivial cofibrations are precisely the double functors which have the left lifting property with respect to all fibrations. This concludes the proof of the existence of the model structure.  

\begin{cor} \label{cor:CcapWLLPF}
We have that $\cC\cap\cW={}^\boxslash\cF$.
\end{cor}

\begin{proof}
By definition of $\cF$, we already know that $\cC\cap\cW\subseteq {}^\boxslash \cF$. The reverse inclusion follows from \cref{thm:factinCcapWandF}, the retract argument, and the fact that $\cC\cap\cW$ is closed under retracts.
\end{proof}

\begin{rem} \label{rem:JwcofinCcapW}
This shows that $\Jwcof\subseteq \cC\cap\cW$. Indeed, we have that $\cF\subseteq \Jwinj$ by \cref{lem:FinJwinj}, and hence $\Jwcof={}^\boxslash\Jwinj\subseteq {}^\boxslash \cF=\cC\cap\cW$.
\end{rem}

\section{Quillen pairs} \label{sec:Quillenpairs}

Having constructed a new model structure on $\DblCat$, it is natural to wonder how it compares to the one defined by the authors in \cite{MSV}. We settle this question by showing that the identity functor induces a Quillen pair between our two model structures, but not a Quillen equivalence.

We then devote the rest of the section to comparing our model structure on $\DblCat$ to Lack's model structure on $\TwoCat$; see \cite{Lack2Cat,LackBicat} for more details. As in \cite{MSV}, the horizontal embedding $\bH\colon\TwoCat\to\DblCat$ is a left Quillen and homotopically fully faithful functor, but it is no longer right Quillen as it does not preserve fibrant objects. Instead, this role is now played by its more homotopical version $\bHsim\colon\TwoCat\to\DblCat$, which is also homotopically fully faithful. Furthermore, the double category $\bHsim\cA$ associated to a $2$-category $\cA$ is weakly horizontally invariant and provides a fibrant replacement for $\bH\cA$. 

First, we show that the identity adjunction is a Quillen reflection, embedding the homotopy theory of weakly horizontally invariant double categories into that of double categories.

\begin{theorem} \label{thm:idQuillenrefl}
The identity adjunction
\begin{tz}
\node[](A) {$\DblCat_\mathrm{whi}$};
\node[right of=A,xshift=1.5cm](B) {$\DblCat$};
\draw[->] ($(B.west)+(0,.25cm)$) to [bend right=25] node[above,la]{$\id$} ($(A.east)+(0,.25cm)$);
\draw[->] ($(A.east)-(0,.25cm)$) to [bend right=25] node[below,la]{$\id$} ($(B.west)+(0,-.25cm)$);
\node[la] at ($(A.east)!0.5!(B.west)$) {$\bot$};
\end{tz}
is a Quillen pair between the model structure on $\DblCat$ for weakly horizontally invariant double categories of \cref{thm:secondMS} and the one of \cite[Theorem 3.19]{MSV}. Moreover, the derived counit is levelwise a weak equivalence.
\end{theorem}

\begin{proof}
The set $\cI'$ of generating cofibrations introduced in \cite[Proposition 4.3]{MSV} for the model structure on $\DblCat$ constructed therein can be described as the set $\cI_w$ where the inclusion $I_3\colon \mathbbm 1\sqcup \mathbbm 1\to \vtwo$ is replaced by the unique map $\emptyset\to \vtwo$. Since the latter is also in $\Iwcof$, it follows that $\cI'\mathrm{-cof}\subseteq \Iwcof$, and hence $\id\colon \DblCat\to \DblCat_\mathrm{whi}$ preserves cofibrations. Furthermore, by \cref{prop:dblbieqareinW}, we have that the class of double biequivalences ---which is precisely the class of weak equivalences for the model structure on $\DblCat$ of \cite{MSV}--- is contained in the class $\cW$ of weak equivalences in $\DblCat_\whi$, and hence $\id\colon \DblCat\to \DblCat_\mathrm{whi}$ also preserves weak equivalences. This shows that the identity adjunction is a Quillen pair. 

It remains to show that the derived counit is levelwise a weak equivalence in $\DblCat_\whi$. Let $\bA$ be a fibrant double category in $\DblCat_\whi$. Then the component of the derived counit at $\bA$ is given by the cofibrant replacement $q_{\bA}\colon \bA^\cof\to \bA$ in the model structure on~$\DblCat$ of \cite{MSV}. In particular, the double functor $q_\bA$ is a double biequivalence, and hence a weak equivalence in $\DblCat_\whi$ by \cref{prop:dblbieqareinW}. 
\end{proof}

However, the identity adjunction does not induce a Quillen equivalence between the two model structures on $\DblCat$, as shown in the following remark. 

\begin{rem}
The derived unit of the identity adjunction above is not a levelwise double biequivalence. To see this, let $\bA$ be the double category generated by the following data,
\begin{tz}
\node[](1) {$A$}; 
\node[below of=1](2) {$A'$}; 
\node[right of=2](2') {$B'$}; 
\node[below of=2'](3') {$B''$}; 
\draw[->,pro] (1) to node[left,la]{$v$} (2);
\draw[->,pro] (2') to node[right,la]{$w$} (3');
\draw[->] (2) to node[above,la]{$a$} node[below,la]{$\simeq$} (2');
\end{tz}
where $a$ is a horizontal adjoint equivalence. By \cite[Proposition~4.11]{MSV}, $\bA$ is cofibrant in the model structure on $\DblCat$ of \cite{MSV}. Then the component of the derived unit at $\bA$ is given by a fibrant replacement of $\bA$ in  $\DblCat_\whi$, and hence we can consider the weakly horizontally invariant replacement $j_\bA\colon \bA\to \bA^\whi$ given in \cref{def:whireplacement}. In particular, a vertical morphism $u\colon A'\arrowdot B'$ is freely added to~$\bA^\whi$, and the composite $wuv\colon A\arrowdot B''$ in~$\bA^\whi$ does not admit a lift along $j_\bA$ as required by (db3) of \cref{def:doublebieq}. This shows that $j_\bA$ is not a double biequivalence.
\end{rem}

As a direct consequence of the above theorem, and the fact that $\bH\dashv \bfH$ is a Quillen pair between Lack's model structure on $\TwoCat$ and the model structure on $\DblCat$ of \cite{MSV}, we get that $\bH\dashv \bfH$ is also a Quillen pair between $\TwoCat$ and the model structure on~$\DblCat$ introduced in this paper. Moreover, the derived unit is levelwise a biequivalence, and so the functor $\bH$ is homotopically fully faithful.

\begin{theorem} \label{thm:HHcorefsecond}
The adjunction 
\begin{tz}
\node[](A) {$\DblCat$};
\node[right of=A,xshift=1cm](B) {$\TwoCat$};
\draw[->] ($(B.west)+(0,.25cm)$) to [bend right=25] node[above,la]{$\bH$} ($(A.east)+(0,.25cm)$);
\draw[->] ($(A.east)-(0,.25cm)$) to [bend right=25] node[below,la]{$\bfH$} ($(B.west)+(0,-.25cm)$);
\node[la] at ($(A.east)!0.5!(B.west)$) {$\bot$};
\end{tz}
is a Quillen pair between Lack's model structure and the model structure of \cref{thm:secondMS}. Moreover, the derived unit is levelwise a biequivalence.
\end{theorem}

\begin{proof}
The fact that this is a Quillen pair follows directly from \cref{thm:idQuillenrefl} and \cite[Proposition 6.1]{MSV}. To show that the derived unit is levelwise a biequivalence, let $\cA$ be a cofibrant $2$-category. The component of the derived unit at $\cA$ is given by the underlying horizontal $2$-functor of a fibrant replacement $\bfH j_{\bH\cA}\colon \cA=\bfH\bH\cA\to \bfH(\bH\cA)^\fib$ of the horizontal double category $\bH\cA$ in $\DblCat$. In particular, if we consider the fibrant replacement given in \cref{def:whireplacement}, it does not change the underlying horizontal $2$-category of $\bH\cA$ by \cref{rem:idonunderhorcat}. Hence $\bfH j_{\bH\cA}$ is an identity, and in particular a biequivalence.  
\end{proof}

As opposed to the case where $\DblCat$ is endowed with the model structure of \cite{MSV} (see \cite[Theorem 6.4]{MSV}), the horizontal embedding is not right Quillen when we consider our new model structure.

\begin{rem} \label{rem:hordblcatarenotwhi}
The functor $\bH$ is not right Quillen as, for example, the horizontal double category $\bH\Eadj$ is not weakly horizontally invariant, where $\Eadj$ denotes the free-living adjoint equivalence. Indeed, there is no vertical morphism in $\bH\Eadj$ filling the diagram below, as $\bH\Eadj$ only contains trivial vertical morphisms.
\begin{tz}
    \node[right of=B,space](A) {$0$};
    \node[right of=A](B) {$1$};
    \node[below of=A](A') {$1$};
    \node[right of=A'](B') {$1$};
    \draw[->] (A) to node[above,la]{$\simeq$} (B);
    \draw[d] (A') to (B');
    \draw[d,pro] (B) to (B');
\end{tz}
Since every 2-category is fibrant, this implies that $\bH$ does not preserve fibrant objects.
\end{rem}

This shortcoming of the horizontal embedding $\bH$ can be remedied by instead considering the homotopical horizontal embedding $\bHsim\colon \TwoCat\to \DblCat$ of \cref{def:htilde}. As we will see, the adjunction $\Lsim\dashv\bHsim$ of \cref{prop:bHsimadj} is compatible with the model structures considered, making the functor $\bHsim$  right Quillen. As a first step towards this, we show that $\bHsim$ provides a levelwise fibrant replacement of $\bH$ in our model structure on $\DblCat$.

\begin{theorem}\label{rem:Hsimiswhi}
Let $\cA$ be a $2$-category. Then the double category $\bHsim\cA$ is weakly horizontally invariant and the inclusion $J_\cA\colon \bH\cA\to \bHsim\cA$ is a double biequivalence. In particular, this exhibits $\bHsim\cA$ as a fibrant replacement of $\bH\cA$ in the model structure on $\DblCat$ of \cref{thm:secondMS}.
\end{theorem}

\begin{proof}
For the first statement, we have by \cite[Lemma A.2.3]{Lyne} that a weakly horizontally invertible square $\sq{\alpha}{a}{c}{u}{w}$ in $\bHsim \cA$ corresponds to a $2$-isomorphism $\alpha\colon wa\Rightarrow cu$ in $\cA$, where $a$ and $c$ are equivalences in $\cA$. In particular, given a boundary in $\bHsim\cA$ as below left,
\begin{tz}
    \node[](A) {$A$};
    \node[right of=A](B) {$C$};
    \node[below of=A](A') {$A'$};
    \node[right of=A'](B') {$C'$};
    \draw[->] (A) to node[above,la] {$a$} node[below,la] {$\simeq$} (B);
    \draw[->] (A') to node[below,la] {$a'$} node[above,la] {$\simeq$}  (B');
    \draw[->,pro] (B) to node[right,la]{$w$} (B');
    
    \node[right of=B, xshift=1cm](A) {$A$};
    \node[right of=A](B) {$C$};
    \node[below of=A](A') {$A'$};
    \node[right of=A'](B') {$C'$};
    \draw[->] (A) to node[above,la] {$a$} node[below,la] {$\simeq$} (B);
    \draw[->] (A') to node[below,la] {$a'$} node[above,la] {$\simeq$}  (B');
    \draw[->,pro] (A) to node[left,la]{$u$} (A');
    \draw[->,pro] (B) to node[right,la]{$w$} (B');
    
    \node[la] at ($(A)!0.5!(B')+(5pt,0)$) {$\simeq$};
    \node[la] at ($(A)!0.5!(B')-(5pt,0)$) {$\alpha$};
\end{tz}
there is an equivalence $u\colon A\xrightarrow{\simeq} A'$ and a $2$-isomorphism $\alpha\colon wa\cong a'u$ in $\cA$, which provides a square as desired, depicted above right. This shows that $\bHsim\cA$ is weakly horizontally invariant.

For the second statement, recall that $\bfH\bH\cA=\cA=\bfH\bHsim\cA$, and thus the inclusion $J_\cA\colon \bH\cA\to \bHsim\cA$ is the identity on underlying horizontal $2$-categories; this shows that $J_\cA$ satisfies (db1-2) and (db4) of \cref{def:doublebieq}. It remains to show (db3). Let $u\colon A\arrowdot A'$ be a vertical morphism in $\bHsim\cA$, i.e., an adjoint equivalence $u\colon A\xrightarrow{\simeq} A'$ in $\cA$. Then the square 
\begin{tz}
 \node[right of=B,xshift=1.5cm](A) {$A$};
    \node[right of=A](B) {$A$};
    \node[below of=A](A') {$A$};
    \node[right of=A'](B') {$A'$};
    \draw[d] (A) to (B);
    \draw[->] (A') to node[below,la]{$u$} node[above,la]{$\simeq$}  (B');
    \draw[d,pro] (A) to node[left,la]{$J_\cA(e_A)$} (A');
    \draw[pro,->] (B) to node[right,la]{$u$} (B');
    
    \node[la] at ($(A)!0.5!(B')$) {$\simeq$};
\end{tz}
induced by the identity at $u$ gives a weakly horizontally invertible square in $\bHsim\cA$ as required. This shows that $J_\cA$ is a double biequivalence. The second statement then follows from \cref{prop:dblbieqareinW}.
\end{proof}

We now show that the double functor $\bHsim$ is right Quillen, and moreover, that it is homotopically fully faithful.

\begin{theorem} \label{thm:LsimHsimrefsecond}
The adjunction 
\begin{tz}
\node[](A) {$\TwoCat$};
\node[right of=A,xshift=1cm](B) {$\DblCat$};
\draw[->] ($(B.west)+(0,.25cm)$) to [bend right=25] node[above,la]{$\Lsim$} ($(A.east)+(0,.25cm)$);
\draw[->] ($(A.east)-(0,.25cm)$) to [bend right=25] node[below,la]{$\bHsim$} ($(B.west)+(0,-.25cm)$);
\node[la] at ($(A.east)!0.5!(B.west)$) {$\bot$};
\end{tz}
is a Quillen pair between Lack's model structure and the model structure of \cref{thm:secondMS}. Moreover, the derived counit is levelwise a weak equivalence.
\end{theorem}

\begin{proof}
We show that $\bHsim\colon \TwoCat\to \DblCat$ preserves fibrations and trivial fibrations. Let $F\colon \cA\to \cB$ be a Lack fibration in $\TwoCat$. Since $\bHsim\cB$ is weakly horizontally invariant (see \cref{rem:Hsimiswhi}), \cref{prop:F=Jinjbtwwhi} guarantees that $\bHsim F$ is a fibration in $\DblCat$ if and only if it is $\cJ_w$-injective. Hence, we need to prove that $\bHsim F$ satisfies (df1-3) of \cref{prop:fibinsecond}.

First note that $\bHsim F$ satisfies (df1-2) by definition of $F$ being a fibration in $\TwoCat$. It remains to prove (df3). Consider a diagram in $\bHsim\cA$ as below left, together with a weakly horizontally invertible square $\beta$ in $\bHsim\cB$, as depicted below right, 
\begin{tz}
    \node[](A) {$A$};
    \node[right of=A](B) {$C$};
    \node[below of=A](A') {$A'$};
    \node[right of=A'](B') {$C'$};
    \draw[->] (A) to node[above,la] {$a$} node[below,la] {$\simeq$} (B);
    \draw[->] (A') to node[above,la] {$a'$} node[below,la] {$\simeq$}  (B');
    \draw[->,pro] (B) to node[right,la]{$w$} (B');
    
    \node[right of=B,xshift=1.5cm](A) {$FA$};
    \node[right of=A](B) {$FC$};
    \node[below of=A](A') {$FA'$};
    \node[right of=A'](B') {$FC'$};
    \draw[->] (A) to node[above,la] {$Fa$} node[below,la]{$\simeq$} (B);
    \draw[->] (A') to node[below,la]{$Fa'$} node[above,la]{$\simeq$}  (B');
    \draw[->,pro] (A) to node[left,la]{$v$} (A');
    \draw[->,pro] (B) to node[right,la]{$Fw$} (B');
    
    \node[la] at ($(A)!0.5!(B')+(5pt,0)$) {$\simeq$};
    \node[la] at ($(A)!0.5!(B')-(5pt,0)$) {$\beta$};
\end{tz}
i.e., a $2$-isomorphism $\beta\colon (Fw)(Fa)\Rightarrow (Fa')v$ in $\cB$ by \cite[Lemma A.2.3]{Lyne}. Let $(c',a',\eta,\epsilon)$ be a choice of adjoint equivalence data for $a'$, and let $\delta$ be the $2$-isomorphism in $\cB$ given by the pasting below left. 
\begin{tz}[node distance=1.8cm]
    \node[](A) {$FA$};
    \node[right of=A](B) {$FC$};
    \node[below of=A](A') {$FA'$};
    \node[right of=A'](B') {$FC'$};
    \draw[->] (A) to node[above,la] {$Fa$} node[below,la]{$\simeq$} (B);
    \draw[->] (A') to node[over,la]{$Fa'$}  (B');
    \draw[->] (A) to node[left,la]{$v$} node[right,la]{$\lsimeq$} (A');
    \draw[->] (B) to node[right,la]{$Fw$} node[left,la]{$\rsimeq$} (B');
    \node[below of=B'](A'') {$FA'$};
    \draw[->] (B') to node[right,la]{$Fc'$} node[left,la]{$\rsimeq$} (A'');
    \draw[d] (A') to node(a)[]{} (A'');
    
    \celli[la,left,yshift=4pt]{B}{A'}{$\beta$};
    \celli[la,left,yshift=5pt][n][0.6]{B'}{a}{$F\epsilon$};
    
    \node[right of=B,xshift=1cm](A) {$A$};
    \node[right of=A](B) {$C$};
    \node[below of=B](B') {$C'$};
    \draw[->] (A) to node[above,la] {$a$} node[below,la]{$\simeq$} (B);
    \node[below of=B'](A'') {$C'$};
    \node[left of=A''](A') {$A'$};
    \draw[->] (B) to node(b)[right,la]{$w$} node[left,la]{$\rsimeq$} (B');
    \draw[->] (A) to node(v)[left,la]{$u$} node[right,la]{$\lsimeq$} (A');
    \draw[->] (A') to node[below,la]{$a'$} node[above,la]{$\simeq$}  (A'');
    \draw[->] (B') to node[above,xshift=-3pt,la]{$c'$} (A');
    \draw[d] (B') to node(a)[]{} (A'');
    
    \celli[la,above,xshift=-10pt]{b}{v}{$\overline{\alpha}$};
    \celli[la,below,xshift=5pt][n][0.28]{a}{A'}{$\eta$};
\end{tz}
Since $F$ is a fibration in $\TwoCat$, there is an equivalence $u\colon A\xrightarrow{\simeq} A'$ in~$\cA$ and a $2$-isomorphism $\overline{\alpha}\colon c'wa\cong u$ in~$\cA$ such that $\delta=F\overline{\alpha}$. We set $\alpha\colon wa\cong a'u$ to be the pasting above right; by the triangle identities for $(\eta,\epsilon)$, we get that $\beta=F\alpha$ as desired. This proves that $\bHsim F$ is a fibration in~$\DblCat$.

Now suppose that $F\colon \cA\to \cB$ is a trivial fibration in $\TwoCat$; by definition, we directly see that $\bHsim F$ is surjective on objects, full on horizontal morphisms, and fully faithful on squares. Fullness on vertical morphisms for $\bHsim F$ follows from the fact that a lift of an adjoint equivalence by a biequivalence is also an adjoint equivalence. Hence $\bHsim F$ is a trivial fibration in $\DblCat$ by \cref{prop:trivfibinsecond}, and this shows that $\bHsim$ is right Quillen. 

It remains to show that the derived counit is levelwise a biequivalence. Let $\cA$ be a $2$-category, and let $q_{\bHsim \cA}\colon (\bHsim\cA)^\cof\to \bHsim\cA$ denote the cofibrant replacement of $\bHsim\cA$ constructed as follows. The double category $(\bHsim\cA)^\cof$ has the same objects as $\cA$; it has a copy $\overline a$ for each morphism $a$ in $\cA$ and horizontal morphisms in $(\bHsim\cA)^\cof$ are given by free composites of $\overline a$'s; it has a copy $\overline{u}$ for each adjoint equivalence $u$ in $\cA$ and vertical morphisms in $(\bHsim\cA)^\cof$ are given by free composites of $\overline{u}$'s; and squares in $(\bHsim\cA)^\cof$ are given by squares of $\bHsim\cA$ whose boundaries are the actual composites in $\bHsim\cA$ of the representative of the free composites.

Then, by studying the data of the $2$-category $\Lsim(\bHsim\cA)^\cof$, we can see that the derived counit at $\cA$
\begin{tz}
\node[](1) {$\Lsim(\bHsim\cA)^\cof$}; 
\node[right of=1,xshift=1.4cm](2) {$\Lsim\bHsim\cA$}; 
\node[right of=2,xshift=.6cm](3) {$\cA$}; 
\draw[->] (1) to node[above,la]{$\Lsim q_{\bHsim\cA}$} (2);
\draw[->] (2) to node[above,la]{$\epsilon_\cA$} (3);
\end{tz}
is a trivial fibration in $\TwoCat$ as it is surjective on objects, full on morphisms, and fully faithful on $2$-morphisms. 
\end{proof}

\begin{rem}
The components of the derived unit of the adjunction  $\Lsim\dashv\bHsim$ are not weak equivalences in $\DblCat$. Indeed, since every 2-category is fibrant, we know that the counit and the derived counit agree on cofibrant double categories. Then, if we consider the component $\eta_{\bV\mathbbm{2}}\colon\bV\mathbbm{2}\to\bHsim\Lsim\bV\mathbbm{2}=\bHsim\Eadj$ of the unit at the cofibrant double category $\bV\mathbbm{2}$, we see that $\bHsim\Eadj$ has non-trivial horizontal morphisms, given by the adjoint equivalence created by $\Lsim$ from the unique vertical morphism of $\bV\mathbbm{2}$, while~$\bV\mathbbm{2}$ does not. Therefore $\eta_{\bV\mathbbm{2}}$ is not a double biequivalence, as it does not satisfy (db2). Then, since $\vtwo$ is weakly horizontally invariant, \cref{prop:Wwithwhisourceisdblbieq} implies that $\eta_{\bV\mathbbm{2}}$ is not a weak equivalence in $\DblCat$.
\end{rem}

While \cref{thm:LsimHsimrefsecond} implies that $\bHsim\colon \TwoCat\to \DblCat$ preserves weak equivalences and fibrations, the following result says that it further reflects these classes of double functors. Hence the model structure on $\TwoCat$ is completely determined from our model structure on $\DblCat$ through its image under $\bHsim$. 

\begin{theorem}\label{thm:2CatriHsim}
Lack's model structure on $\TwoCat$ is right-induced along the adjunction
\begin{tz}
\node[](A) {$\TwoCat$};
\node[right of=A,xshift=1cm](B) {$\DblCat$};
\draw[->] ($(B.west)+(0,.25cm)$) to [bend right=25] node[above,la]{$\Lsim$} ($(A.east)+(0,.25cm)$);
\draw[->] ($(A.east)-(0,.25cm)$) to [bend right=25] node[below,la]{$\bHsim$} ($(B.west)+(0,-.25cm)$);
\node[la] at ($(A.east)!0.5!(B.west)$) {$\bot$};
\end{tz}
from the model structure on $\DblCat$ of \cref{thm:secondMS}.
\end{theorem}

\begin{proof}
We need to show that a $2$-functor $F$ is a fibration (resp.\ biequivalence) in $\TwoCat$ if and only if $\bHsim F$ is a fibration (resp.\ weak equivalence) in $\DblCat$. Since $\bHsim$ is right Quillen, we know it preserves fibrations and trivial fibrations. Moreover, since all $2$-categories are fibrant, by Ken Brown's Lemma (see \cite[Lemma 1.1.12]{Hovey}) the functor $\bHsim$ preserves all weak equivalences. Therefore, if $F$ is a fibration (resp.\ biequivalence), then $\bHsim F$ is a fibration (resp.\ weak equivalence).

If $\bHsim F$ is a fibration in $\DblCat$, then by \cref{prop:F=Jinjbtwwhi} it is $\cJ_w$-injective, since its target is weakly horizontally invariant by \cref{rem:Hsimiswhi}. Hence, conditions (df1-2) of \cref{prop:fibinsecond} for $\bHsim F$ say that $F$ is a fibration in $\TwoCat$. 

Finally, if $\bHsim F$ is a weak equivalence in $\DblCat$, then by \cref{prop:Wwithwhisourceisdblbieq} it is a double biequivalence, since its source is weakly horizontally invariant. By (db1) and (db2) of \cref{def:doublebieq}, we have that $F$ is bi-essentially surjective on objects and essentially full on morphisms. Fully faithfulness on $2$-morphisms follows from applying (db4) of \cref{def:doublebieq} to squares with trivial vertical boundaries. Hence $F$ is a biequivalence.  
\end{proof}

Finally, we can use the above result to deduce that Lack's model structure on $\TwoCat$ is also left-induced from our model structure on $\DblCat$ along the horizontal embedding~$\bH$.

\begin{theorem}
Lack's model structure on $\TwoCat$ is left-induced along the adjunction
\begin{tz}
\node[](A) {$\DblCat$};
\node[right of=A,xshift=1cm](B) {$\TwoCat$};
\draw[->] ($(B.west)+(0,.25cm)$) to [bend right=25] node[above,la]{$\bH$} ($(A.east)+(0,.25cm)$);
\draw[->] ($(A.east)-(0,.25cm)$) to [bend right=25] node[below,la]{$\bfH$} ($(B.west)+(0,-.25cm)$);
\node[la] at ($(A.east)!0.5!(B.west)$) {$\bot$};
\end{tz}
from the model structure on $\DblCat$ of \cref{thm:secondMS}.
\end{theorem}

\begin{proof}
We need to show that a $2$-functor $F\colon \cA\to \cB$ is a cofibration (resp.\ biequivalence) in $\TwoCat$ if and only if $\bH F$ is a cofibration (resp.\ weak equivalence) in $\DblCat$. 

By \cref{thm:charcofDblsecond}, the double functor $\bH F$ is a cofibration if and only if its underlying functors $U\bfH \bH F$ and $U\bfV \bH F$ have the left lifting property with respect to all surjective on objects and full functors. Since $U\bfV \bH F$ trivially satisfies this condition, this holds if and only if $UF=U\bfH \bH F$ has the mentioned lifting property. By \cite[Lemma 4.1]{Lack2Cat}, this is equivalent to saying that $F$ is a cofibration.

Finally, since $\bHsim \cA$ and $\bHsim\cB$ are fibrant replacements of $\bH\cA$ and $\bH\cB$ in $\DblCat$ by \cref{rem:Hsimiswhi}, we have that $\bH F$ is a weak equivalence if and only if $\bHsim F$ is a weak equivalence. By \cref{thm:2CatriHsim}, this is the case if and only if $F$ is a biequivalence. 
\end{proof}

\section{Compatibility with the Gray tensor product} \label{sec:monoidality}

We now explore the monoidality of the model structure on $\DblCat$ constructed in this paper. Although a similar argument as the one in \cite[Remark 7.1]{MSV} quickly shows that our model structure is not monoidal with respect to the cartesian product, in this section we prove that it is monoidal when we instead consider the Gray tensor product for double categories introduced by B\"ohm in \cite[\S 3]{Bohm}. This resembles the case of Lack's model structure on $\TwoCat$, which is monoidal with respect to the Gray tensor product of 2-categories, and improves upon the model structure on $\DblCat$ of \cite{MSV}, which is only $\TwoCat$-enriched.

We first give a description of the Gray tensor product of two double categories. 

\begin{descr} \label{descr:Graytensordbl}
The Gray tensor product $\bA\otimes_\Gray \bX$ of two double categories $\bA$ and $\bX$ can be described as the double category given by the following data.
\begin{rome}
    \item The objects are pairs $(A,X)$ of objects $A\in \bA$ and $X\in \bX$.
    \item The generating horizontal morphisms are of two kinds: pairs $(a,X)\colon (A,X)\to (C,X)$, where $a\colon A\to C$ is a horizontal morphism in $\bA$ and $X$ is an object in $\bX$, which compose as in $\bA$; and pairs $(A,x)\colon (A,X)\to (A,Z)$, where $A$ is an object in $\bA$ and $x\colon X\to Z$ is a horizontal morphism in $\bX$, which compose as in $\bX$.
    \item Similarly, the generating vertical morphisms are given by pairs $(u,X)$ and $(A,t)$ with $A$ and $X$ being objects of $\bA$ and $\bX$ respectively, and $u$ and $t$ being vertical morphisms of $\bA$ and $\bX$ respectively. 
    \item There are six kinds of generating squares: the ones determined by a square $\sq{\alpha}{a}{a'}{u}{w}$ in~$\bA$ and an object $X\in \bX$ as shown below left, the ones given by an object $A\in \bA$ and a square $\sq{\chi}{x}{x'}{t}{v}$ in $\bX$ as below right,
    \begin{tz}
    \node[](1) {$(A,X)$}; 
    \node[right of=1,xshift=.5cm](2) {$(C,X)$}; 
    \node[below of=1](3) {$(A',X)$}; 
    \node[below of=2](4) {$(C',X)$};
    \draw[->,pro] (1) to node[left,la]{$(u,X)$} (3);
    \draw[->,pro] (2) to node[right,la]{$(w,X)$} (4);
    \draw[->] (1) to node[above,la]{$(a,X)$} (2);
    \draw[->] (3) to node[below,la]{$(a',X)$} (4);
    \node[la] at ($(1)!0.5!(4)$) {$(\alpha,X)$};
    
    \node[right of=2,xshift=1cm](1) {$(A,X)$}; 
    \node[right of=1,xshift=.5cm](2) {$(A,Z)$}; 
    \node[below of=1](3) {$(A,X')$}; 
    \node[below of=2](4) {$(A,Z')$};
    \draw[->,pro] (1) to node[left,la]{$(A,t)$} (3);
    \draw[->,pro] (2) to node[right,la]{$(A,v)$} (4);
    \draw[->] (1) to node[above,la]{$(A,x)$} (2);
    \draw[->] (3) to node[below,la]{$(A,x')$} (4);
    \node[la] at ($(1)!0.5!(4)$) {$(A,\chi)$};
    \end{tz}
    the squares determined by a horizontal morphism $a$ in $\bA$ and a vertical morphism~$t$ in~$\bX$ as displayed below left, and the ones given by a horizontal morphism $x$ in~$\bX$ and a vertical morphism $u$ in $\bA$ as below right,
    \begin{tz}
    \node[](1) {$(A,X)$}; 
    \node[right of=1,xshift=.5cm](2) {$(C,X)$}; 
    \node[below of=1](3) {$(A,X')$}; 
    \node[below of=2](4) {$(C,X')$};
    \draw[->,pro] (1) to node[left,la]{$(A,t)$} (3);
    \draw[->,pro] (2) to node[right,la]{$(C,t)$} (4);
    \draw[->] (1) to node[above,la]{$(a,X)$} (2);
    \draw[->] (3) to node[below,la]{$(a,X')$} (4);
    \node[la] at ($(1)!0.5!(4)$) {$(a,t)$};
    
    \node[right of=2,xshift=1cm](1) {$(A,X)$}; 
    \node[right of=1,xshift=.5cm](2) {$(A,Z)$}; 
    \node[below of=1](3) {$(A',X)$}; 
    \node[below of=2](4) {$(A',Z)$};
    \draw[->,pro] (1) to node[left,la]{$(u,X)$} (3);
    \draw[->,pro] (2) to node[right,la]{$(u,Z)$} (4);
    \draw[->] (1) to node[above,la]{$(A,x)$} (2);
    \draw[->] (3) to node[below,la]{$(A',x)$} (4);
    \node[la] at ($(1)!0.5!(4)$) {$(u,x)$};
    \end{tz}
vertically invertible squares determined by horizontal morphisms $a$ in $\bA$ and $x$ in~$\bX$, as shown below,
    \begin{tz}
    \node[](1) {$(A,X)$}; 
    \node[right of=1,xshift=.5cm](2) {$(C,X)$};
    \node[right of=2,xshift=.5cm](3) {$(C,Z)$}; 
    \node[below of=1](1') {$(A,X)$}; 
    \node[below of=2](2') {$(A,Z)$};
    \node[below of=3](3') {$(C,Z)$}; 
    \draw[d,pro] (1) to (1');
    \draw[d,pro] (3) to (3');
    \draw[->] (1) to node[above,la]{$(a,X)$} (2);
    \draw[->] (2) to node[above,la]{$(C,x)$} (3);
    \draw[->] (1') to node[below,la]{$(A,x)$} (2');
    \draw[->] (2') to node[below,la]{$(a,Z)$} (3');
    \node[la] at ($(1)!0.5!(3')+(7pt,0)$) {$\vcong$}; 
    \node[la] at ($(1)!0.5!(3')-(7pt,0)$) {$(a,x)$}; 
    \end{tz}
    and horizontally invertible squares given by vertical morphisms $u$ in $\bA$ and $t$ in~$\bX$, as below,
    \begin{tz}
    \node[](1) {$(A,X)$}; 
    \node[below of=1](2) {$(A',X)$};
    \node[below of=2](3) {$(A',X')$}; 
    \node[right of=1,xshift=.5cm](1') {$(A,X)$}; 
    \node[below of=1'](2') {$(A,X')$};
    \node[below of=2'](3') {$(A',X')$}; 
    \draw[d] (1) to (1');
    \draw[d] (3) to (3');
    \draw[->,pro] (1) to node[left,la]{$(u,X)$} (2);
    \draw[->,pro] (2) to node[left,la]{$(A',t)$} (3);
    \draw[->,pro] (1') to node[right,la]{$(A,t)$} (2');
    \draw[->,pro] (2') to node[right,la]{$(u,X')$} (3');
    \node[la] at ($(1)!0.5!(3')+(0,5pt)$) {$\cong$}; 
    \node[la] at ($(1)!0.5!(3')-(0,5pt)$) {$(u,t)$}; 
    \end{tz}
    subject to conditions which are equivalent to requiring that the projection double functor $\Pi_{\bA,\bX}\colon \bA\otimes_\Gray \bX\to \bA\times \bX$ is fully faithful on squares.
\end{rome}
\end{descr}

We can show that the projection double functor $\Pi_{\bA,\bX}\colon \bA\otimes_\Gray \bX\to \bA\times \bX$ is a trivial fibration in our model structure on $\DblCat$.

\begin{lemme} \label{lem:Piistrivfib}
The projection double functor $\Pi_{\bA,\bX}\colon \bA\otimes_\Gray \bX\to \bA\times \bX$ is a trivial fibration, for all double categories $\bA$ and $\bX$.
\end{lemme}

\begin{proof}
We use the characterization of trivial fibrations of \cref{prop:trivfibinsecond}. Since $\Pi_{\bA,\bX}$ is the identity on objects, it is clearly surjective on objects. Given a horizontal morphism $(a,x)\colon (A,X)\to (C,Z)$ in $\bA\times \bX$, the composite
\begin{tz}
\node[](1) {$(A,X)$}; 
\node[right of=1,xshift=.5cm](2) {$(C,X)$}; 
\node[right of=2,xshift=.5cm](3) {$(C,Z)$}; 
\draw[->] (1) to node[above,la]{$(a,X)$} (2);
\draw[->] (2) to node[above,la]{$(C,x)$} (3);
\end{tz}
of horizontal morphisms in $\bA\otimes_\Gray \bX$ is sent by $\Pi_{\bA,\bX}$ to $(a,x)$, which shows that $\Pi_{\bA,\bX}$ is full on horizontal morphisms. Similarly, one can show that $\Pi_{\bA,\bX}$ is full on vertical morphisms. Fully faithfulness on squares holds by \cref{descr:Graytensordbl} (iv). 
\end{proof}

We now show that $\bA^\whi\times \bX^\whi$ gives a fibrant replacement for $\bA\times\bX$, where $(-)^\whi$ is the weakly horizontally invariant replacement of \cref{def:whireplacement}.

\begin{lemme} \label{lem:fibreplproduct}
Let $\bA$ and $\bX$ be double categories. Then $j_\bA\times j_\bX\colon \bA\times \bX\to \bA^\whi\times \bX^\whi$ provides a fibrant replacement for the double category $\bA\times \bX$. 
\end{lemme}

\begin{proof}
First, note that $\bA^\whi\times \bX^\whi$ is fibrant, as fibrant objects are closed under products. Now consider the commutative triangle below, where the bottom map is induced by the projections $\pi_\bA\colon \bA\times \bX\to \bA$ and $\pi_\bX\colon \bA\times \bX\to \bX$.
\begin{tz}
\node[](1) {$\bA\times \bX$}; 
\node[below of=1,xshift=-2cm,yshift=2pt](2) {$(\bA\times \bX)^\whi$};
\node[below of=1,xshift=2cm,yshift=2pt](3) {$\bA^\whi \times \bX^\whi$};
\draw[->] (1) to node[pos=0.4,left,la,xshift=-4pt]{$j_{\bA\times\bX}$} (2); 
\draw[->] (1) to node[pos=0.4,right,la,xshift=4pt]{$j_\bA\times j_\bX$} (3); 
\draw[->] (2) to node[below,la]{$(\pi_\bA^\whi,\pi_\bX^\whi)$} (3);
\end{tz}
Since $j_{\bA\times \bX}$ is a weak equivalence by \cref{prop:j_AinCcapW}, to prove that $j_\bA\times j_\bX$ is a weak equivalence it suffices to show that $(\pi_\bA^\whi,\pi_\bX^\whi)$ is a weak equivalence; we use \cref{prop:trivfibinsecond} to prove that it is in fact a trivial fibration.

One can see that $(\pi_\bA^\whi,\pi_\bX^\whi)$ is the identity on underlying horizontal categories, and that it is fully faithful on squares since $j_{\bA\times \bX}$ and $j_\bA\times j_\bX$ are so. Finally, by studying the weakly horizontally invariant replacements, we can see that it is also full on vertical morphisms. Indeed, all the vertical morphisms that were freely added to $\bA^\whi \times \bX^\whi$ from the image of $\bA\times \bX$ were also freely added to $(\bA \times \bX)^\whi$ from the image of $\bA\times \bX$. 
\end{proof}

Mirroring the proof in \cite[\S 7]{Lack2Cat}, we show that the cartesian product and the Gray tensor product of a weak equivalence with an identity is also a weak equivalence. 

\begin{rem} \label{lem:FGrayiddoublebieq}
Given a double biequivalence $F\colon \bA\to \bB$ and a double category $\bX$, the product $F\times \id_\bX\colon \bA\times\bX\to \bB\times \bX$ is a double biequivalence. Indeed, it is straightforward to see that (db1-4) of \cref{def:doublebieq} hold for $F\times \id_\bX$ since they do for $F$.
\end{rem}

\begin{prop} \label{cor:FGrayidW}
Let $F\colon \bA\to \bB$ be a weak equivalence in the model structure on $\DblCat$ of \cref{thm:secondMS}. Then, for every double category $\bX$, the induced double functors 
\[ F\times \id_\bX\colon \bA\times\bX\to \bB\times \bX \ \ \text{and} \ \ F\otimes_\Gray \id_\bX\colon \bA\otimes_\Gray \bX\to \bB\otimes_\Gray \bX \]
are also weak equivalences in $\DblCat$. 
\end{prop}

\begin{proof}
First note that the  weakly horizontally invariant replacement $F^\whi$ is a double biequivalence, since $F$ is a weak equivalence. Hence, by \cref{lem:FGrayiddoublebieq}, the double functor $F^\whi\times \id_{\bX^\whi}\colon \bA^\whi\times \bX^\whi\to \bB^\whi\times \bX^\whi$ is also a double biequivalence. Since $\bA^\whi\times \bX^\whi$ and $\bB^\whi\times \bX^\whi$ are fibrant replacements for $\bA\times \bX$ and $\bB\times \bX$ by \cref{lem:fibreplproduct}, this shows that $F\times \id_\bX$ is a weak equivalence by 2-out-of-3. 

For the statement regarding the Gray tensor product, we know by \cref{lem:Piistrivfib} that the double functors $\Pi_{\bA,\bX}$ and $\Pi_{\bB,\bX}$ are trivial fibrations. Since the diagram
\begin{tz}
\node[](1) {$\bA\otimes_\Gray\bX$}; 
\node[right of=1,xshift=1.5cm](2) {$\bB\otimes_\Gray\bX$}; 
\node[below of=1](3) {$\bA\times \bX$};
\node[below of=2](4) {$\bB\times \bX$}; 
\draw[->] (1) to node[above,la]{$F\otimes_\Gray\id_\bX$}(2);
\draw[->] (1) to node[left,la]{$\Pi_{\bA,\bX}$} (3); 
\draw[->] (2) to node[right,la]{$\Pi_{\bB,\bX}$} (4);
\draw[->] (3) to node[below,la]{$F\times \id_\bX$} (4);
\end{tz}
commutes, then $F\otimes_\Gray\id_\bX$ is also a weak equivalence by $2$-out-of-$3$. 
\end{proof}

This allows us to prove that our model structure on $\DblCat$ is monoidal with respect to the Gray tensor product, inspired by the proof of the monoidality of Lack's model structure on $\TwoCat$ of \cite[\S 7]{Lack2Cat}.

\begin{notation}
Given two double functors $I\colon \bA\to \bB$ and $J\colon \bC\to \bD$, we write $\push{I}{J}{}$ for the pushout-product double functor
\[ \pushprod{I}{\bA}{\bB}{J}{\bC}{\bD}{\Gray}. \]
\end{notation}

\begin{theorem} \label{thm:DblCat2monoidal}
The model structure on $\DblCat$ of \cref{thm:secondMS} is monoidal with respect to the Gray tensor product $\otimes_\Gray$.
\end{theorem}

\begin{proof} 
We begin by showing that whenever $I$ and $J$ are cofibrations, the pushout-product $\push{I}{J}{}$ is also a cofibration; it is enough to consider the case when $I$ and $J$ are in the set of generating cofibrations $\cI_w=\{I_1,I_2,I_3,I_4,I_5\}$ of \cref{not:gencofDblCat2}. Moreover, since the Gray tensor product is symmetric, if we show the result for $\push{I}{J}{}$, then it also holds for $\push{J}{I}{}$. Note that $\push{I_1}{J}{}\cong J$, which proves the cases involving $I_1$.

To show the cases involving $I_4$ or $I_5$, we observe the following three facts: the functors $U\bfH,U\bfV\colon\DblCat\to\Cat$ preserve pushouts; $U\bfH(I_4), U\bfH(I_5), U\bfV(I_4),$ and $U\bfV(I_5)$ are identities; and $U\bfH(\bA\otimes_\Gray\bB)$ (resp. $U\bfV(\bA\otimes_\Gray\bB)$) only depends on $U\bfH(\bA)$ and $U\bfH(\bB)$ (resp. $U\bfV(\bA)$ and $U\bfV(\bB)$). It then follows that $U\bfH(\push{I}{J}{})$ and $U\bfV(\push{I}{J}{})$ are isomorphisms, and thus $\push{I}{J}{}$ is a cofibration by \cref{thm:charcofDblsecond}, if either $I$ or $J$ is in $\{I_4,I_5\}$.

For the remaining cases, one can check that $\push{I_2}{I_2}{}$ is given by the boundary inclusion $\delta(\bH \mathbbm{2}\otimes_\Gray \bH\mathbbm{2})\to \bH \mathbbm{2}\otimes_\Gray \bH\mathbbm{2}$, which is a cofibration by \cref{thm:charcofDblsecond}, since it is the identity on underlying horizontal and vertical categories. Similarly, one can show that the pushout-products $\push{I_3}{I_3}{}$ and $\push{I_2}{I_3}{}$ are cofibrations, as they are given by the boundary inclusions $\delta(\bH \mathbbm{2}\otimes_\Gray \bV\mathbbm{2})\to \bH \mathbbm{2}\otimes_\Gray \bV\mathbbm{2}$ and $\delta(\bV \mathbbm{2}\otimes_\Gray \bV\mathbbm{2})\to \bV \mathbbm{2}\otimes_\Gray \bV\mathbbm{2}$, respectively.

It remains to show that if $I\in\cI_w$ and $J\colon\bA\to\bB$ is a trivial cofibration, then $\push{I}{J}{}$ is a weak equivalence. Note that $I$ is of the form $I\colon \bC\to \bD$ with $\bC$ cofibrant, and consider the pushout diagram below.
\begin{tz}
\node[](1) {$\bC\otimes_\Gray\bA$};
\node[right of=1,xshift=1.5cm](2) {$\bD\otimes_\Gray\bB$};
\node[below of=1](3) {$\bD\otimes_\Gray\bA$};
\node[below of=2](4) {$\bP$};
\node[below right of=4,xshift=1cm](5) {$\bD\otimes_\Gray\bB$};
\draw[right hook->] (1) to node[above,la]{$\id_\bC\otimes_\Gray J$} node[below,la]{$\sim$} (2);
\draw[->] (1) to node[left,la]{$I\otimes_\Gray\id_\bA$} (3);
\draw[->] (2) to (4);
\draw[right hook->] (3) to node[above,la]{$K$} node[below,la]{$\sim$} (4);
\draw[->,dashed] (4) to node[right,la,yshift=10pt,xshift=-15pt]{$\push{I}{J}{}$} (5);
\draw[->,bend right] (3) to node[below,la,xshift=-10pt] {$\id_\bD\otimes_\Gray J$} node[above,la,sloped,yshift=-2pt] {$\sim$} (5);
\draw[->,bend left] (2) to node[right,la] {$I\otimes_\Gray \id_\bB$} (5);
  \node at ($(4)+(-.3cm,.25cm)$) {$\ulcorner$};
\end{tz}
Since $\bC$ is cofibrant and $J$ is a cofibration, we know that $\push{(\emptyset\to \bC)}{J}{}=\id_\bC\otimes_\Gray J$ is also a cofibration by the above. Since $J$ is a trivial cofibration by assumption, the double functor $\id_\bC\otimes_\Gray J$ is a weak equivalence by \cref{cor:FGrayidW}. Then $\id_\bC\otimes_\Gray J$ is a trivial cofibration, and therefore so is $K$ since these are stable under pushouts. \cref{cor:FGrayidW} also guarantees that $\id_\bD\otimes_\Gray J$ is a weak equivalence, and then so is $\push{I}{J}{}$ by $2$-out-of-$3$.
\end{proof}

\begin{rem} \label{rem:2catenrichedsecond}
Recall that, by restricting the Gray tensor product $\otimes_\Gray$ in one variable along $\bH\colon \TwoCat\to \DblCat$, we get the tensoring functor $\otimes\colon \DblCat\times \TwoCat\to \DblCat$ which gives an enrichment $\bfH[-,-]_\ps$ of $\DblCat$ over $\TwoCat$ as in \cite[Proposition 7.6]{MSV}. Since the functor $\bH$ is left Quillen by \cref{thm:HHcorefsecond}, as a corollary of \cref{thm:DblCat2monoidal} we get that the model structure on $\DblCat$ of \cref{thm:secondMS} is also $\TwoCat$-enriched.
\end{rem}

\section{Whitehead theorem}\label{section:whitehead}

In this section we show a Whitehead Theorem for double categories, that characterizes the weak equivalences between fibrant objects (which, by \cref{prop:Wwithwhisourceisdblbieq}, are double biequivalences) as the double functors that admit a pseudo inverse up to horizontal pseudo natural equivalence. Such a statement is reminiscent of the Whitehead Theorem for $2$-categories: a $2$-functor $F\colon\A\to\B$ is a biequivalence if and only if there is a pseudo functor $G\colon\B\to\A$ together with two pseudo natural equivalences $\id_\A\simeq GF$ and $FG\simeq \id_\B$.

Under the hypothesis that the double categories involved are \emph{horizontally invariant} ---defined analogously to the weakly horizontally invariant double categories with horizontal equivalences replaced by the stronger notion of horizontal isomorphisms; see \cite[Theorem and Definition 4.1.7]{Grandis}---, Grandis characterizes in \cite[Theorem 4.4.5]{Grandis} the double functors $F$ such that $U\bfH F$ and $U\bfH[\vtwo,F]$ are both equivalences of categories as the ones which admit a pseudo inverse up to \emph{horizontal natural isomorphism}. In analogy, double biequivalences can be defined as the double functors such that $\bfH F$ and $\bfH[\vtwo,F]$ are biequivalences of $2$-categories (see \cite[Proposition 3.12]{MSV}). Altogether our Whitehead Theorem can be seen as a $2$-categorical version of Grandis' result.

In the theorem below, whose proof is the content of this section, it is actually enough to require that the \emph{source} be weakly horizontally invariant. 

\begin{theorem}[Whitehead Theorem] \label{thm:whitehead}
Let $\bA$ and $\bB$ be double categories such that $\bA$ is weakly horizontally invariant. Then a double functor $F\colon \bA\to \bB$ is a weak equivalence (or equivalently, a double biequivalence) if and only if there is a pseudo double functor $G\colon \bB\to \bA$ together with horizontal pseudo natural equivalences $\id_\bA\simeq GF$ and $FG\simeq \id_\bB$.
\end{theorem}

\begin{rem}
This retrieves a formulation of the usual  Whitehead Theorem for model categories (see \cite[Lemma 4.24]{DS}) in our setting; such a result characterizes the weak equivalences between cofibrant-fibrant objects in a model structure as the homotopy equivalences.
\end{rem}

Let us now introduce what we mean by a pseudo double functor. 

\begin{defn} \label{defn:pseudodouble}
A \textbf{pseudo double functor} $G\colon \bB\to \bA$ consists of maps on objects, horizontal morphisms, vertical morphisms, and squares, compatible with sources and targets, which preserve
\begin{rome}
    \item horizontal compositions and identities up to coherent  vertically invertible squares
\begin{tz}
\node (A) at (0,0) {$GB$};
\node (B) at (1.5,0) {$GC$};
\node (C) at (3,0) {$GD$};
\node (A') at (0,-1.5) {$GB$};
\node (C') at (3,-1.5) {$GD$};
\draw[->] (A) to node[above,scale=0.8] {$Gb$} (B);
\draw[->] (B) to node[above,scale=0.8] {$Gd$} (C);
\draw[->] (A') to node[below,scale=0.8] {$G(db)$} (C');
\draw[d] (A) to (A');
\draw[d] (C) to (C'); 

\node at (0,-.75) {$\bullet$};
\node at (3,-.75) {$\bullet$};

\node[scale=0.8] at (1.3,-.75) {$\Phi_{b,d}$};
\node[scale=0.8] at (1.7,-.75) {$\vcong$};

\node (A) at (5,0) {$GB$};
\node (B) at (6.5,0) {$GB$};
\node (A') at (5,-1.5) {$GB$};
\node (B') at (6.5,-1.5) {$GB$};

\draw[d] (A) to (B);
\draw[d] (A) to (A');
\draw[d] (B) to (B');
\draw[->] (A') to node[below,scale=0.8] {$G\id_B$} (B');

\node at (5,-.75) {$\bullet$};
\node at (6.5,-.75) {$\bullet$};

\node[scale=0.8] at (5.6,-.75) {$\Phi_{B}$};
\node[scale=0.8] at (6,-.75) {$\vcong$};
\end{tz}
    for every object $B\in \bB$, and every pair of composable horizontal morphisms $b\colon B\to C$ and $d\colon C\to D$ in~$\bB$,
    \item vertical compositions and identities up to coherent horizontally invertible squares $\Psi_{v,v'}$ and $\Psi_B$ ---the transposed versions of those in (i)---, for every object $B\in \bB$, and every pair of composable vertical morphisms $v$, $v'$ in~$\bB$, 
    \item horizontal and vertical compositions of squares accordingly. 
\end{rome}

For a detailed description of the coherences, the reader can see \cite[Definition 3.5.1]{Grandis}.

The pseudo double functor $G$ is said to be \textbf{normal} if the squares $\Phi_B$ and $\Psi_B$ are identities for every object $B\in \bB$.
\end{defn}

\begin{rem}
There are notions of horizontal pseudo natural transformations between pseudo double functors, and modifications between them (with trivial vertical boundaries). These are defined analogously to \cite[\S 3.8]{Grandis}.
\end{rem}

The notion that we now introduce has also been independently considered by Grandis and Par\'e in \cite[\S 3]{GraPar19} under the name of \emph{pointwise equivalences}.

\begin{defn} 
Let $F,G\colon \bA\to \bB$ be pseudo double functors. A \textbf{horizontal pseudo natural equivalence} $\varphi\colon F\Rightarrow G$ is an equivalence in the $2$-category of pseudo double functors $\bA\to \bB$, horizontal pseudo natural transformations, and modifications with trivial vertical boundaries.
\end{defn}

Equivalently, the horizontal pseudo natural equivalences can be described as follows; see \cite[Theorem 4.4]{GraPar19} for a proof.

\begin{lemme}
Let $F,G\colon \bA\to \bB$ be pseudo double functors. A horizontal pseudo natural transformation $\varphi\colon F\Rightarrow G$ is a horizontal pseudo natural equivalence if and only if 
\begin{rome}
  \item the horizontal morphism $\varphi_A\colon FA\xrightarrow{\simeq} GA$ is a horizontal equivalence, for every object $A\in \bA$, and
  \item the square $\sq{\varphi_u}{\varphi_A}{\varphi_{A'}}{Fu}{Gu}$ is weakly horizontally invertible, for every vertical morphism $u\colon A\arrowdot A'$ in $\bA$.
\end{rome}
\end{lemme}

We will use the term \emph{horizontal biequivalence} to refer to the double functors which admit a pseudo inverse up to horizontal pseudo natural equivalence.

\begin{defn}\label{def:horbieq}
A double functor $F\colon \bA\to \bB$ is a \textbf{horizontal biequivalence} if there is a pseudo double functor $G\colon \bB\to \bA$ together with horizontal pseudo natural equivalences ${\eta\colon \id_\bA\Rightarrow GF}$ and $\epsilon\colon FG\Rightarrow \id_\bB$.
\end{defn}

\begin{rem} \label{horbieq:otherdata}
If $F\colon \bA\to \bB$ is a horizontal biequivalence, then the tuple $(G,\eta,\epsilon)$ can always be promoted to the following data:
\begin{rome}
    \item a \emph{normal} pseudo double functor $G\colon \bB\to \bA$, 
    \item a horizontal pseudo natural \emph{adjoint} equivalence \[ (\eta\colon \id_\bA\Rightarrow GF, \  \eta'\colon GF\Rightarrow \id_\bA, \  \lambda\colon \id\cong \eta'\eta, \ \kappa\colon \eta\eta'\cong \id), \] 
    \item a horizontal pseudo natural \emph{adjoint} equivalence \[ (\epsilon\colon FG\Rightarrow \id_\bB, \ \epsilon'\colon \id_\bB\Rightarrow FG,  \ \mu\colon \id\cong \epsilon'\epsilon,  \ \nu\colon \epsilon\epsilon'\cong \id),\] 
    \item two invertible modifications $\Theta\colon \id_F\cong \epsilon_F\circ F\eta $ and $\Sigma\colon \id_G\cong  G\epsilon\circ \eta_G$, expressing the triangle (pseudo-)identities for $\eta$ and $\epsilon$. 
\end{rome}
This follows from the fact that a pseudo double functor can always be promoted to a normal one, and from a result by Gurski \cite[Theorem 3.2]{Gurski} saying that a biequivalence can always be promoted to a biadjoint biequivalence.
\end{rem}

\cref{thm:whitehead} now amounts to showing that a double functor whose source is weakly horizontally invariant is a double biequivalence if and only if it is a horizontal biequivalence. However, it is always true that a horizontal biequivalence is a double biequivalence; no additional hypothesis is needed here. In order to prove this first result, we need the following lemma.

\begin{lemme}\label{lem:epsilon_iso_to_eta}
The data of \cref{horbieq:otherdata} induces an invertible modification $\theta\colon F\eta'\cong \epsilon_F$.
\end{lemme}

\begin{proof}
Given an object $A\in \bA$, we define the component of $\theta$ at $A$ to be the vertically invertible square
\begin{tz}
\node (A) at (-3.5,-.75) {$FGFA$};
\node (B) at (-2,-.75) {$FA$};
\node (C) at (-3.5,-2.25) {$FGFA$};
\node (D) at (-2,-2.25) {$FA$};

\draw[->] (A) to node[above, scale=0.8] {$F\eta'_A$} (B);
\draw[->] (C) to node[below, scale=0.8] {$\epsilon_{FA}$} (D);
\draw[d] (A) to (C);
\draw[d] (B) to (D);

\node at (-3.5,-1.5) {$\bullet$};
\node (f) at (-2,-1.5) {$\bullet$};
\node[scale=0.8] at (-2.9,-1.5) {$\theta_A$};
\node[scale=0.8] at (-2.6,-1.5) {$\vcong$};

\node (A) at (0,0) {$FGFA$};
\node (B) at (1.5,0) {$FA$};
\node (C) at (4.5,0) {$FA$};
\node (D) at (0,-1.5) {$FGFA$};
\node (E) at (1.5,-1.5) {$FA$};
\node (F) at (3,-1.5) {$FGFA$};
\node (G) at (4.5,-1.5) {$FA$};
\node (H) at (0,-3) {$FGFA$};
\node (I) at (3,-3) {$FGFA$};
\node (J) at (4.5,-3) {$FA$};
\node at ($(J.east)-(0,4pt)$) {.};
\node at ($(f)!0.5!(D.west)$) {$=$};

\draw[d] (B) to (C);
\draw[->] (A) to node[above, scale=0.8] {$F\eta'_A$} (B);
\draw[->] (D) to node[below, scale=0.8] {$F\eta'_A$} (E);
\draw[->] (E) to node[above, scale=0.8] {$F\eta_A$} (F);
\draw[->] (F) to node[above, scale=0.8] {$\epsilon_{FA}$} (G);
\draw[->] (I) to node[below, scale=0.8,yshift=-1pt] {$\epsilon_{FA}$} (J);
\draw[d] (H) to (I);
\draw[d] (A) to (D);
\draw[d] (D) to (H);
\draw[d] (B) to (E);
\draw[d] (C) to (G);
\draw[d] (F) to (I);
\draw[d] (G) to (J);

\node at (0,-.75) {$\bullet$};
\node at (0,-2.25) {$\bullet$};
\node at (1.5,-.75) {$\bullet$};
\node at (4.5,-.75) {$\bullet$};
\node at (3,-2.25) {$\bullet$};
\node at (4.5,-2.25) {$\bullet$};

\node[scale=0.8] at (1.3,-2.25) {$F\kappa_A$};
\node[scale=0.8] at (1.8,-2.25) {$\vcong$};
\node[scale=0.8] at (3.75,-2.25) {$e_{\epsilon_{FA}}$};
\node[scale=0.8] at (.75,-.75) {$e_{F\eta'_A}$};
\node[scale=0.8] at (2.8,-.75) {$\Theta_A$};
\node[scale=0.8] at (3.2,-.75) {$\vcong$};
\end{tz}
The proof of horizontal and vertical coherences for $\theta$ is a standard check that stems from the constructions of the squares $\theta_A$ and from the horizontal and vertical coherences of the modifications $F\kappa\colon (F\eta)(F\eta')\cong \id$ and $\Theta\colon \id\cong \epsilon_F \circ F\eta$. 
\end{proof}

\begin{prop}\label{prop:horbieq_is_doublebieq}
If $F\colon\bA\to\bB$ is a horizontal biequivalence, then $F$ is a double biequivalence.
\end{prop}

\begin{proof}
We check that $F$ satisfies (db1-4) of \cref{def:doublebieq}. Let $(F,G,\eta,\epsilon)$ be the data of a horizontal adjoint biequivalence as in \cref{horbieq:otherdata}. We first show (db1). For every object $B\in\bB$, we want to find an object $A\in\bA$ and a horizontal equivalence $B\xrightarrow{\simeq} FA$ in $\bB$. Setting $A=GB$, we have that $\epsilon'_B\colon B\xrightarrow{\simeq}FGB=FA$ gives such a horizontal equivalence.
    
We now show (db2). Let -$A$, $C$ be objects in $\bA$, and $b\colon FA\to FC$ be a horizontal morphism in $\bB$. We want to find a horizontal morphism $a\colon A\to C$ in $\bA$ and a vertically invertible square $({e_{FA}} \; ^{{b}}_{\substack{{Fa}}} \; {e_{FC}})$ in $\bB$. Let $a\colon A\to C$ be the composite  
\[ A\xrightarrow{\eta_A} GFA\xrightarrow{Gb}GFC\xrightarrow{\eta'_C}C; \]
we then have a vertically invertible square as desired,
\begin{tz}
\node (A) at (0,0) {$FA$};
\node (B) at (3.6,0) {$FA$};
\node (C) at (5.4,0) {$FC$};
\node (D) at (0,-1.5) {$FA$};
\node (E) at (1.8,-1.5) {$FGFA$};
\node (F) at (3.6,-1.5) {$FA$};
\node (G) at (5.4,-1.5) {$FC$};
\node (H) at (0,-3) {$FA$};
\node (I) at (1.8,-3) {$FGFA$};
\node (J) at (3.6,-3) {$FGFC$};
\node (K) at (5.4,-3) {$FC$};
\node (L) at (0,-4.5) {$FA$};
\node (M) at (1.8,-4.5) {$FGFA$};
\node (N) at (3.6,-4.5) {$FGFC$};
\node (O) at (5.4,-4.5) {$FC$};
\draw[d] (A) to (B);
\draw[->] (B) to node[above, scale=0.8] {$b$} (C);
\draw[->] (D) to node[above, scale=0.8] {$F\eta_A$} (E);
\draw[->] (E) to node[above, scale=0.8] {$\epsilon_{FA}$} (F);
\draw[->] (F) to node[above, scale=0.8] {$b$} (G);
\draw[->] (H) to node[above, scale=0.8] {$F\eta_A$} (I);
\draw[->] (I) to node[above, scale=0.8,yshift=2pt] {$FGb$} (J);
\draw[->] (J) to node[above, scale=0.8] {$\epsilon_{FC}$} (K);
\draw[->] (L) to node[below, scale=0.8] {$F\eta_A$} (M);
\draw[->] (M) to node[below, scale=0.8,yshift=-2pt] {$FGb$} (N);
\draw[->] (N) to node[below, scale=0.8] {$F\eta'_C$} (O);
\draw[d] (A) to (D);
\draw[d] (B) to (F);
\draw[d] (C) to (G);
\draw[d] (D) to (H);
\draw[d] (E) to (I);
\draw[d] (G) to (K);
\draw[d] (H) to (L);
\draw[d] (J) to (N);
\draw[d] (K) to (O);

\node at (0,-.75) {$\bullet$};
\node at (3.6, -.75) {$\bullet$};
\node at (5.4, -.75) {$\bullet$};
\node at (0,-2.25) {$\bullet$};
\node at (1.8,-2.25) {$\bullet$};
\node at (5.4,-2.25) {$\bullet$};
\node at (0,-3.75) {$\bullet$};
\node at (3.6,-3.75) {$\bullet$};
\node at (5.4,-3.75) {$\bullet$};

\node[scale=0.8] at (1.6,-.75) {$\Theta_A$};
\node[scale=0.8] at (2,-.75) {$\vcong$};
\node[scale=0.8] at (4.5,-.75) {$e_b$};
\node[scale=0.8] at (.9,-2.25) {$e_{F\eta_A}$};
\node[scale=0.8] at (3.4,-2.25) {$\epsilon_b$};
\node[scale=0.8] at (3.8,-2.25) {$\vcong$};
\node[scale=0.8] at (1.8,-3.75) {$e_{(FGb)(F\eta_A)}$};
\node[scale=0.8] at (4.4,-3.75) {$\theta^{-1}_C$};
\node[scale=0.8] at (4.8,-3.75) {$\vcong$};
\end{tz}
where $\theta_C$ is the component at $C$ of the invertible modification $\theta$ of \cref{lem:epsilon_iso_to_eta}.

We now show (db3). Let $v\colon B\arrowdot B'$ be a vertical morphism in $\bB$. We want to find a vertical morphism $u\colon A\arrowdot A'$ in $\bA$ and a weakly horizontally invertible square $({v} \; ^{{\simeq}}_{\substack{{\simeq}}} \; {Fu})$ in $\bB$. Let $u\colon A\arrowdot A'$ be the vertical morphism $Gv\colon GB\arrowdot GB'$. Then $\epsilon'_v$ gives the desired weakly horizontally invertible square.
\begin{tz}
\node (A) at (0,0) {$B$};
\node (B) at (1.5,0) {$FGB$};
\node (A') at (0,-1.5) {$B'$};
\node (B') at (1.5,-1.5) {$FGB'$};
\draw[->] (A) to node[below, scale=0.8] {$\simeq$} node[above, scale=0.8] {$\epsilon'_B$} (B);
\draw[->] (A') to node[above, scale=0.8] {$\simeq$} node[below, scale=0.8] {$\epsilon'_{B'}$} (B');
\draw[->] (A) to node[left,scale=0.8] {$v\;$} (A');
\draw[->] (B) to node[right,scale=0.8] {$\;FGv$} (B');

\node at (0,-.75) {$\bullet$};
\node at (1.5, -.75) {$\bullet$};

\node[scale=0.8] at (.6,-.75) {$\epsilon'_v$};
\node[scale=0.8] at (.9,-.75) {$\simeq$};
\end{tz}

Finally, let $\sq{\beta}{Fa}{Fc}{Fu}{Fu'}$ be a square in $\bB$ as in (db4). We want to show that there is a unique square $\sq{\alpha}{a}{c}{u}{u'}$ in $\bA$ such that $F\alpha=\beta$. Define $\alpha$ to be the square given by the following pasting.
\begin{tz}
\node (A) at (-3.5,-3) {$A$};
\node (B) at (-2,-3) {$C$};
\node (A') at (-3.5,-4.5) {$A'$};
\node (B') at (-2,-4.5) {$C'$};
\draw[->] (A) to node[above, scale=0.8] {$a$} (B);
\draw[->] (A') to node[below, scale=0.8] {$c$} (B');
\draw[->] (A) to node[left,scale=0.8] {$u\;$} (A');
\draw[->] (B) to node[right,scale=0.8] {$\;u'$} (B');

\node at (-3.5,-3.75) {$\bullet$};
\node at (-2, -3.75) {$\bullet$};

\node[scale=0.8] at (-2.75,-3.75) {$\alpha$};

\node at (-1,-3.75) {$=$};

\node (A) at (0,0) {$A$};
\node (C) at (3,0) {$A$};
\node (D) at (4.5,0) {$C$};
\draw[d] (A) to (C);
\draw[->] (C) to node[above,scale=0.8] {$a$} (D);

\node (A') at (0,-1.5) {$A$};
\node (B') at (1.5,-1.5) {$GFA$};
\node (C') at (3,-1.5) {$A$};
\node (D') at (4.5,-1.5) {$C$};
\draw[->] (A') to node[above,scale=0.8] {$\eta_A$} (B');
\draw[->] (B') to node[above,scale=0.8] {$\eta'_A$} (C');
\draw[->] (C') to node[above,scale=0.8] {$a$} (D');
\draw[d] (A) to (A');
\draw[d] (C) to (C');
\draw[d] (D) to (D');

\node at (0,-.75) {$\bullet$};
\node at (3,-.75) {$\bullet$};
\node at (4.5,-.75) {$\bullet$};

\node[scale=0.8] at (1.3,-.75) {$\mu_A$};
\node[scale=0.8] at (1.7,-.75) {$\vcong$};
\node[scale=0.8] at (3.75,-.75) {$e_a$};

\node (A) at (0,-3) {$A$};
\node (B) at (1.5,-3) {$GFA$};
\node (C) at (3,-3) {$GFC$};
\node (D) at (4.5,-3) {$C$};
\draw[->] (A) to node[above,scale=0.8] {$\eta_A$} (B);
\draw[->] (B) to node[above,scale=0.8,yshift=1pt] {$GFa$} (C);
\draw[->] (C) to node[above,scale=0.8] {$\eta'_C$} (D);
\draw[d] (A') to (A);
\draw[d] (B') to (B);
\draw[d] (D') to (D);

\node at (0,-2.25) {$\bullet$};
\node at (1.5,-2.25) {$\bullet$};
\node at (4.5,-2.25) {$\bullet$};

\node[scale=0.8] at (2.8,-2.25) {$\eta'_a$};
\node[scale=0.8] at (3.2,-2.25) {$\vcong$};
\node[scale=0.8] at (.75,-2.25) {$e_{\eta_A}$};

\node (A') at (0,-4.5) {$A'$};
\node (B') at (1.5,-4.5) {$GFA'$};
\node (C') at (3,-4.5) {$GFC'$};
\node (D') at (4.5,-4.5) {$C'$};
\draw[->] (A') to node[below,scale=0.8] {$\eta_{A'}$} (B');
\draw[->] (B') to node[below,scale=0.8,yshift=-1pt] {$GFc$} (C');
\draw[->] (C') to node[below,scale=0.8] {$\eta'_{C'}$} (D');
\draw[->] (A) to node[left,scale=0.8] {$u\;$} (A');
\draw[->] (B) to node[over,scale=0.8,yshift=.3cm] {$GFu$} (B');
\draw[->] (C) to node[over,scale=0.8,yshift=.3cm] {$GFu'$} (C');
\draw[->] (D) to node[right,scale=0.8] {$\;u'$} (D');

\node at (0,-3.75) {$\bullet$};
\node at (1.5,-3.85) {$\bullet$};
\node at (3,-3.85) {$\bullet$};
\node at (4.5,-3.75) {$\bullet$};

\node[scale=0.8] at (.75,-3.75) {$\eta_u$};
\node[scale=0.8] at (2.25,-3.75) {$G\beta$};
\node[scale=0.8] at (3.75,-3.75) {$\eta'_{u'}$};

\node (A) at (0,-6) {$A'$};
\node (B) at (1.5,-6) {$GFA'$};
\node (C) at (3,-6) {$A'$};
\node (D) at (4.5,-6) {$C'$};
\draw[->] (A) to node[below,scale=0.8] {$\eta_{A'}$} (B);
\draw[->] (B) to node[below,scale=0.8,yshift=1pt] {$\eta'_{A'}$} (C);
\draw[->] (C) to node[below,scale=0.8] {$c$} (D);
\draw[d] (A') to (A);
\draw[d] (B') to (B);
\draw[d] (D') to (D);

\node at (0,-5.25) {$\bullet$};
\node at (1.5,-5.25) {$\bullet$};
\node at (4.5,-5.25) {$\bullet$};

\node[scale=0.8] at (2.8,-5.25) {${\eta'_c}^{-1}$};
\node[scale=0.8] at (3.2,-5.25) {$\vcong$};
\node[scale=0.8] at (.75,-5.25) {$e_{\eta_{A'}}$};

\node (A') at (0,-7.5) {$A'$};
\node (C') at (3,-7.5) {$A'$};
\node (D') at (4.5,-7.5) {$C'$};
\draw[d] (A') to (C');
\draw[->] (C') to node[below,scale=0.8] {$c$} (D');
\draw[d] (A) to (A');
\draw[d] (C) to (C');
\draw[d] (D) to (D');

\node at (0,-6.75) {$\bullet$};
\node at (3,-6.75) {$\bullet$};
\node at (4.5,-6.75) {$\bullet$};

\node[scale=0.8] at (1.3,-6.75) {$\mu_{A'}^{-1}$};
\node[scale=0.8] at (1.7,-6.75) {$\vcong$};
\node[scale=0.8] at (3.75,-6.75) {$e_c$};
\end{tz}
The thorough reader might check that $F\alpha=\beta$ by completing the following steps. First transform $F\eta'_{u'}$ by using the invertible modification $\theta\colon F\eta'\cong \epsilon_F$ of \cref{lem:epsilon_iso_to_eta}; then apply, in the given order: the horizontal coherence of the modification $F\nu\colon (F\eta')(F\eta)\cong\id$, the horizontal coherence of the modification $\Theta\colon \id\cong \epsilon_F \circ F\eta$, the triangle identity for $(\mu,\nu)$, the compatibility of $\epsilon_F\colon FGF\Rightarrow F$ with $FG\beta$ and $\beta$, and finally the horizontal coherence of the modification $\Theta\colon \id\cong \epsilon_F \circ F\eta$.

Suppose now that $\sq{\alpha'}{a}{c}{u}{u'}$ is another square in $\bA$ such that $F\alpha'=\beta$. If we replace~$G\beta$ with $GF\alpha'$ in the pasting diagram above, then it follows from the compatibility of $\eta'\colon GF\Rightarrow \id_\bA$ with $GF\alpha'$ and $\alpha'$, and the vertical coherence of the modification $\mu\colon \id\cong \eta'\eta$, that this pasting is also equal to $\alpha'$. Therefore, we must have $\alpha=\alpha'$. This completes the proof of (db4).
\end{proof}

It is not true in general that a double biequivalence is a horizontal biequivalence, unless we impose an additional condition on the source or on the target. In \cite[Theorem 5.15]{MSV} we provide a Whitehead Theorem, where the target satisfies a condition related to cofibrancy in the model structure of \cite{MSV}. Here, we prove that such a result holds when the source of the double biequivalence is fibrant, which completes the proof of our Whitehead theorem, \cref{thm:whitehead}.

\begin{prop}\label{prop:doublebieq_is_horbieq}
    Let $F\colon \bA\to \bB$ be a double biequivalence such that $\bA$ is weakly horizontally invariant. Then $F$ is a horizontal biequivalence.
\end{prop}

\begin{proof}
We simultaneously define the pseudo double functor $G\colon \bB\to \bA$ and the horizontal pseudo natural transformation $\epsilon\colon FG\Rightarrow \id_\bB$.

{\bf $G$ and $\epsilon$ on objects.} Let $B\in \bB$ be an object. By (db1) of \cref{def:doublebieq}, there is an object $A\in \bA$ and a horizontal equivalence $b\colon FA\xrightarrow{\simeq} B$ in $\bB$. We set $GB\coloneqq A$ and $\epsilon_B\coloneqq b\colon FGB\xrightarrow{\simeq} B$, and also fix horizontal equivalence data $(\epsilon_B,\epsilon'_B,\mu_B,\nu_B)$. 

{\bf $G$ and $\epsilon$ on horizontal morphisms.} Now let $b\colon B\to C$ be a horizontal morphism in~$\bB$. By (db2), there is a horizontal morphism $a\colon GB\to GC$ in $\bA$ and a vertically invertible square $\overline{\epsilon}_b$ as depicted inside the right-hand side of the pasting below. We set $Gb\coloneqq a\colon GB\to GC$ and $\epsilon_b$ to be the square given by the following pasting.
\begin{tz}
\node (A) at (0,0) {$FGB$};
\node (B) at (1.5,0) {$B$};
\node (C) at (3,0) {$C$};
\node (A') at (0,-1.5) {$FGB$};
\node (B') at (1.5,-1.5) {$FGC$};
\node (C') at (3,-1.5) {$C$};
\draw[d] (A) to (A');
\draw[d] (C) to (C');
\draw[->] (A) to node[above,scale=0.8] {$\epsilon_B$} (B);
\draw[->] (B) to node[above,scale=0.8] {$b$} (C);
\draw[->] (A') to node[below,scale=0.8,yshift=-1pt] {$FGb$} (B');
\draw[->] (B') to node[below,scale=0.8] {$\epsilon_C$} (C');

\node at (0,-.75) {$\bullet$};
\node at (3,-.75) {$\bullet$};

\node[scale=0.8] at (1.3,-.75) {$\epsilon_b$};
\node[scale=0.8] at (1.6,-.75) {$\vcong$};

\node at (4,-.75) {$=$};

\node (A) at (5,.75) {$FGB$};
\node (B) at (6.5,.75) {$B$};
\node (C) at (8,.75) {$C$};
\node (E) at (11,.75) {$C$};

\node (A') at (5,-.75) {$FGB$};
\node (B') at (6.5,-.75) {$B$};
\node (C') at (8,-.75) {$C$};
\node (D') at (9.5,-.75) {$FGC$};
\node (E') at (11,-.75) {$C$};

\draw[d] (A) to (A');
\draw[d] (C) to (C');
\draw[d] (E) to (E');
\draw[->] (A) to node[above,scale=0.8] {$\epsilon_B$} (B);
\draw[->] (B) to node[above,scale=0.8] {$b$} (C);
\draw[d] (C) to (E);

\node at (5,0) {$\bullet$};
\node at (8,0) {$\bullet$};
\node at (11,0) {$\bullet$};

\node[scale=0.8] at (6.5,0) {$e_{b \epsilon_B}$};
\node[scale=0.8] at (9.3,0) {$\nu_C^{-1}$};
\node[scale=0.8] at (9.7,0) {$\vcong$};

\node (A'') at (5,-2.25) {$FGB$};
\node (D'') at (9.5,-2.25) {$FGC$};
\node (E'') at (11,-2.25) {$C$};
\draw[d] (A') to (A'');
\draw[d] (D') to (D'');
\draw[d] (E') to (E'');
\draw[->] (A') to node[above,scale=0.8] {$\epsilon_B$} (B');
\draw[->] (B') to node[above,scale=0.8] {$b$} (C');
\draw[->] (C') to node[above,scale=0.8] {$\epsilon'_C$} (D');
\draw[->] (D') to node[above,scale=0.8] {$\epsilon_C$} (E');
\draw[->] (A'') to node[below,scale=0.8,yshift=-1pt] {$Fa=FGb$} (D'');
\draw[->] (D'') to node[below,scale=0.8] {$\epsilon_C$} (E'');

\node at (5,-1.5) {$\bullet$};
\node at (9.5,-1.5) {$\bullet$};
\node at (11,-1.5) {$\bullet$};

\node[scale=0.8] at (7,-1.5) {$\overline{\epsilon}_b$};
\node[scale=0.8] at (7.3,-1.5) {$\vcong$};
\node[scale=0.8] at (10.25,-1.5) {$e_{\epsilon_C}$};
\end{tz}
If $b=\id_B$, we can choose $G\id_B\coloneqq\id_{GB}$ and $\overline{\epsilon}_{\id_B}\coloneqq\mu_B^{-1}$. Then $\epsilon_{\id_B}=e_{\epsilon_B}$ by the triangle identities for $(\mu_B,\nu_B)$. 

{\bf Horizontal coherence.} Given horizontal morphisms $b\colon B\to C$ and $d\colon C\to D$ in $\bB$, we define the vertically invertible comparison square between $Gd\circ Gb$ and $G(db)$ as follows. Let us denote by $\Theta_{b,d}$ the following pasting.
\begin{tz}
\node (A) at (0,0) {$FGB$};
\node (B) at (4.5,0) {$FGC$};
\node (C) at (9,0) {$FGD$};
\draw[->] (A) to node[above,scale=0.8] {$FGb$} (B);
\draw[->] (B) to node[above,scale=0.8] {$FGd$} (C);

\node (A') at (0,-1.5) {$FGB$};
\node (X) at (1.5,-1.5) {$B$};
\node (Y) at (3,-1.5) {$C$};
\node (B') at (4.5,-1.5) {$FGC$};
\node (T) at (6,-1.5) {$C$};
\node (U) at (7.5,-1.5) {$D$};
\node (C') at (9,-1.5) {$FGD$};
\draw[->] (A') to node[above,scale=0.8] {$\epsilon_B$} (X);
\draw[->] (X) to node[above,scale=0.8] {$b$} (Y);
\draw[->] (Y) to node[above,scale=0.8] {$\epsilon'_C$} (B');
\draw[->] (B') to node[above,scale=0.8] {$\epsilon_C$} (T);
\draw[->] (T) to node[above,scale=0.8] {$d$} (U);
\draw[->] (U) to node[above,scale=0.8] {$\epsilon'_D$} (C');
\draw[d] (A) to (A');
\draw[d] (B) to (B');
\draw[d] (C) to (C');

\node at (0,-.75) {$\bullet$};
\node at (4.5,-.75) {$\bullet$};
\node at (9,-.75) {$\bullet$};

\node[scale=0.8] at (2,-.75) {$\overline{\epsilon}^{-1}_b$};
\node[scale=0.8] at (2.4,-.75) {$\vcong$};
\node[scale=0.8] at (6.5,-.75) {$\overline{\epsilon}^{-1}_d$};
\node[scale=0.8] at (6.9,-.75) {$\vcong$};

\node (A) at (0,-3) {$FGB$};
\node (X') at (1.5,-3) {$B$};
\node (Y') at (3,-3) {$C$};
\node (T') at (6,-3) {$C$};
\node (U') at (7.5,-3) {$D$};
\node (C) at (9,-3) {$FGD$};
\draw[->] (A) to node[below,scale=0.8] {$\epsilon_B$} (X');
\draw[->] (X') to node[below,scale=0.8] {$b$} (Y');
\draw[d] (Y') to (T');
\draw[->] (T') to node[below,scale=0.8] {$d$} (U');
\draw[->] (U') to node[below,scale=0.8] {$\epsilon'_D$} (C);
\draw[d] (A') to (A);
\draw[d] (C') to (C);
\draw[d] (Y) to (Y');
\draw[d] (T) to (T');

\node at (0,-2.25) {$\bullet$};
\node at (3,-2.25) {$\bullet$};
\node at (6,-2.25) {$\bullet$};
\node at (9,-2.25) {$\bullet$};

\node[scale=0.8] at (1.5,-2.25) {$e_{b\epsilon_B}$};
\node[scale=0.8] at (7.5,-2.25) {$e_{\epsilon'_D d}$};
\node[scale=0.8] at (4.3,-2.25) {$\nu_C$};
\node[scale=0.8] at (4.6,-2.25) {$\vcong$};

\node (A') at (0,-4.5) {$FGB$};
\node (C') at (9,-4.5) {$FGD$};
\draw[d] (A) to (A');
\draw[d] (C) to (C');
\draw[->] (A') to node[below,scale=0.8] {$FG(db)$} (C');

\node at (0,-3.75) {$\bullet$};
\node at (9,-3.75) {$\bullet$};

\node[scale=0.8] at (4.3,-3.75) {$\overline{\epsilon}_{db}$};
\node[scale=0.8] at (4.7,-3.75) {$\vcong$};
\end{tz}
Then, by (db4), there is a unique vertically invertible square $\Phi_{b,d}$ as in \cref{defn:pseudodouble} (i) such that $F\Phi_{b,d}=\Theta_{b,d}$. In particular, one can check that with this definition of $\Phi_{b,d}$, the squares $\epsilon_b$, $\epsilon_d$, and $\epsilon_{db}$ satisfy the required pasting equality pictured below.
\begin{tz}
\node (A) at (0,0) {$B$};
\node (B) at (1.5,0) {$FGB$};
\node (C) at (3,0) {$FGC$};
\node (D) at (4.5,0) {$FGD$};
\draw[->] (A) to node[above,scale=0.8] {$\epsilon_B$} (B);
\draw[->] (B) to node[above,scale=0.8,yshift=2pt] {$FGb$} (C);
\draw[->] (C) to node[above,scale=0.8,yshift=2pt] {$FGd$} (D);

\node (A') at (0,-1.5) {$B$};
\node (B') at (1.5,-1.5) {$FGB$};
\node (D') at (4.5,-1.5) {$FGD$};
\draw[->] (A') to node[below,scale=0.8] {$\epsilon_B$} (B');
\draw[->] (B') to node[below,scale=0.8] {$FG(db)$} (D');
\draw[d] (A) to (A');
\draw[d] (B) to (B');
\draw[d] (D) to (D');

\node at (0,-.75) {$\bullet$};
\node at (1.5,-.75) {$\bullet$};
\node at (4.5,-.75) {$\bullet$};

\node[scale=0.8] at (2.7,-.75) {$F\Phi_{b,d}$};
\node[scale=0.8] at (3.3,-.75) {$\vcong$};
\node[scale=0.8] at (.75,-.75) {$e_{\epsilon_B}$};

\node (A) at (0,-3) {$B$};
\node (B) at (1.5,-3) {$C$};
\node (C) at (3,-3) {$D$};
\node (D) at (4.5,-3) {$FGD$};
\draw[->] (A) to node[below,scale=0.8] {$b$} (B);
\draw[->] (B) to node[below,scale=0.8] {$d$} (C);
\draw[->] (C) to node[below,scale=0.8] {$\epsilon_D$} (D);
\draw[d] (A') to (A);
\draw[d] (D') to (D);

\node at (0,-2.25) {$\bullet$};
\node at (4.5,-2.25) {$\bullet$};

\node[scale=0.8] at (2,-2.25) {$\epsilon_{db}$};
\node[scale=0.8] at (2.4,-2.25) {$\vcong$};

\node at (5.5, -1.5) {$=$};

\node (A) at (6.5,0) {$B$};
\node (B) at (8,0) {$FGB$};
\node (C) at (9.5,0) {$FGC$};
\node (D) at (11,0) {$FGD$};
\draw[->] (A) to node[above,scale=0.8] {$\epsilon_B$} (B);
\draw[->] (B) to node[above,scale=0.8,yshift=2pt] {$FGb$} (C);
\draw[->] (C) to node[above,scale=0.8,yshift=2pt] {$FGd$} (D);

\node (A') at (6.5,-1.5) {$B$};
\node (B') at (8,-1.5) {$C$};
\node (C') at (9.5,-1.5) {$FGC$};
\node (D') at (11,-1.5) {$FGD$};
\draw[->] (A') to node[below,scale=0.8] {$b$} (B');
\draw[->] (B') to node[below,scale=0.8] {$\epsilon_C$} (C');
\draw[->] (C') to node[below,scale=0.8,yshift=-1pt] {$FGd$} (D');
\draw[d] (A) to (A');
\draw[d] (C) to (C');
\draw[d] (D) to (D');

\node at (6.5,-.75) {$\bullet$};
\node at (9.5,-.75) {$\bullet$};
\node at (11,-.75) {$\bullet$};

\node[scale=0.8] at (7.8,-.75) {$\epsilon_b$};
\node[scale=0.8] at (8.1,-.75) {$\vcong$};
\node[scale=0.8] at (10.25,-.75) {$e_{FGd}$};

\node (A) at (6.5,-3) {$B$};
\node (B) at (8,-3) {$C$};
\node (C) at (9.5,-3) {$D$};
\node (D) at (11,-3) {$FGD$};
\draw[->] (A) to node[below,scale=0.8] {$b$} (B);
\draw[->] (B) to node[below,scale=0.8] {$d$} (C);
\draw[->] (C) to node[below,scale=0.8] {$\epsilon_D$} (D);
\draw[d] (A') to (A);
\draw[d] (B') to (B);
\draw[d] (D') to (D);

\node at (6.5,-2.25) {$\bullet$};
\node at (8,-2.25) {$\bullet$};
\node at (11,-2.25) {$\bullet$};

\node[scale=0.8] at (9.3,-2.25) {$\epsilon_{d}$};
\node[scale=0.8] at (9.6,-2.25) {$\vcong$};
\node[scale=0.8] at (7.25,-2.25) {$e_b$};

\end{tz} 

{\bf $G$ and $\epsilon$ on vertical morphisms.} Now let $v\colon B\arrowdot B'$ be a vertical morphism in~$\bB$. By (db3), there is a vertical morphism $u'\colon A\arrowdot A'$ and a weakly horizontally invertible square $\sq{\gamma_v}{b}{d}{v}{Fu'}$, where $b\colon B\xrightarrow{\simeq} FA$ and $d\colon B'\xrightarrow{\simeq} FA'$ are horizontal equivalences. If we consider the composites of horizontal equivalences $b\epsilon_B\colon FGB\xrightarrow{\simeq} FA$ and $d\epsilon_{B'}\colon FGB'\xrightarrow{\simeq} FA'$, there is horizontal morphisms $a\colon GB\to A$ and $c\colon GB'\to A'$ in $\bA$ and vertically invertible squares $\sq{\gamma_b}{b\epsilon_B}{Fa}{e_{FGB}}{e_{FA}}$ and $\sq{\gamma_d}{d\epsilon_{B'}}{Fc}{e_{FGB'}}{e_{FA'}}$. Since lifts of horizontal equivalences by a double biequivalence are horizontal equivalences, we have that $a\colon GB\xrightarrow{\simeq} A$ and $c\colon GB'\xrightarrow{\simeq} A'$ are horizontal equivalences in $\bA$; thus, since $\bA$ is weakly horizontally invariant, there is a vertical morphism $u\colon GB\arrowdot GB'$ and a weakly horizontally invertible square
\begin{tz}
\node (A) at (0,0) {$GB$};
\node (B) at (1.5,0) {$A$};
\node (A') at (0,-1.5) {$GB'$};
\node (B') at (1.5,-1.5) {$A'$};
\node at ($(B'.east)-(0,4pt)$) {.};
\draw[->] (A) to node[above,scale=0.8] {$a$} node[below,scale=0.8] {$\simeq$} (B);
\draw[->] (A') to node[below,scale=0.8] {$c$} node[above,scale=0.8] {$\simeq$} (B');
\draw[->] (A) to node[left,scale=0.8] {$u\;$} (A');
\draw[->] (B) to node[right,scale=0.8] {$\;u'$} (B');
\node[scale=0.8] at (.6,-.75) {$\alpha_v$};
\node[scale=0.8] at (1,-.75) {$\simeq$};

\node at (0,-.75) {$\bullet$};
\node at (1.5,-.75) {$\bullet$};
\end{tz}
We set $Gv\coloneqq u\colon GB\arrowdot GB'$. To define the weakly horizontally invertible square $\epsilon_v$, let us first fix a weak inverse $\gamma'_v$ of $\gamma_v$ with respect to some horizontal equivalences $(b,b',\lambda,\kappa)$ and $(d,d',\lambda',\kappa')$. We set $\epsilon_v$ to be the square given by the following pasting.
\begin{tz}
\node (A) at (0,0) {$FGB$};
\node (B) at (1.5,0) {$B$};
\node (A') at (0,-1.5) {$FGB'$};
\node (B') at (1.5,-1.5) {$B'$};
\draw[->] (A) to node[above,scale=0.8] {$\epsilon_B$} (B);
\draw[->] (A') to node[below,scale=0.8] {$\epsilon_{B'}$}(B');
\draw[->] (A) to node[left,scale=0.8] {$FGv\;$} (A');
\draw[->] (B) to node[right,scale=0.8] {$\;v$} (B');
\node[scale=0.8] at (.75,-.75) {$\epsilon_v$};

\node at (0,-.75) {$\bullet$};
\node at (1.5,-.75) {$\bullet$};

\node at (2.3,-.75) {$=$};

\node (A) at (3.5,3) {$FGB$};
\node (B) at (5,3) {$B$};
\node (D) at (8,3) {$B$};
\draw[->] (A) to node[above,scale=0.8] {$\epsilon_B$} (B);
\draw[d] (B) to (D);

\node (A') at (3.5,1.5) {$FGB$};
\node (B') at (5,1.5) {$B$};
\node (C') at (6.5,1.5) {$FA$};
\node (D') at (8,1.5) {$B$};
\draw[d] (A) to (A');
\draw[d] (B) to (B');
\draw[d] (D) to (D');
\draw[->] (A') to node[above,scale=0.8] {$\epsilon_B$} (B');
\draw[->] (B') to node[above,scale=0.8] {$b$} (C');
\draw[->] (C') to node[above,scale=0.8] {$b'$} (D');

\node at (3.5,2.25) {$\bullet$};
\node at (5,2.25) {$\bullet$};
\node at (8,2.25) {$\bullet$};

\node[scale=0.8] at (4.25,2.25) {$e_{\epsilon_B}$};
\node[scale=0.8] at (6.3,2.25) {$\lambda$};
\node[scale=0.8] at (6.6,2.25) {$\vcong$};

\node (A) at (3.5,0) {$FGB$};
\node (C) at (6.5,0) {$FA$};
\node (D) at (8,0) {$B$};
\draw[d] (A') to (A);
\draw[d] (C') to (C);
\draw[d] (D') to (D);
\draw[->] (A) to node[above,scale=0.8] {$Fa$} (C);
\draw[->] (C) to node[above,scale=0.8] {$b'$} (D);

\node at (3.5,.75) {$\bullet$};
\node at (6.5,.75) {$\bullet$};
\node at (8,.75) {$\bullet$};

\node[scale=0.8] at (7.25,.75) {$e_{b'}$};
\node[scale=0.8] at (4.8,.75) {$\gamma_b$};
\node[scale=0.8] at (5.1,.75) {$\vcong$};

\node (A') at (3.5,-1.5) {$FGB'$};
\node (C') at (6.5,-1.5) {$FA'$};
\node (D') at (8,-1.5) {$B'$};
\draw[->] (A) to node[left,scale=0.8] {$FGv\;$} (A');
\draw[->] (C) to node[left,scale=0.8] {$Fu'\;$} (C');
\draw[->] (D) to node[right,scale=0.8] {$\;v$} (D');
\draw[->] (A') to node[below,scale=0.8] {$Fc$} (C');
\draw[->] (C') to node[below,scale=0.8] {$d'$} (D');

\node at (3.5,-.75) {$\bullet$};
\node at (6.5,-.75) {$\bullet$};
\node at (8,-.75) {$\bullet$};

\node[scale=0.8] at (7.1,-.75) {$\gamma'_v$};
\node[scale=0.8] at (7.5,-.75) {$\simeq$};
\node[scale=0.8] at (4.7,-.75) {$F\alpha_v$};
\node[scale=0.8] at (5.2,-.75) {$\simeq$};

\node (A) at (3.5,-3) {$FGB'$};
\node (B) at (5,-3) {$B'$};
\node (C) at (6.5,-3) {$FA'$};
\node (D) at (8,-3) {$B'$};
\draw[d] (A') to (A);
\draw[d] (C') to (C);
\draw[d] (D') to (D);
\draw[->] (A) to node[below,scale=0.8] {$\epsilon_{B'}$} (B);
\draw[->] (B) to node[below,scale=0.8] {$d$} (C);
\draw[->] (C) to node[below,scale=0.8] {$d'$} (D);

\node at (3.5,-2.25) {$\bullet$};
\node at (6.5,-2.25) {$\bullet$};
\node at (8,-2.25) {$\bullet$};

\node[scale=0.8] at (7.25,-2.25) {$e_{d'}$};
\node[scale=0.8] at (4.8,-2.25) {$\gamma_d^{-1}$};
\node[scale=0.8] at (5.2,-2.25) {$\vcong$};

\node (A') at (3.5,-4.5) {$FGB'$};
\node (B') at (5,-4.5) {$B'$};
\node (D') at (8,-4.5) {$B'$};
\draw[->] (A') to node[below,scale=0.8] {$\epsilon_{B'}$} (B');
\draw[d] (B') to (D');
\draw[d] (A) to (A');
\draw[d] (B) to (B');
\draw[d] (D) to (D');

\node at (3.5,-3.75) {$\bullet$};
\node at (5,-3.75) {$\bullet$};
\node at (8,-3.75) {$\bullet$};

\node[scale=0.8] at (4.25,-3.75) {$e_{\epsilon_{B'}}$};
\node[scale=0.8] at (6.3,-3.75) {${\lambda'}^{-1}$};
\node[scale=0.8] at (6.7,-3.75) {$\vcong$};

\end{tz}
Note that all the squares in the pasting are weakly horizontally invertible by \cite[Lemma A.2.1]{Lyne}, and thus so is $\epsilon_v$. We write $\epsilon'_v$ for its unique weak inverse with respect to the horizontal adjoint equivalences $(\epsilon_B,\epsilon'_B,\mu_B,\nu_B)$ and $(\epsilon_{B'},\epsilon'_{B'},\mu_{B'},\nu_{B'})$, as given by \cref{uniqueweakinverse}.

If $v=e_B$, we can choose $Ge_B\coloneqq e_{GB}$ and $\gamma_{e_B}\coloneqq e_{\epsilon_B}$. Then $\alpha_{e_B}$ can be chosen to be the identity square at the object $GB$ and we get $\epsilon_{e_B}=e_{\epsilon_B}$.

{\bf Vertical coherence.} Given vertical morphisms $v\colon B\arrowdot B'$ and $v'\colon B'\arrowdot B''$ in $\bB$, we define the horizontally invertible comparison square between $Gv'\bullet Gv$ and $G(v' v)$ as follows. Let us denote by $\Omega_{v,v'}$ the following pasting.
\begin{tz}
\node (A) at (0,0) {$FGB$};
\node (C) at (3,0) {$FGB$};
\draw[d] (A) to (C);

\node (A') at (0,-1.5) {$FGB$};
\node (B') at (1.5,-1.5) {$B$};
\node (C') at (3,-1.5) {$FGB$};
\draw[d] (A) to (A');
\draw[d] (C) to (C');
\draw[->] (A') to node[above,scale=0.8] {$\epsilon_B$} (B');
\draw[->] (B') to node[above,scale=0.8] {$\epsilon'_B$} (C');

\node at (0,-.75) {$\bullet$};
\node at (3,-.75) {$\bullet$};

\node[scale=0.8] at (1.3,-.75) {$\mu_B$};
\node[scale=0.8] at (1.7,-.75) {$\vcong$};

\node (A) at (0,-3) {$FGB'$};
\node (B) at (1.5,-3) {$B'$};
\draw[->] (A') to node[left,scale=0.8] {$FGv\;$} (A);
\draw[->] (B') to node[right,scale=0.8] {$\;v$} (B);
\draw[->] (A) to node[above,scale=0.8] {$\epsilon_{B'}$} (B);

\node at (0,-2.25) {$\bullet$};
\node at (1.5,-2.25) {$\bullet$};

\node[scale=0.8] at (.75,-2.25) {$\epsilon_v$};

\node (A') at (0,-4.5) {$FGB''$};
\node (B') at (1.5,-4.5) {$B''$};
\node (C) at (3,-4.5) {$FGB''$};
\draw[->] (A) to node[left,scale=0.8] {$FGv'\;$} (A');
\draw[->] (B) to node[right,scale=0.8] {$\;v'$} (B');
\draw[->] (C') to node[right,scale=0.8] {$\;FG(v'v)$} (C);
\draw[->] (A') to node[below,scale=0.8] {$\epsilon_{B''}$} (B');
\draw[->] (B') to node[below,scale=0.8] {$\epsilon'_{B''}$} (C);

\node at (0,-3.75) {$\bullet$};
\node at (1.5,-3.75) {$\bullet$};
\node at (3,-3) {$\bullet$};

\node[scale=0.8] at (.75,-3.75) {$\epsilon_{v'}$};
\node[scale=0.8] at (2.25,-3) {$\epsilon'_{v'v}$};

\node (A) at (0,-6) {$FGB''$};
\node (C') at (3,-6) {$FGB''$};
\draw[d] (A) to (C');
\draw[d] (A') to (A);
\draw[d] (C) to (C');

\node at (0,-5.25) {$\bullet$};
\node at (3,-5.25) {$\bullet$};

\node[scale=0.8] at (1.2,-5.25) {$\mu_{B''}^{-1}$};
\node[scale=0.8] at (1.7,-5.25) {$\vcong$};

\end{tz}
Note that this square is horizontally invertible, since it is weakly horizontally invertible and its horizontal boundaries are identities. By (db4), there is a unique horizontally invertible square $\Psi_{v,v'}$ as depicted below left such that $F\Psi_{v,v'}=\Omega_{v,v'}$. In particular, one can check that, with this definition of $\Psi_{v,v'}$, the squares $\epsilon_v$, $\epsilon_{v'}$ and $\epsilon_{v'v}$ satisfy the pasting equality below right. 
\begin{tz}
\node (A') at (-5.5,-1.5) {$GB$};
\node (B) at (-4,-1.5) {$GB$};
\draw[d] (A') to (B);

\node (A) at (-5.5,-3) {$GB'$};
\draw[->] (A') to node[left,scale=0.8] {$Gv\;$} (A);

\node at (-5.5,-2.25) {$\bullet$};

\node (A') at (-5.5,-4.5) {$GB''$};
\node (B') at (-4,-4.5) {$GB''$};
\draw[->] (A) to node[left,scale=0.8] {$Gv'\;$} (A');
\draw[->] (B) to node[right,scale=0.8] {$\;G(v'v)$} (B');
\draw[d] (A') to (B');

\node at (-5.5,-3.75) {$\bullet$};
\node at (-4,-3) {$\bullet$};

\node[scale=0.8] at (-4.75,-3.2) {$\Psi_{v,v'}$};
\node[scale=0.8] at (-4.75,-2.8) {$\cong$};

\node (A') at (-.25,-1.5) {$FGB$};
\node (B) at (1.5,-1.5) {$FGB$};
\draw[d] (A') to (B);

\node (A) at (-.25,-3) {$FGB'$};
\draw[->] (A') to node[left,scale=0.8] {$FGv\;$} (A);

\node at (-.25,-2.25) {$\bullet$};

\node (A') at (-.25,-4.5) {$FGB''$};
\node (B') at (1.5,-4.5) {$FGB''$};
\draw[->] (A) to node[left,scale=0.8] {$FGv'\;$} (A');
\draw[->] (B) to node[over,scale=0.8,yshift=.7cm] {$FG(v'v)$} (B');
\draw[d] (A') to (B');

\node at (-.25,-3.75) {$\bullet$};
\node at (1.5,-3) {$\bullet$};

\node[scale=0.8] at (.75,-3.2) {$F\Psi_{v,v'}$};
\node[scale=0.8] at (.75,-2.8) {$\cong$};

\node (C) at (3.25,-1.5) {$B$};
\node (C') at (3.25,-4.5) {$B''$};
\draw[->] (B) to node[above,scale=0.8] {$\epsilon_{B}$} (C);
\draw[->] (B') to node[below,scale=0.8] {$\epsilon_{B''}$} (C');
\draw[->] (C) to node[right,scale=0.8] {$\;v'v$} (C');
\node at (3.25,-3) {$\bullet$};

\node[scale=0.8] at (2.4,-3) {$\epsilon_{v'v}$};

\node at (4.25,-3) {$=$};

\node (A') at (5.25,-1.5) {$FGB$};
\node (B') at (7,-1.5) {$B$};
\draw[->] (A') to node[above,scale=0.8] {$\epsilon_B$} (B');

\node (A) at (5.25,-3) {$FGB'$};
\node (B) at (7,-3) {$B'$};
\draw[->] (A') to node[left,scale=0.8] {$FGv\;$} (A);
\draw[->] (B') to node[right,scale=0.8] {$\;v$} (B);
\draw[->] (A) to node[above,scale=0.8] {$\epsilon_{B'}$} (B);

\node at (5.25,-2.25) {$\bullet$};
\node at (7,-2.25) {$\bullet$};

\node[scale=0.8] at (6.15,-2.25) {$\epsilon_v$};

\node (A') at (5.25,-4.5) {$FGB''$};
\node (B') at (7,-4.5) {$B''$};
\draw[->] (A) to node[left,scale=0.8] {$FGv'\;$} (A');
\draw[->] (B) to node[right,scale=0.8] {$\;v'$} (B');
\draw[->] (A') to node[below,scale=0.8] {$\epsilon_{B''}$} (B');

\node at (5.25,-3.75) {$\bullet$};
\node at (7,-3.75) {$\bullet$};

\node[scale=0.8] at (6.15,-3.75) {$\epsilon_{v'}$};
\end{tz}

{\bf $G$ on squares.} Let $\sq{\beta}{b}{d}{v}{v'}$ be a square in $\bB$. Let us denote by $\delta$ the following pasting. 
\begin{tz}
\node (A) at (0,0) {$FGB$};
\node (D) at (4.5,0) {$FGC$};
\draw[->] (A) to node[above,scale=0.8] {$FGb$} (D);

\node (A') at (0,-1.5) {$FGB$};
\node (B') at (1.5,-1.5) {$B$};
\node (C') at (3,-1.5) {$C$};
\node (D') at (4.5,-1.5) {$FGC$};
\draw[d] (A) to (A');
\draw[d] (D) to (D');
\draw[->] (A') to node[above,scale=0.8] {$\epsilon_B$} (B');
\draw[->] (B') to node[above,scale=0.8] {$b$} (C');
\draw[->] (C') to node[above,scale=0.8] {$\epsilon'_C$} (D');

\node at (0,-.75) {$\bullet$};
\node at (4.5,-.75) {$\bullet$};

\node[scale=0.8] at (2,-.75) {$\overline{\epsilon}^{-1}_b$};
\node[scale=0.8] at (2.4,-.75) {$\vcong$};

\node (A) at (0,-3) {$FGB'$};
\node (B) at (1.5,-3) {$B'$};
\node (C) at (3,-3) {$C'$};
\node (D) at (4.5,-3) {$FGC'$};
\draw[->] (A') to node[left,scale=0.8] {$FGv\;$} (A);
\draw[->] (B') to node[left,scale=0.8] {$v\;$} (B);
\draw[->] (C') to node[right,scale=0.8] {$v'\;$} (C);
\draw[->] (D') to node[right,scale=0.8] {$FGv'\;$} (D);
\draw[->] (A) to node[below,scale=0.8] {$\epsilon_{B'}$} (B);
\draw[->] (B) to node[below,scale=0.8] {$d$} (C);
\draw[->] (C) to node[below,scale=0.8] {$\epsilon'_{C'}$} (D);

\node at (0,-2.25) {$\bullet$};
\node at (1.5,-2.25) {$\bullet$};
\node at (3,-2.25) {$\bullet$};
\node at (4.5,-2.25) {$\bullet$};

\node[scale=0.8] at (.75,-2.25) {$\epsilon_v$};
\node[scale=0.8] at (2.25,-2.25) {$\beta$};
\node[scale=0.8] at (3.75,-2.25) {$\epsilon'_{v'}$};

\node (A') at (0,-4.5) {$FGB'$};
\node (D') at (4.5,-4.5) {$FGC'$};
\draw[->] (A') to node[below,scale=0.8] {$FGd$} (D');
\draw[d] (A') to (A);
\draw[d] (D') to (D);

\node at (0,-3.75) {$\bullet$};
\node at (4.5,-3.75) {$\bullet$};

\node[scale=0.8] at (2,-3.75) {$\overline{\epsilon}_d$};
\node[scale=0.8] at (2.3,-3.75) {$\vcong$};
\end{tz}
Then, by (db4), there is a unique square $\sq{\alpha}{Gb}{Gd}{Gv}{Gv'}$ such that $F\alpha=\delta$. We set $G\beta\coloneqq\sq{\alpha}{Gb}{Gd}{Gv}{Gv'}$. 

Let $b\colon B\to C$ be a horizontal morphism in $\bB$, and $\sq{\beta=e_b}{b}{b}{e_B}{e_C}$. Then we have that $\delta=e_{FGb}$, since $\epsilon_{e_B}=e_{\epsilon_B}$ and $\epsilon'_{e_C}=e_{\epsilon_C}$, and the unique square $\sq{\alpha}{Gb}{Gb}{e_{GB}}{e_{GC}}$ such that $F\alpha=e_{FGb}$ is given by $e_{Gb}$. Therefore, $Ge_b=e_{Gb}$. 

Now let $v\colon B\arrowdot B'$ be a vertical morphism in $\bB$, and $\sq{\beta=\id_v}{\id_B}{\id_{B'}}{v}{v}$. Then we have that $\delta=\id_{FGv}$, since $\overline{\epsilon}_{\id_B}^{-1}=\mu_B$ and $\overline{\epsilon}_{\id_{B'}}=\mu_{B'}^{-1}$ and $\epsilon'_v$ is the weak inverse of $\epsilon_B$ with respect to the horizontal adjoint equivalence data $(\epsilon_B,\epsilon'_B,\mu_B,\nu_B)$ and $(\epsilon_{B'},\epsilon'_{B'},\mu_{B'},\nu_{B'})$. The unique square $\sq{\alpha}{\id_{GB}}{\id_{GB'}}{Gv}{Gv}$ such that $F\alpha=\id_{FGv}$ is given by $\id_{Gv}$. Therefore, $G\id_v=\id_{Gv}$.

{\bf Naturality and adjointness of $\epsilon$ and $\epsilon'$.} The assignment of $G$ on squares is natural with the data of $\epsilon_B$, $\epsilon_b$ and $\epsilon_v$, and therefore the latter assemble into a horizontal pseudo natural equivalence $\epsilon\colon FG\Rightarrow \id_\bB$. Moreover, since $(\epsilon_B,\epsilon'_B,\mu_B,\nu_B)$ are horizontal adjoint equivalences, the data of $\epsilon'_B$, $\epsilon'_b$ and $\epsilon'_v$ also assemble into a horizontal pseudo natural equivalence $\epsilon'\colon \id_\bB\Rightarrow FG$, where $\epsilon'_b$ is defined in a similar manner as $\epsilon_b$ was. In particular, $\epsilon\colon FG\Rightarrow \id_\bB$ and $\epsilon'\colon \id_\bB\Rightarrow FG$ are adjoint equivalences, where the invertible modifications are given by $\mu\colon \id\cong \epsilon'\epsilon$ and $\nu\colon \epsilon\epsilon'\cong \id$. 

It remains to define the horizontal pseudo natural equivalence $\eta\colon \id_\bA\Rightarrow GF$. For this purpose, we use the horizontal pseudo natural equivalence $\epsilon'\colon \id_\bB\Rightarrow FG$.

{\bf $\eta$ on objects.} Let $A\in \bA$, and consider the horizontal equivalence $\epsilon'_{FA}\colon FA\xrightarrow{\simeq} FGFA$. By (db2), there is a horizontal morphism $a\colon A\to GFA$ and a vertically invertible square $\sq{\rho_A}{\epsilon'_{FA}}{Fa}{e_{FA}}{e_{FGFA}}$. We set $\eta_A\coloneqq a \colon A\to GFA$. Note that $\eta_A\colon A\xrightarrow{\simeq} GFA$ is a horizontal equivalence.

{\bf $\eta$ on horizontal morphisms.} Let $a\colon A\to C$ be a horizontal morphism in $\bA$. We denote by $\psi_a$ the pasting below left. By (db4) there is a unique vertically invertible square $\alpha$ as below right such that $F\alpha=\psi_a$; let $\eta_a \coloneqq \alpha$. 
\begin{tz}
\node (A) at (0,0) {$FA$};
\node (B) at (1.8,0) {$FGFA$};
\node (C) at (3.6,0) {$FGFC$};
\draw[->] (A) to node[above,scale=0.8] {$F\eta_A$} (B);
\draw[->] (B) to node[above,scale=0.8,yshift=2pt] {$FGFa$} (C);

\node (A') at (0,-1.5) {$FA$};
\node (B') at (1.8,-1.5) {$FGFA$};
\node (C') at (3.6,-1.5) {$FGFC$};
\draw[->] (A') to node[below,scale=0.8] {$\epsilon'_{FA}$} (B');
\draw[->] (B') to node[below,scale=0.8,yshift=-2pt] {$FGFa$} (C');
\draw[d] (A) to (A');
\draw[d] (B) to (B');
\draw[d] (C) to (C');

\node at (0,-.75) {$\bullet$};
\node at (1.8,-.75) {$\bullet$};
\node at (3.6,-.75) {$\bullet$};

\node[scale=0.8] at (.7,-.75) {$\rho_A^{-1}$};
\node[scale=0.8] at (1.1,-.75) {$\vcong$};
\node[scale=0.8] at (2.7,-.75) {$e_{FGFa}$};

\node (A) at (0,-3) {$FA$};
\node (B) at (1.8,-3) {$FC$};
\node (C) at (3.6,-3) {$FGFC$};
\draw[->] (A) to node[below,scale=0.8] {$Fa$} (B);
\draw[->] (B) to node[below,scale=0.8] {$\epsilon'_{FC}$} (C);
\draw[d] (A') to (A);
\draw[d] (C') to (C);

\node at (0,-2.25) {$\bullet$};
\node at (3.6,-2.25) {$\bullet$};

\node[scale=0.8] at (1.6,-2.25) {$\epsilon'_{Fa}$};
\node[scale=0.8] at (2,-2.25) {$\vcong$};

\node (A') at (0,-4.5) {$FA$};
\node (B') at (1.8,-4.5) {$FC$};
\node (C') at (3.6,-4.5) {$FGFC$};
\draw[->] (A') to node[below,scale=0.8] {$Fa$} (B');
\draw[->] (B') to node[below,scale=0.8] {$F\eta_C$} (C');
\draw[d] (A) to (A');
\draw[d] (B) to (B');
\draw[d] (C) to (C');

\node at (0,-3.75) {$\bullet$};
\node at (1.8,-3.75) {$\bullet$};
\node at (3.6,-3.75) {$\bullet$};

\node[scale=0.8] at (2.5,-3.75) {$\rho_C$};
\node[scale=0.8] at (2.9,-3.75) {$\vcong$};
\node[scale=0.8] at (.9,-3.75) {$e_{Fa}$};

\node (A) at (6,-1.5) {$A$};
\node (B) at (7.5,-1.5) {$GFA$};
\node (C) at (9,-1.5) {$GFC$};
\node (A') at (6,-3) {$A$};
\node (B') at (7.5,-3) {$C$};
\node (C') at (9,-3) {$GFC$};
\draw[d] (A) to (A');
\draw[d] (C) to (C');
\draw[->] (A) to node[above,scale=0.8] {$\eta_A$} (B);
\draw[->] (B) to node[above,scale=0.8,yshift=1pt] {$GFa$} (C);
\draw[->] (A') to node[below,scale=0.8] {$a$} (B');
\draw[->] (B') to node[below,scale=0.8] {$\eta_C$} (C');

\node at (6,-2.25) {$\bullet$};
\node at (9,-2.25) {$\bullet$};

\node[scale=0.8] at (7.3,-2.25) {$\alpha$};
\node[scale=0.8] at (7.6,-2.25) {$\vcong$};
\end{tz}

{\bf $\eta$ on vertical morphisms.} Let $u\colon A\arrowdot A'$ be a vertical morphism in $\bA$. We denote by $\psi_u$ the pasting below left.
\begin{tz}
\node (A) at (0,0) {$FA$};
\node (B) at (1.8,0) {$FGFA$};
\node (A') at (0,-1.5) {$FA$};
\node (B') at (1.8,-1.5) {$FGFA$};
\draw[->] (A) to node[above,scale=0.8] {$F\eta_A$} (B);
\draw[->] (A') to node[above,scale=0.8] {$\epsilon'_{FA}$} (B');
\draw[d] (A) to (A');
\draw[d] (B) to (B');

\node at (0,-.75) {$\bullet$};
\node at (1.8,-.75) {$\bullet$};

\node[scale=0.8] at (.7,-.65) {$\rho_A^{-1}$};
\node[scale=0.8] at (1.1,-.65) {$\vcong$};

\node (A) at (0,-3) {$FA'$};
\node (B) at (1.8,-3) {$FGFA'$};
\draw[->] (A) to node[below,scale=0.8] {$\epsilon'_{FA'}$} (B);
\draw[->] (A') to node[left,scale=0.8] {$Fu\;$} (A);
\draw[->] (B') to node[right,scale=0.8] {$\;FGFu$} (B);

\node at (0,-2.25) {$\bullet$};
\node at (1.8,-2.25) {$\bullet$};

\node[scale=0.8] at (.7,-2.25) {$\epsilon'_{Fu}$};
\node[scale=0.8] at (1.1,-2.25) {$\simeq$};

\node (A') at (0,-4.5) {$FA'$};
\node (B') at (1.8,-4.5) {$FGFA'$};
\draw[->] (A') to node[below,scale=0.8] {$F\eta_{A'}$} (B');
\draw[d] (A) to (A');
\draw[d] (B) to (B');

\node at (0,-3.75) {$\bullet$};
\node at (1.8,-3.75) {$\bullet$};

\node[scale=0.8] at (.7,-3.85) {$\rho_{A'}$};
\node[scale=0.8] at (1.1,-3.85) {$\vcong$};

\node (A) at (4.5,-1.5) {$A$};
\node (B) at (6,-1.5) {$GFA$};
\node (A') at (4.5,-3) {$A'$};
\node (B') at (6,-3) {$GFA'$};
\draw[->] (A) to node[above,scale=0.8] {$\eta_A$} (B);
\draw[->] (A') to node[below,scale=0.8] {$\eta_{A'}$}(B');
\draw[->] (A) to node[left,scale=0.8] {$u\;$} (A');
\draw[->] (B) to node[right,scale=0.8] {$\;GFu$} (B');
\node[scale=0.8] at (5.25,-2.25) {$\gamma$}; 
\node at (4.5,-2.25) {$\bullet$};
\node at (6,-2.25) {$\bullet$};
\end{tz}
Note that all the squares in $\psi_u$ are weakly horizontally invertible by \cite[Lemma A.2.1]{Lyne}, and thus so is $\psi_u$. By (db4) there is a unique weakly horizontally invertible square $\sq{\gamma}{\eta_A}{\eta_{A'}}{u}{GFu}$ as above right such that $F\gamma=\psi_u$; let $\eta_u\coloneqq \gamma$.

{\bf Naturality of $\eta$.} Since $\epsilon'\colon \id_\bB\Rightarrow FG$ is a horizontal pseudo natural transformation, then $\eta_A$, $\eta_a$, and $\eta_u$ assemble into a horizontal pseudo natural transformation $\eta\colon \id_\bA\Rightarrow GF$. Note that $\eta$ is a horizontal pseudo natural equivalence, because $\eta_A$ are horizontal equivalences and $\eta_u$ are weakly horizontally invertible squares.  Moreover, $\rho\colon \epsilon'_F\cong F\eta$ gives the data of an invertible modification.
\end{proof}

\bibliographystyle{plain}
\bibliography{References}

\end{document}